\documentclass{amsart}

\usepackage{graphicx}

\usepackage{amsfonts, amsmath, amssymb,amsthm,comment}  
\usepackage{times,enumerate}
\usepackage{nccmath}
\newenvironment{mpmatrix}{\begin{medsize}\begin{pmatrix}}%
    {\end{pmatrix}\end{medsize}}%
\usepackage{dsfont,bbm}
\usepackage{rotating}
\usepackage{lscape}
\usepackage{url}
\usepackage{multicol}
\usepackage{thmtools, thm-restate}
\usepackage{blkarray}
\usepackage{multicol}
\usepackage[margin=2.5cm]{geometry}
 \usepackage[foot]{amsaddr}
 \usepackage{hyperref}

\usepackage{graphicx}
\usepackage[usenames, dvipsnames]{color}
\DeclareGraphicsExtensions{.pdf,.png,.jpg}

\hypersetup{
    colorlinks,
    linkcolor={blue},
    citecolor={blue},
    urlcolor={blue}
}

\DeclareMathOperator{\Span}{span}

\DeclareMathOperator{\Rank}{rank}
\DeclareMathOperator{\cyc}{cyc}

\DeclareMathOperator{\Card}{card}
\DeclareMathOperator{\sign}{sign}

\newcommand{\Sym}{\mathbb{S}}

\newcommand{\RR}{\mathbb R}

\newcommand{\NN}{\mathbb N}

\newcommand{\CC}{\mathbb C}

\newcommand{\dd}{{\rm d}}

\newcommand{\cM}{\mathcal M}

\newcommand{\Tr}{\mathrm{Tr}}

\newcommand{\bX}{\mathbb{X}}
\newcommand{\bY}{\mathbb{Y}}
\newcommand{\benu}{\begin{enumerate}}
\newcommand{\eenu}{\end{enumerate}}
\newcommand{\bop}{\begin{opomba}}
\newcommand{\eop}{\end{opomba}}

\newcommand{\tr}{\mathrm{tr}}
\newcommand{\supp}{\mathrm{supp}}

\newtheorem{theorem}{Theorem}[section]
\newtheorem{corollary}[theorem]{Corollary}
\newtheorem{lemma}[theorem]{Lemma}
\newtheorem{proposition}[theorem]{Proposition}

\theoremstyle{definition}

\newtheorem{example}[theorem]{Example}

\newcommand{\mc}{\mathcal}
\newcommand{\mbb}{\mathbb}
\newcommand{\mbf}{\mathbf}
\newcommand{\mds}{\mathds}

\theoremstyle{remark}
\newtheorem{remark}[theorem]{Remark}

\numberwithin{equation}{section}

\begin{document}

\title{The singular bivariate quartic tracial moment problem}

\author{Abhishek Bhardwaj}
\address{Mathematical Sciences Institute,
The Australian National University,
Union Lane,
Canberra  ACT  2601}
\email{Abhishek.Bhardwaj@anu.edu.au}

\author{Alja\v z Zalar$^{1}$}
\address{Institute of Mathematics, Physics, and Mechanics, Jadranska 19, 1000
Ljubljana, Slovenia}
\email{aljaz.zalar@imfm.si}
\thanks{$^{1}$ Supported by the Slovenian Research Agency and in part by the Slovene Human Resources Development
and Scholarship Fund.}


\date{\today}



\begin{abstract}
The (classical) truncated moment problem, extensively studied by Curto and Fialkow, asks to characterize when a finite sequence of real numbers indexes by words in commuting variables can be represented with moments of a positive Borel measure $\mu$ on $\RR^n$. In \cite{BK12} Burgdorf and Klep introduced its tracial analog, the truncated tracial moment problem, 
which replaces commuting variables with non-commuting ones and moments of $\mu$ with tracial moments of matrices.
In the bivariate quartic case, where indices run over words in two variables of degree at most four,
every sequence with a positive definite $7\times 7$ moment matrix $\mc M_2$ can be represented with tracial moments
\cite{BK10,BK12}.
In this article the case of singular $\mc M_2$ is studied. 
For $\mc M_2$ of rank at most 5 the problem is solved completely; namely, concrete measures are obtained whenever they exist and the uniqueness 
question of the minimal measures is answered.
For $\mc M_2$ of rank 6 the problem splits into four cases, in two of which it is equivalent to the feasibility problem of certain linear matrix inequalities. Finally, the question of a flat extension of the moment matrix $\mc M_2$ is addressed. While this is the most powerful tool for solving the classical case,
it is shown here by examples that, while sufficient, flat extensions are mostly not a necessary condition for the existence of a measure in the tracial case.
\end{abstract}

\maketitle

\section{Introduction}

\subsection{Context} \label{context}

The Moment problem (MP) is a classical question in analysis and concerns the existence of a positive Borel measure $\mu$ supported on a subset $K$ of $\mbb{R}^{n}$, representing a given sequence of real or complex numbers indexed by monomials as the integration of the corresponding monomials w.r.t.\ $\mu$; nice expositions on the MP are \cite{Akh65, KN77}. The solution to the MP on $\mbb{R}^{n}$ is given by Haviland's theorem \cite{Hav35}, which establishes the duality with positive polynomials and relates the MP to real algebraic geometry (RAG). One of the cornerstones of RAG is the celebrated Schm\"udgen theorem \cite{Sch91}, which solves the problem on compact basic closed semialgebraic sets and is the beginning of extensive research of the MP in RAG; we refer the reader to \cite{Put93,PV99,DP01,PS01,KM02,Lau05,Lau09,Mar08,Las09} and the references therein for further details. Another important aspect of the MP is uniqueness of the representing measures. For compact sets 
the measure is unique (see e.g., \cite{Mar08}), while for noncompact sets, the question of uniqueness is highly nontrivial (see \cite{PS06,PS08}).

There are various generalizations of the MP. Functional analysis studies various versions of matrix and operator MPs; see 
\cite{Kre49,Kov83,AV03,Vas03,BW11,CZ12,KW13} and references therein. The quantum MP from quantum physics is considered in \cite{DLTW08}. The rational MP, which extends Schm\"udgen theorem from the polynomial algebra to its localizations, is solved in \cite{CMN11}, while \cite{GKM16} investigates the MP for the polynomial algebra in infinitely many variables. The MP on semialgebraic sets of generalized function is considered in \cite{IKR14}.
The beginning of free RAG is the solution of the full non-commutative (nc) MP by McCullough [McC01] and Helton [Hel02]. 
The nc MP has been further investigated in \cite{HM04} and \cite{HKM12}.
In \cite{HM04} the authors solve the full nc MP for nc matrix polynomials on a bounded nc semialgebraic set, while in 
\cite{HKM12} the truncated nc MP for nc matrix polynomials on a convex nc semialgebraic set is solved.
Finally, the most recent free MP is a tracial MP \cite{BK10,BK12,BCKP13,BKP16} which is also the contents of this article.

The (classical) truncated moment problem (TMP) refers to the MP where only finitely many numbers in the sequence are given and one wants to know if they can be generated from a measure. By the Bayer-Teichmann version \cite{BT06} of the 
Tchakaloff theorem
\cite{Tch57},
it is sufficient to study only  finitely atomic measures.
Furthermore, the TMP is more general than the full MP by a result of Stochel \cite{Sto01}. Curto and Fialkow shined new light on the TMP in their series 
of papers \cite{CF96,CF98-1,CF98-2,CF02,CF04,CF05,CF08,Fia14}. One of their crowning achievements is the discovery that if a moment matrix admits a rank preserving extension (to a moment matrix), then the corresponding sequence admits a representing measure. Using this result, they completely solved the bivariate quartic TMP. 
Recently, the representing measures for the bivariate nonsingular quartic MP were constructed by the use of computer algebra \cite{CS16}. 

The study of the truncated tracial moment problem (TTMP) was initiated by Burgdorf and Klep in \cite{BK10,BK12}, followed by \cite{BCKP13}.  A very nice reference including the results of this paragraph that also deals with polynomial optimization problems in matrix unknowns is \cite{BKP16}. The motivation to study the TTMP comes from trace-positive polynomials, which are very interesting due to important applications, e.g.,  
Connes' embedding problem \cite{Con76} from functional analysis and Bessis-Moussa-Villani conjecture from statistical quantum mechanics
have reformulations in terms of trace-positive polynomials \cite{KS08-1,KS08-2,Bur11}. Determining if a polynomial is trace-positive, being the dual problem to the TTMP, is the connection of the TTMP with free RAG.
Using this duality the bivariate quartic tracial MP with a positive definite moment matrix is solved in \cite{BK10} (for alternative proof see \cite{Caf13}) by showing that bivariate quartic trace-positive polynomials are always sums of hermitian squares and commutators. This fact does not generalize to higher powers or more variables (see examples in \cite{KS08-1,KS08-2,Qua15}).
In \cite{BK12,BCKP13} the authors obtain the tracial analogs of the results on the classical moment problem of Curto and Fialkow, Stochel, Bayer and Teichmann, Fialkow and Nie \cite{FN10}, providing powerful means to tackle the special cases of the TTMP in a way analogous way to the classical one.

The origins of this paper lie in the first authors MSc thesis \cite{Bha16} which contains alternative proofs for some preliminary results and some computational techniques to construct flat extensions of $\mc{M}_2$ to $\mc{M}_3$ in special cases.

In this section we state the main concepts and results of this paper. In Subsection \ref{BQTMP} we introduce some essential definitions before stating the main results in Subsection \ref{main-results}. Finally,
Subsection \ref{read-guide} is a guide to the organization of the rest of the paper.

\subsection{Bivariate quartic tracial moment problem} \label{BQTMP}

	In this subsection we state the main problem and introduce basic
	definitions used throughout this article.

\subsubsection{Noncommutative bivariate polynomials}

	We denote by $\left\langle X,Y\right\rangle$ the \textbf{free monoid} generated by 
	the noncommuting letters $X,Y$ and call its elements \textbf{words} in $X, Y$. 
	Consider the free algebra $\mbb{R}\!\left\langle  X,Y\right\rangle$ of polynomials in $X,Y$
	with coefficients in $\mbb{R}$. Its elements are called \textbf{noncommutative (nc) polynomials}. Endow $\mbb{R}\!\left\langle X,Y\right\rangle$ with the involution $p \mapsto p^{*}$ fixing $\mbb{R} \cup \{ X,Y\}$ pointwise. The length of the longest word in a polynomial $f\in \mbb{R}\!\left\langle  X,Y \right\rangle$  is the \textbf{degree} 
of $f$ and is denoted by $\deg(f)$ or $|f|$.  We write $\mbb{R}\!\left\langle  X,Y\right\rangle_{\leq k}$ for
	all polynomials of degree  at most $k$.
	For a word $w\in \left\langle X,Y\right\rangle$, $w^{*}$ is its reverse, and $v\in\left\langle X,Y\right\rangle$ is \textbf{cyclically equivalent} to $w$, which we denote by $\displaystyle v\overset{\cyc}{\sim} w$, if and only if $v$ is a cyclic permutation of $w$. 
	
\subsubsection{Bivariate quartic real tracial moment problem}

Given a sequence of real numbers $\beta\equiv \beta^{(4)}=(\beta_w)_{|w|\leq 4}$, indexed by words $w$ of length at most 4 such that
	$$\beta_{v}=\beta_{w}\quad\text{whenever } v\overset{\cyc}{\sim} w
	\quad \text{and}\quad \beta_w=\beta_{w^\ast}\quad\text{for all } |w|\leq 4,$$
i.e.,
	\begin{eqnarray*}
	\beta&=&\big(\beta_{1}, \beta_X, \beta_Y, \beta_{X^2},
		\beta_{XY}=\beta_{YX}, 
		\beta_{Y^{2}}, \beta_{X^3}, 
		\beta_{X^2Y}=\beta_{XYX}=\beta_{YX^2},\\ 
		&&\beta_{XY^2}=\beta_{YXY}=\beta_{Y^2X}, 
		\beta_{Y^3}, \beta_{X^4}, 
		\beta_{X^3Y}=\beta_{X^2YX}=\beta_{XYX^2}=\beta_{YX^3},\\ 
		&&\beta_{X^2Y^2}=\beta_{XY^2X}=\beta_{Y^2X^2}=\beta_{YX^2Y}, 
		\beta_{XYXY}=\beta_{YXYX}, \\
		&&\beta_{XY^3}=\beta_{YXY^2}=\beta_{Y^2XY}=\beta_{Y^3X}, 
		\beta_{Y^4}\big),
	\end{eqnarray*}
 the \textbf{bivariate quartic real tracial moment problem (BQTMP)} for $\beta$ asks to find conditions for the existence of
 $N\in\mathbb{N}$, $t_i\in\mathbb{N}$, $\lambda_i\in \RR_{> 0}$ with 
$\sum_{i=1}^N \lambda_i=1$ and pairs of real symmetric matrices $(A_{i},B_{i})\in (\mathbb{SR}^{t_i\times t_i})^2$,
such that 
\begin{equation}\label{truncated tracial moment sequence special-biv}
	\beta_{w}=\sum_{i=1}^N \lambda_i \text{Tr}(w(A_{i},B_{i})),
\end{equation}
where $w$ runs over the indices of the sequence $\beta$ and $\text{Tr}$ denotes the \textbf{normalized trace}, i.e., 
	$$\text{Tr}(A)=\frac{1}{t}\tr(A)\quad \text{for every } A\in \RR^{t\times t}.$$
If such data exist, we say that $\beta$ admits a representing measure. The vectors $(A_i,B_i)$ are  \textbf{atoms of size $t_i$} and the numbers $\lambda_i$ are \textbf{densities}.
We say that $\mu$ is a representing measure of \textbf{type}
$(m_1,m_2,\ldots,m_r)$ if it consists of exactly $m_i\in \NN\cup\{0\}$ atoms of size $i$ and $m_r\neq 0$.
A representing measure of type $(m_1^{(1)},m_2^{(1)},\ldots,m_{r_1}^{(1)})$ is
\textbf{minimal}, if there does not exist another representing measure of type $(m_1^{(2)}$, $m_2^{(2)}$,$\ldots$, $m_{r_2}^{(2)})$ such that
	$$r_2<r_1\quad \text{or} \quad (r_2=r_1, m^{(2)}_{r_2}<m_{r_1}^{(1)})\quad \text{or}\quad 
		(r_2=r_1, m^{(2)}_{r_2}=m_{r_2}^{(1)},  m^{(2)}_{r_2-1}<m_{r_1-1}^{(1)})$$
		$$ \text{or}\quad \ldots
		\quad \text{or}\quad 
		(r_2=r_1, m^{(2)}_{r_2}=m_{r_2}^{(1)},  \ldots, m^{(2)}_{2}=m_{2}^{(1)}, m^{(2)}_{1}<m_{1}^{(1)}).
	$$
We say that $\beta$ admits a \textbf{noncommutative (nc) measure, if it admits a minimal meausre of type $(m_1,m_2,\ldots,m_r)$ with $r>1$.}
If $\beta_1=1$, then we say $\beta$ is \textbf{normalized}. 
We may always assume that $\beta$ is normalized (otherwise we replace $\Tr$ with $\frac{1}{\beta_1}\Tr$). 
If $\beta_{X^2Y^2}=\beta_{XYXY}$, we call $\beta$ a \textbf{commutative (cm) sequence} and the MP reduces to the classical one solved by Curto and Fialkow.
If $\beta_{X^2Y^2}\neq \beta_{XYXY}$, then $\beta$ is a \textbf{noncommutative (nc) sequence}.

\begin{remark}
\begin{enumerate}
	\item Note that replacing a vector $(A_i,B_i)$ with any 
		vector $$(U_i A_i U_i^t , U_i B_i U_i^{t})\in (\mathbb{SR}^{t_i\times t_i})^2$$			where $U_i\in \RR^{t_i\times t_i}$ is an orthogonal matrix, preserves	
		(\ref{truncated tracial moment sequence special-biv}).
	\item By the tracial version \cite[Theorem 3.8]{BCKP13} of Bayer-Teichmann theorem 
		\cite{BT06}, the problem (\ref{truncated tracial moment sequence special-biv}) is 
		equivalent to the more general problem of finding a probability measure $\mu$
		on $(\mathbb{SR}^{t\times t})^2$ such that 
		$\beta_{w}=\int_{(\mathbb{SR}^{t\times t})^2} \Tr(w(A,B))\; \dd\mu(A,B)$.
\end{enumerate}
\end{remark}

We associate to the sequence $\beta$ the \textbf{truncated moment matrix
$\mc{M}_2=\mc{M}_2(\beta)$ of order $2$}
with rows and columns indexed by words in $\RR\!\left\langle X,Y\right\rangle_{\leq 2}$
in the degree-lexicographic order.
The entry in row $U$ and column $V$ is $\beta_{U^\ast V}$, i.e., 
\begin{equation}\label{moment-matrix-k-n-2-2}
\mc{M}_2=
\begin{blockarray}{cccccccc}
	 & \mds{1}&\bX&\bY&\bX^{2}&\bX\bY&\bY\bX&\bY^{2}\\
	 \begin{block}{c(ccccccc)}
	\mds{1}&\beta_1 & \beta_{X} & \beta_Y & \beta_{X^2} & \beta_{XY} & \beta_{XY} & \beta_{Y^2}\\
	\bX&\beta_X & \beta_{X^2} & \beta_{XY} & \beta_{X^3} & \beta_{X^2Y} & \beta_{X^2Y} & \beta_{XY^2}\\
	\bY&\beta_Y & \beta_{XY} & \beta_{Y^2} & \beta_{X^2Y} & \beta_{XY^2} & \beta_{XY^2} & \beta_{Y^3}\\
	\bX^{2}&\beta_{X^2} & \beta_{X^3} & \beta_{X^2Y} & \beta_{X^4} & \beta_{X^3Y} & \beta_{X^3Y} & \beta_{X^2Y^2}\\
	\bX\bY&\beta_{XY} & \beta_{X^2Y} & \beta_{XY^2} & \beta_{X^3Y} & \beta_{X^2Y^2} & \beta_{XYXY} & \beta_{XY^3}\\
	\bY\bX&\beta_{XY} & \beta_{X^2Y} & \beta_{XY^2} & \beta_{X^3Y} & \beta_{XYXY} & \beta_{X^2Y^2} & \beta_{XY^3}\\
	\bY^{2}&\beta_{Y^2} & \beta_{XY^2} & \beta_{Y^3} & \beta_{X^2Y^2} & \beta_{XY^3} & \beta_{XY^3} & \beta_{Y^4}\\
	\end{block}
\end{blockarray}.
\end{equation}

\noindent Observe that the matrix $\mc{M}_2$ is symmetric.
Let $S_1,S_2$ be subsets of $\{\mds{1},\bX,\bY,\bX^{2},\bX\bY,\bY\bX,\bY^{2}\}$. We will denote by ${\mc M_2}|_{S_1,S_2}$ the submatrix of $\mc M_2$ consisting of
the rows indexed by the elements of $S_1$ and the columns indexed by the elements of $S_2$. In case $S:=S_1=S_2$, we write ${\mc M_2}|_{S}:= {\mc M_2}|_{S,S}$ for short.
If $\beta$ admits a measure, then $\mc{M}_2$ is 
positive semidefinite (psd); see Proposition \ref{M2-psd}. If $\mc{M}_2$ represents a cm sequence, we call it a \textbf{cm moment matrix}. Otherwise $\mc{M}_2$ is a \textbf{nc moment matrix}. By \cite[Corollaries 3.19, 3.20]{BK12}, $\beta$ admits a measure if and only if there exists
a moment matrix $\mc{M}_{2+k}$ extending $\mc{M}_2$, which admits a rank preserving extension $\mc{M}_{2+k+1}$.
Furthermore, by \cite[Corollary 3.2]{BK12} in this case the atoms of size at most $\Rank(\mc M_{2+k})$ are sufficient.
If $\mc M_2$ is positive definite, then $\beta$ admits a measure since 
all trace-positive polynomials of degree 4 are cyclically equivalent to sums of hermitian squares \cite{BK10}.
This is the duality established by \cite[Theorem 4.4]{BK12}. Moreover, the measure consists of at most 15 atoms of size 2 \cite[Remark 3.9]{BCKP13}.
By dualizing a certain cone of trace-positive polynomials of degree 4 one obtains the so called completely positive semidefinite cone introduced by Laurent and Piovesan \cite{LP15} which was used to study certain quantum graph parameters. 

\subsection{Results.} \label{main-results}
In this paper we study the BQTMP for sequences with \emph{singular} moment matrices. Initially, we approached this problem using a nc analog of the main tool for studying commutative sequences, i.e., finding rank preserving extensions of the moment matrices involved. As is already well established by Curto and Fialkow, the existence of a measure usually implies the existense of a rank preserving extension of a moment matrix, and hence a minimal measure with rank$(\mc{M}_2)$ atoms; see Theorem \ref{com-case} below. However, in stark contrast to the commutative case, our research soon revealed that this does not apply for noncommutative sequences $\beta$. Characterizing moment matrices $\mc{M}_2$ which admit a flat extension is insufficient for solving the BQTMP. The most versatile tool in tackling the BQTMP is the application of appropriate affine linear transformations on the sequences $\beta$, see Subsection \ref{affine linear-trans} below. This splits the BQTMP into finitely many cases according to the column relations that exist in $\mc{M}_2$. For $\mc{M}_2$ of rank at most 5, we completely characterize the existence, minimality and uniqueness of a measure in terms of the parameters $\beta_{w}$. In two out the four cases of $\mc{M}_2$ being of rank at most 6, we prove that the BQTMP is equivalent to the feasibility problem of certain linear matrix inequalities with an additional rank-to-variety condition for one of them. In all but a single case - that of $\mc{M}_2$ being rank 6 and satisfying $\mds{Y}^{2}=\mds{1}$ - we show that atoms of size at most 2 suffice in the minimal measure of $\beta$. We now give a brief outline of the results we prove in this paper.\\

\noindent{\textbf{Outline of the results on BQTMP:} (We assume that
$\beta$ is a nc sequence.)}
\begin{enumerate}
	\item If $\Rank(\mc{M}_2)\leq 3$, then $\beta$ does not admit a nc measure.
		Namely, if $\beta$ admits a nc measure, then the columns $\mds 1, \bX, \bY, \bX\bY$ 
		in $\mc{M}_2$ must be linearly independent (see Corollary \ref{lin-ind-of-4-col} below).
	\item If $\Rank(\mc{M}_2)= 4$, then it suffices, after applying an appropriate affine linear
		transformation, to study the case when $\mc{M}_2$
		satisfies the column relations 
			$$\bX^2=\mds{1},\quad \bX\bY+\bY\bX=a\mds 1,\quad \bY^2=\mds{1},$$ 
		where $a\in (-2,2)$.
		By finding the representing atom of size 2 it turns out that such $\beta$ 
		always admits a nc measure.
		Moreover, the representing atom is unique (up to orthogonal equivalence); see Theorem 
		\ref{rank4-soln-aljaz}.
	\item If $\Rank(\mc{M}_2)= 5$, then it suffices, after applying an appropriate 
		affine linear transformation, to study four special 
		cases when $\mc{M}_2$ satisfies 
		the column relation
			\begin{equation} \label{rel1-r5} \bX\bY+\bY\bX=\mbf 0 \end{equation}
		and one of
			\begin{equation} \label{rel2-r5}
				\bX^2+\bY^2=\mds 1\quad \text{or}\quad \bY^2=\mds 1 \quad \text{or}\quad  
				\bY^2-\bX^2=\mds 1 \quad \text{or}\quad  \bY^2=\bX^2
			\end{equation}
		(see Proposition \ref{structure-of-rank5-2} (\ref{point-1-str-rank5})).
		Due to the first relation the atoms in the minimal measure are of special form 
		(see Proposition \ref{anticommute}).
		In particular, all the atoms of size bigger than 1 have trace 0. 
		Analyzing moment matrices $\mc{M}_2$ generated only by atoms of size bigger than 
		1 with trace 0, we show that  only one such atom of size 2 is needed. 
		Since every atom of size 2 generates a moment matrix of rank 4,
		$\beta$ admits a nc measure if and only if there
		is a nc moment matrix $\widetilde{\mc{M}}_2$ of rank 4 representing a nc sequence 
		$\widetilde\beta$ admitting a nc measure such that 
		$\mc{M}_2-\alpha \widetilde{\mc{M}}_2$, $\alpha>0$,
		is a cm moment matrix admitting a measure.
		However, there are infinitely many possible atoms of size 2 satisfying relations 
		(\ref{rel1-r5}) and (\ref{rel2-r5}). But there are at most 4 atoms of size 1 satisying 
		(\ref{rel1-r5}) and (\ref{rel2-r5}).
		Therefore it is easier to subtract candidates for a cm moment matrices 
		$\widehat{\mc{M}}_2$. Namely, we
		compute the smallest $\alpha>0$ such that 
		$\Rank\left(\mc{M}_2-\alpha\widehat{\mc{M}}_2\right)=4$. By the solution of the
		rank 4 case we can characterize exactly, in terms of the parameters $\beta_w$,
		when a  nc measure exists, type of a minimal measure and its uniqueness.
	\item If $\Rank(\mc{M}_2)=6$, then it suffices, after applying an appropriate
		affine linear 
		transformation, to understand the four special cases when $\mc{M}_2$
		satisfies one of the column relations
			\begin{equation*} \label{rel-r6}
				\bX^2+\bY^2=\mds 1\quad \text{or}\quad \bX\bY+\bY\bX=\mbf 0 \quad \text{or}\quad  
				\bY^2-\bX^2=\mds 1 \quad \text{or}\quad  \bY^2=\mds 1
			\end{equation*}
		(see Proposition \ref{structure-of-rank5-2} (\ref{point-2-str-rank5})).
		In the first three cases the atoms in the minimal measure for $\beta$ are of size at most 2.
		In the first two cases the nc measure always exists if 
		$\beta_X=\beta_Y=\beta_{X^3}=\beta_{X^2Y}=\beta_{Y^3}=0$.
		This is proved by computing
		the smallest $\alpha>0$ such that 
		$$\Rank\left(\mc{M}_2-\alpha\left(\mc{M}_{2}^{(1,0)}+
			\mc{M}_{2}^{(-1,0)}\right)\right)<6$$ 
		(resp.\ $\Rank\left(\mc{M}_2-\alpha\mc{M}_{2}^{(0,0)}\right)<6$),
		where $\mc{M}_{2}^{(x,y)}$ is the moment matrix generated 
		by the atom $(x,y)\in\RR^2$, 
		and using the results of $\Rank(\mc{M}_2)\leq 5$.
		Otherwise, still referring to the first two cases of rank 6, if one of the moments $\beta_X,\beta_Y,\beta_{X^3},\beta_{X^2Y}, \beta_{Y^3}$ is nonzero, then 
		the existence of the nc measure is equivalent to the feasibilty 
		of certain linear matrix inequalities.
\end{enumerate}

\subsection{Reader's guide.} \label{read-guide}

The paper is organized as follows. In Section \ref{prel} we prove some preliminary results
about tracial moment sequences (see Subsections \ref{general-ttmm}-\ref{affine linear-trans})
and present the solution of the classical singular bivariate quartic MP (see Theorem \ref{com-case}). In Section \ref{Rank4-sol} we solve BQTMP
with $\mc{M}_2$ of rank 4. In Section \ref{reduc} we reduce 
the study of BQTMP with $\mc{M}_2$ of rank 5 or 6 to four basic cases (see Proposition \ref{structure-of-rank5-2}).
In Section \ref{atoms-minim} we prove that in the basic cases of rank 5 and three basic cases of rank 6,
atoms are of a special form and all the atoms of size at least 2 in some minimal measure have trace 0 (see Proposition \ref{anticommute}).
In Section \ref{rank5-section} we solve all four basic cases of rank 5 (see 
Theorems \ref{M(2)-XY+YX=0-bc1}, \ref{M(2)-XY+YX=0},
\ref{M(2)-XY+YX=0-bc3}, \ref{M(2)-XY+YX=0-bc4}).
In Section \ref{rank6-section} we prove that in the first three basic cases of rank 6 atoms of size 2 suffice in the minimal measure for $\beta$.
Then we study the relation $\bY^2=\mds 1-\bX^2$ in Subsection \ref{r6-subs1} (see Theorem \ref{M(2)-bc2-r6-new1} and Corollary \ref{M(2)-bc2-r6-new1-cor}) and the relation $\bX\bY+\bY\bX=\mbf 0$ in Subsection \ref{r6-subs2} (see Theorem \ref{M(2)-XY+YX=0-bc4-r6} and Corollary \ref{r6-bc3}).
In Section \ref{flat-rank6} we analyze the existence of flat extensions with a moment structure $\mc{M}_3$ 
for moment matrices $\mc{M}_2$ of rank 6.\\

 \noindent\textbf{Acknowledgement.}  Part of this paper was written at The University of Auckland under the supervision of Igor Klep who was
 the MSc supervisor of the first author and the PhD co-supervisor of the second author. Both authors wish to thank him for introducing us to this topic, the many insightful and inspiring discussions and support throughout the research. We are also thankful to two anonymous referees for useful comments and suggestions for improvements of the paper.\\

\section{Preliminaries} \label{prel}
	
	This section is devoted to terminology, notation and some preliminary results. Since these results
	hold for sequence of any degree (not only of degree 4) we will work with a general case.

\subsection{Bivariate truncated tracial moment problem}
\label{general-ttmm}

BQTMP is a special case of a \textbf{bivariate truncated tracial moment problem}:
Given a sequence of real numbers $\beta\equiv \beta^{(2k)}=(\beta_w)_{|w|\leq 2k}$, indexed by words $w$ of length at most $2k$ such that
	\begin{equation}\label{tracial-condition}
		\beta_{v}=\beta_{w}\quad\text{whenever } v\overset{\cyc}{\sim} w
		\quad \text{and}\quad \beta_w=\beta_{w^\ast}\quad\text{for all } |w|\leq 2k,
	\end{equation}
when does there exist 
 $N\in\mathbb{N}$, $t_i\in\mathbb{N}$, $\lambda_i\in \RR_{> 0}$ with 
$\sum_{i=1}^N \lambda_i=1$ and vectors $(A_{i},B_{i})\in (\mathbb{SR}^{t_i\times t_i})^2$,
such that 
\begin{equation}\label{truncated tracial moment sequence special-2}
	\beta_{w}=\sum_{i=1}^N \lambda_i \text{Tr}(w(A_{i},B_{i})),
\end{equation}
where $w$ runs over the indices of the sequence $\beta$.

\subsection{Riesz functional and truncated moment matrix}
\label{Riesz-funct}

For a polynomial $p\in \mbb{R}\!\langle X,Y\rangle_{\leq 2k}$, let $\hat{p}=(a_{w})_{w}$ be its coefficient vector with respect to the  lexicographically-ordered basis 
	$$\left\{1, X, Y, X^2, XY, YX, Y^2,\ldots,X^{2k},\ldots,Y^{2k}\right\}$$
of  $\mbb{R}\langle X,Y\rangle _{\leq 2k}$.
Any sequence $\beta\equiv\beta^{(2k)}:\beta_{1},\dotsc,\beta_{X^{2k}},\dotsc,\beta_{Y^{2k}}$, which satisfies (\ref{tracial-condition})
defines the \textbf{Riesz functional} $L_{\beta^{(2k)}}:\mbb{R}\!\left\langle  X,Y\right\rangle_{\leq 2k}\to \RR$ by
$$
	L_{\beta^{(2k)}}(p):=\sum_{|w|\leq 2k} a_{w}\beta_{w},\quad \text{where }p=\sum_{|w|\leq 2k} a_{w}w.
$$
Notice that 
	$$\beta_w=L_{\beta^{(2k)}}(w)\quad \text{for every }|w|\leq 2k.$$
The \textbf{truncated moment matrix $\mc{M}_k(\beta)$ of order $k$} is defined by 
	$$\mc{M}_k=\mc{M}_k(\beta^{(2k)})=(\beta_{U^{\ast}V})_{|U|\leq k,|V|\leq k},$$
where the rows and columns are indexed by monomials in $\RR\!\langle X,Y\rangle_{\leq k}$ in lexicographic order. 
When $k=2$, $\mc{M}_2$ is of the form (\ref{moment-matrix-k-n-2-2}). 
$\mc{M}_k$ is the unique matrix such that for $p,q\in\mbb{R}\!\langle X,Y\rangle _{\leq k}$ we have that
$$
	\langle \mc{M}_k \hat{p}, \hat{q}\rangle = L_{\beta^{(2k)}}(pq^{*}),
$$
where $\langle \hat p, \hat q\rangle :=\hat p^t \hat q$. In particular, the row $w_1(X,Y)$ and column $w_2(X,Y)$ entry of $\mc{M}_k$ is equal to
	$$
		\Big\langle \mc{M}_k \widehat{w_2(X,Y)}, \widehat{w_1(X,Y)}\Big\rangle = L_{\beta^{(2k)}}(w_2 w_1^{*}).
	$$
If $\beta^{(2k)}$ admits a measure, i.e., (\ref{truncated tracial moment sequence special-2}) holds for every 
$\beta_w$, then for $p\in\RR\!\left\langle X,Y\right\rangle$ of degree at most $k$ we have that 
	$$\langle \mc{M}_k \hat{p}, \hat{p}\rangle = L_{\beta^{(2k)}}(pp^{*})=
	\sum_{i=1}^m \lambda_i \Tr\left(p(X_i,Y_i) \left(p(X_i,Y_i)\right)^{\ast}\right)\geq 0,$$
where $\lambda_i, X_i, Y_i$ are as in (\ref{truncated tracial moment sequence special-2}).
This proves the following proposition.

\begin{proposition} \label{M2-psd}
	If $\beta^{(2k)}$ admits a measure, then $\mc{M}_k$ is positive semidefinite.
\end{proposition}

\subsection{Support of a measure and RG relations} \label{supp-and-RG}
Let $A$ be a matrix with its rows and columns indexed by words in $\mbb{R}\langle X,Y\rangle$.
Writing $w(\mathbb{X,Y})$ we mean the column of $A$ indexed by the word $w$.
$[A]_{E}$ denotes the compression of $A$ to the rows and columns indexed by the elements of the set $E$. Similarly for vectors $\mbf{v}$, $\mbf{v}_{E}$ denotes the compression of $\mbf{v}$ to the entries indexed by words in $E$. $\mbf{0_{m\times n}}$ stand for the $m\times n$ matrix with zero entries.
Usually we will omit the subindex $m\times n$, when the size will be clear from the context.

Let $\mc{C}_{\mc{M}_k}$ denote the column space of $\mc{M}_k$, i.e., 
	$$\mc{C}_{\mc{M}_k}=\Span\{\mds{1}, \mathbb{X},\mathbb{Y},\dotsc,\mathbb{X}^{k}\dotsc, \mathbb{Y}^{k}\}.$$ 
For a polynomial $p\in \mbb{R}\!\langle X,Y\rangle _{\leq k}$ of the form $p=\sum a_{w}w(X,Y)$,
we define 
	$$p(\mathbb{X},\mathbb{Y})=\sum_w a_{w} w(\bX,\bY)$$ and notice that 
	$p(\bX,\bY)\in \mc{C}_{\mc{M}_k}$. We express linear dependencies among the columns of $\mc{M}_k$ as 
	$$p_1(\bX,\bY)=\mbf 0,\ldots, p_m(\bX,\bY)=\mbf 0,$$ for some $p_1,\ldots,p_m\in\mbb{R}\!\langle X,Y\rangle _{\leq k}$, $m\in \NN$. We define the \textbf{free zero set} $\mc{Z}(p)$ of $p\in\mbb{R}\!\langle X,Y\rangle$ by
	$$\mc{Z}(p):=\left\{ (A,B)\in(\mbb{SR}^{t\times t})^{2} : t\in \NN,\; p(A,B)=\mbf 0_{t\times t}\right\}.$$
Theorem \ref{support lemma} (\ref{point-1-support}) (resp.\ (\ref{point-3-support}))
is a real tracial analogue of \cite[Proposition 3.1]{CF96} (resp.\ \cite[Theorem 1.6]{CF98-2}).

\begin{theorem}\label{support lemma}
	Suppose $\beta^{(2k)}$ admits a representing measure consisting of finitely many atoms  
	$(X_i,Y_i)\in (\mathbb{SR}^{t_i\times t_i})^2$, $t_i\in \NN$, 
	with the corresponding densities $\lambda_i\in (0,1)$.
	Let $p\in\mbb{R}\!\left\langle X,Y \right\rangle_{\leq k}$ be a polynomial.	
	Then the following are true:
	\begin{enumerate}
		\item\label{point-1-support} 
			We have 
				$$\bigcup_{i}\; (X_i,Y_i)\subseteq \mc{Z}(p) \quad \Leftrightarrow \quad 
					p(\bX,\bY)=\mbf 0\; \text{ in }\mc{M}_k.$$
		\item\label{point-2-support}  Suppose the sequence $\beta^{(2k+2)}=(\beta_w)_{|w|\leq k+1}$ is the extension of $\beta$ generated by 
				$$\beta_w=\sum_{i} \lambda_i \Tr(w(X_i,Y_i)).$$ 
			Let $\mc{M}_{k+1}$ be the corresponding moment matrix.
			Then:
			$$p(\bX,\bY)=\mbf 0\;\text{ in }\mc{M}_k\quad \Rightarrow \quad
				p(\bX,\bY)=\mbf 0 \;\text{ in }\mc{M}_{k+1}.$$
		\item\label{point-3-support}  (Recursive generation) For $q \in\mbb{R}\!\left\langle X,Y \right\rangle_{\leq k}$ such that $pq \in\mbb{R}\!\left\langle X,Y \right\rangle_{\leq k}$, we have
		 		\begin{equation*} \label{pq-lemma}
					p(\bX,\bY)=\mbf{0}\;\text{ in }\mc{M}_k\quad\Rightarrow\quad
					(pq)(\bX,\bY)=(qp)(\bX,\bY)=\mbf{0}\;\text{ in }\mc{M}_k.
				\end{equation*}
	\end{enumerate}
\end{theorem}

\begin{proof}
	Write $p=\sum_{|w|\leq n} a_w w$  where $a_w\in \RR$. 
	We have that
		$$\langle \mc{M}_k\hat p,\hat p\rangle
			= L_\beta(pp^\ast)=\sum_{|w|,|v|\leq k} 
				a_w a_v \beta_{wv^\ast}= \sum_{i} \lambda_i \Tr(p(X_i,Y_i)p^{\ast}(X_i,Y_i)).
		$$
	Observe that 
		\begin{equation}\label{eq-1}
			\mc{M}_k\hat p= p(\bX,\bY) \; \text{in }\mc{M}_k.
		\end{equation}
	Since $\mc{M}_k$ is psd, 
		\begin{equation}\label{eq-2}
			\mc{M}_k\hat p=\mbf 0 \quad \Leftrightarrow \quad
			\langle \mc{M}_k\hat p,\hat p\rangle=0.
		\end{equation}
	Since $p(X_i,Y_i)p^{\ast}(X_i,Y_i)$ is psd for each $i$, we have that
		\begin{equation}\label{eq-3}
				\sum_{i} \lambda_i \Tr(p(X_i,Y_i)p^{\ast}(X_i,Y_i))=0  
			\quad \Leftrightarrow \quad
				p(X_i,Y_i)=\mbf 0_{t_i\times t_i}.
		\end{equation}
	By (\ref{eq-1})-(\ref{eq-3}), Theorem \ref{support lemma} (\ref{point-1-support}) is true. Theorem \ref{support lemma} (\ref{point-2-support}) follows easily.
	
	It remains to prove Theorem \ref{support lemma} (\ref{point-3-support}). If $\deg(p)=k$, then $q\in \RR$ and statement is clear.
	Else $\deg(p)<k$. It suffices to prove that
		\begin{equation}\label{suffices-for-pq-lemma}
			(Xp)(\bX,\bY)=(pX)(\bX,\bY)=(Yp)(\bX,\bY)=(pY)(\bX,\bY)=\mbf{0}\; \text{ in }\mc{M}_k.
		\end{equation}
	First we will prove that 
		$(Xp)(\bX,\bY)=\mbf{0}$ in $\mc{M}_k$.
	By Theorem \ref{support lemma} (\ref{point-1-support}), we know that
		\begin{eqnarray}
			\bigcup_{i} \; (X_i,Y_i) \subseteq \mc{Z}(p)\quad  &\Leftrightarrow& 
				\quad p(\bX,\bY)=
					\mbf 0 \text{ in }\mc{M}_k,\label{eq-pq-1}\\
			\bigcup_{i} \; (X_i,Y_i) \subseteq \mc{Z}(Xp)\quad  &\Leftrightarrow& 
				\quad (Xp)(\bX,\bY)=
					\mbf 0 \text{ in }\mc{M}_k.\label{eq-pq-2}
		\end{eqnarray}
	Since by assumption $p(\bX,\bY)=\mbf 0$ in $\mc{M}_k$, it follows by 
	$\mc{Z}(p)\subseteq \mc{Z}(Xp)$, (\ref{eq-pq-1}) and (\ref{eq-pq-2}) that
		$$(Xp)(\bX,\bY)=\mbf{0} \;\text{ in }\mc{M}_k.$$ 
	By noticing that
	also the other three equalities of (\ref{suffices-for-pq-lemma}) are proved analogously, 
	Theorem \ref{support lemma} (\ref{point-3-support}) is true.
\end{proof}

Column relations forced upon $\mc{M}_k$ with an application of 
	Theorem \ref{support lemma}  (\ref{point-3-support})  
will be important in solving BQTMP and we will refer to them as the \textbf{RG relations}.
	If $\mc M_k$ satisfies RG relations, we say $\mc{M}_k$ is \textbf{recursively generated}.
The first consequence of the RG relations is the following important observation about a nc
moment matrix $\mc{M}_k$.

\begin{corollary} \label{lin-ind-of-4-col}
	Suppose $k\geq 2$ and $\beta^{(2k)}$ be a sequence such that 
	$\beta_{X^2Y^2}\neq \beta_{XYXY}$. 
	Then the columns $\mds 1, \bX, \bY, \bX\bY$ of $\mc{M}_k$ are linearly independent.
\end{corollary}

\begin{proof}
		Let us say that $\mbf 0=a\mds{1}+b\bX+c\bY+d\bX\bY$ for some $a,b,c,d\in \RR.$
	If $d\neq 0$, then we have that $\beta_{X^2Y^2}=\beta_{XYXY}$, which is a contradiction with the
	assumption.
	Hence $d=0$. From $\mbf 0=a\mds{1}+b\bX+c\bY$ it follows by the RG relations that
		$$\mbf{0}=a\bX+b\bX^2+c\bX\bY=a\bY+b\bX\bY+c\bY^2.$$ 
	If $b\neq 0$ or $c\neq 0$, it follows that
	 $\beta_{X^2Y^2}=\beta_{XYXY}$. Hence $b=c=0$. Finally $\mbf{0}=a\mds{1}$ implies that $a=0$.
	This proves the corollary.
\end{proof}

\begin{corollary}
	Suppose $k\geq 2$ and $\beta^{(2k)}$ be a sequence such that 
	$\beta_{X^2Y^2}\neq \beta_{XYXY}$.
	If $\mc{M}_k$ is of rank at most 3, then $\beta$ does not admit a measure.
\end{corollary}

\subsection{Flat extensions} \label{flat-ext-prel}

For a matrix $A\in \Sym\RR^{s\times s}$, an \textbf{extension} 
$\widetilde{A}\in \Sym\RR^{(s+u)\times (s+u)}$ of the form
	$$\widetilde A=\begin{pmatrix} A & B \\ B^t & C \end{pmatrix}$$
for some $B\in \RR^{s\times u}$ and $C\in \RR^{u\times u}$, is called $\textbf{flat}$ if 
$\Rank(A)=\Rank(\widetilde A)$. This is equivalent to saying that there is a matrix
$W\in \RR^{s\times u}$ such that $B=AW$ and $C=W^t A W$. The connection between flat extensions and BTTMP is \cite[Theorem 3.19]{BK12}.

\begin{theorem} \label{flat-meas}
		Let $\beta\equiv\beta^{(2k)}$ be a sequence satisfying (\ref{tracial-condition}). 
	If $\mc M_k(\beta)$ is psd and is a flat extension of $\mc M_{k-1}(\beta)$, then $\beta$ admits a
	representing measure.
\end{theorem}

\subsection{Affine linear transformations} \label{affine linear-trans}

An important result for converting a given moment problem into a simpler, equivalent moment problem is the application of affine linear transformations to a sequence $\beta$. For $a,b,c,d,e,f\in \RR$ with $bf-ce \neq 0$, let us define 
	$$\phi(x,y)=(\phi_1(x,y),\phi_2(x,y)):=(a+bx+cy,d+ex+fy),\; (x,y)\in \RR^{2}.$$
Let $\widetilde \beta^{(2k)}$ be the sequence obtained by the rule
	$$\widetilde \beta_{w}=L_{\beta^{(2k)}}(w\circ\phi)\quad \text{for every }|w|\leq k.$$
Notice that 
	$$L_{\widetilde\beta^{(2k)}}(p) =L_{\beta^{(2k)}}(p\circ \phi)\quad \text{for every }p \in \mbb{R}\!\langle X,Y\rangle_{\leq k}.$$

The following is the tracial analogue of \cite[Proposition 1.9]{CF04}, which will allow us to make affine linear changes of variables.

\begin{proposition}\label{linear transform invariance-nc}
	Suppose $\beta^{(2k)}$ and $\widetilde{\beta}^{(2k)}$ are as above and  $\mc{M}_k$ and 
	$\widetilde{\mc{M}}_k$ the corresponding moment matrices.
	Let $J_\phi: \RR\!\left\langle X,Y\right\rangle_{\leq 2k}\to \RR\!\left\langle X,Y\right\rangle_{\leq 2k}$ be 	the linear map given by 
		$$J_\phi \widehat p:=\widehat{p\circ \phi}.$$
	Then the following hold:
\begin{enumerate}
	\item $\widetilde{\mc{M}}_k=(J_\phi)^t \mc{M}_k J_\phi.$
	\item $J_\phi$ is invertible.
	\item  $\widetilde{\mc{M}}_k\succeq 0 \Leftrightarrow  \mc{M}_k\succeq 0.$
	\item $\Rank \widetilde{\mc{M}}_k=\Rank \mc{M}_k.$
	\item\label{invariance-point5} The formula $\mu =\tilde \mu \circ\phi$ establishes a one-to-one correspondence between the sets of 			representing measures of $\beta$ and $\tilde \beta$, 
		and $\phi$ maps $\supp(\mu)$ bijectively onto $\supp(\tilde \mu)$.
	\item \label{flat-ext-M(n)} $\mc{M}_k$ admits a flat extension if and only if $\widetilde{\mc{M}}_k$ admits a flat extension.
	\item For $p\in \RR\!\left\langle X,Y\right\rangle_{\leq k}$, we have $p(\widetilde X, \widetilde Y)=(J_\phi)^t((p\circ \phi)(X,Y))$.
\end{enumerate}
\end{proposition}

\begin{proof}	
	The proof is the same to the proof of the corresponding statement in the commutative case
	\cite[Proposition 1.9]{CF05}.
\end{proof}

\subsection{Classical bivariate quartic real moment problem}
 
The classical bivariate quartic MP has been solved by Curto and Fialkow in  
a series of papers, e.g., \cite{CF96,CF98-1,CF98-2,CF02,CF04,CF05,CF08,Fia14}.
The main technique used was the analysis of the existence of a flat extension of a moment matrix $\mc{M}_2$. The solution of the singular bivariate quartic real MP is Theorem \ref{com-case} below. Given a polynomial
	$p\in \RR[x,y]_{\leq 2}$ we write 
	$\mathcal Z_{cm}(p)=\left\{ (x,y)\in \RR^2\colon p(x,y)=0 \right\}$
	for the variety generated by $p$.
	\begin{theorem} \label{com-case}
			Suppose $\beta\equiv\beta^{(4)}$ is a commutative sequence with the associated moment matrix $\mc{M}_2$.
		Let $$\mathcal V:=\displaystyle\bigcap_{\substack{g\in \RR[x,y]_{\leq 2}\\ g(\bX,\bY)=\mbf 0}} \mathcal 
		Z_{cm}(g)$$ be the variety associated to $\mc{M}_2$ and  $p\in \RR[x,y]$ a polynomial with
		$\deg(p)=2$. Then
		$\beta$ has a representing measure supported in 
		$\mathcal Z_{cm}(p)$
			if and only if
		$M(2)$ is positive semidefinite, recursively generated, satisfies
			$\displaystyle\Rank(M(2))\leq \Card \mathcal V$
		and has a column dependency relation $p(\mds{X},\mds{Y})=0$.
		
		Moreover, assume that $\mc{M}_2$ is positive semidefinite, recursively generated and satisfies the column dependency relation $p(\mds{X},\mds{Y})$. The following statements are true:
		\begin{enumerate}
			\item\label{pt1} If $\Rank(\mc{M}_2)\leq 3$,
				then $\mc{M}_2$ always admits a flat extension to a moment matrix $\mc M_3$ and hence $\beta$ admits
				a $\Rank(\mc{M}_2)$-atomic minimal measure.
			\item\label{pt2} If $\Rank(\mc{M}_2)=4$, then $\beta$
				does not necessarily 
				come from a nc measure.
			\item\label{pt3} If $\Rank(\mc{M}_2)=5$, then $\beta$ always admits a nc measure, but $\mc{M}_2$ does not necessarily admit a 
				flat extension to a moment matrix $\mc M_3$. There exists an affine 	
				linear transformation such that
				$\mathcal{V}$ is one of $x^2+y^2=1$, $y=x^2$, $xy=1$, $x^2=1$ or $xy=0$. In the first 
				four cases $\mc{M}_2$ always admits a flat extension to a moment matrix $\mc M_3$ and hence $\beta$ admits
				a 5-atomic measure. However, in the last case there always exists a measure with 6 
				representing atoms, but not necessarily 5. 
			\item\label{pt4} If $\Rank(\mc{M}_2)=6$, then $\mc{M}_2$ always admits a flat extension to a moment matrix $\mc M_3$  and hence $\beta$ admits
				a 6-atomic measure.
		\end{enumerate}
	\end{theorem}
	
\begin{proof}
	For the proof of the first part see \cite{Fia14} and references therein.
	Let us now prove points (\ref{pt1})-(\ref{pt4}) of the second part.
	Defining $z:=x+iy$ and $\bar z:=x-iy$, $\beta$ has a representing measure if and only if
	the complex sequence $\gamma_{ij}^{(4)}:=L_\beta(\bar z^i z^j)$ has a representing measure
	by \cite[Proposition 1.12]{CF02}. We write $\mc{M}^\CC_2$ for the associated complex moment
	matrix.
	If $\mds 1, \bX, \bY$ are linearly dependent, then 
	$\overline{\mathbb{Z}}\in \Span\{\mds 1, \mathbb{Z}\}$ in $\mc{M}^\CC_2$
	and hence $\mc{M}_2$
	admits a flat extension to a moment matrix $\mc M_3$ by \cite[Theorem 2.1]{CF98-1}.
	In particular, this is true if $\Rank(\mc{M}_2)\leq 2$. If $\Rank(\mc{M}_2)=3$ and 
	$\mds 1, \bX, \bY$ are linearly independent, then
	$\mds 1,\mathbb Z, \overline{\mathbb {Z}}$ are linearly independent,
	$\mathbb Z\overline{\mathbb {Z}}\in \Span\{\mds 1, \mathbb{Z}, \overline{\mathbb Z}\}$
	and $\mc{M}_2$ admits a flat extension by \cite[Theorem 1.1]{CF02}.
	This proves (\ref{pt1}). Parts (\ref{pt2}) and (\ref{pt3}) follow by the results in 
	\cite{CF02,Fia14}. Part (\ref{pt4}) follows by \cite[Theorem 2.1]{CS16}.
\end{proof}

\section{Solution of the BQTMP for $\mc{M}_2$ of rank 4} \label{Rank4-sol}

In this section we solve the BQTMP for $\mc{M}_2$ of rank 4. In Theorem \ref{rank4-soln-aljaz} we characterize exactly when the corresponding nc sequence $\beta$ admits a nc measure.
Moreover, we prove that the minimal measure is unique (up to orthogonal equivalence) of type $(0,1)$ and find the concrete atom. In particular, $\beta$ admits a nc measure if and only if $\mc{M}_2$
admits a flat extension to a moment matrix $\mc M_3$. 

	Let $\beta\equiv \beta^{(4)}$ be a nc sequence such that the moment matrix $\mc{M}_2\equiv \mc{M}_2(\beta)$
	has rank 4. By Corollary \ref{lin-ind-of-4-col} we may assume that the set 
	$\left\{\mds{1},\bX,\bY,\bX\bY\right\}$ is linearly independent
	and hence a basis for the column space $\mc{C}_{\mc{M}_2}$. 
The main result of this section is the following.

\begin{theorem} \label{rank4-soln-aljaz}
	Suppose $\beta\equiv \beta^{(4)}$ is a normalized nc sequence with the moment matrix $\mc{M}_2$ of rank 4.
	Let the set $\left\{\mds{1},\bX,\bY,\bX\bY\right\}$ be a basis for the column space $\mc{C}_{\mc{M}_2}$. Write
		\begin{eqnarray}
				\bX^2&=&a_1 \mds{1}+b_1\bX+c_1\bY +d_1 \bX\bY \label{eq-1-rank4-aljaz},\\
				\bY\bX&=& a_2\mds{1}+b_2\bX+c_2\bY+d_2 \bX\bY \label{eq-2-rank4-aljaz},\\
				\bY^2&=&a_3\mds{1}+b_3\bX+c_3\bY+d_3 \bX\bY \label{eq-3-rank4-aljaz},
		\end{eqnarray}
	where $a_{j},b_{j},c_{j}, d_j\in\mbb{R}$ for $j=1,2,3$.
	The following statements are true:
		\begin{enumerate}
			\item\label{point-1-rank4}	$d_1=d_3=0$, $d_2=-1$.
			\item\label{point-2-rank4} $\beta$ admits a nc measure if and only if $\mc{M}_2$ has a 
				flat extension to a moment matrix $\mc M_3$.
			\item\label{point-3-rank4} $\beta$ admits a nc measure if and only if $\mc{M}_2$ is positive
				semidefinite and
					$$c_1=b_3=0,\quad  
					b_2=c_3,\quad 
					c_2=b_1.$$
			\item\label{point-4-rank4} Suppose $\beta$ admits a nc measure. Then the minimal measure 
			is of type $(0,1)$ with a unique (up to orthogonal equivalence) atom 
			$(X,Y)\in (\mathbb{SR}^{2\times 2})^2$ given
				by 
				\begin{eqnarray*}
					X&=&
						\begin{mpmatrix} \sqrt{a_1+\frac{b_1^2}{4}} & 0 \\ 0 & 
							-\sqrt{a_1+\frac{b_1^2}{4}} 		
						\end{mpmatrix}+\frac{b_1}{2}\cdot I_2 ,\\
					Y&=& \sqrt{a_3+\frac{c_3^2}{4}} \cdot \begin{pmatrix} \frac{a}{2} & 
						\frac{1}{2}\sqrt{4-a^2} \\ 
						\frac{1}{2}\sqrt{4-a^2} & -\frac{a}{2}
						\end{pmatrix}+\frac{c_3}{2}\cdot I_2,
				\end{eqnarray*}
				where $a=\frac{4a_2+2b_1c_3}{\sqrt{(4a_1+b_1^2)(4a_3+c_3^2)}}$ and $I_2$
				is the $2\times 2$ identity matrix.
		\end{enumerate}
\end{theorem}
			
\begin{proof}
	Part (\ref{point-1-rank4}) follows by comparing the rows $\bX\bY$ and $\bY\bX$ on both sides of (\ref{eq-1-rank4-aljaz}),
	(\ref{eq-2-rank4-aljaz}), (\ref{eq-3-rank4-aljaz}) and noticing
	 that the columns $\mds 1$, $\bX$, $\bY$, $\bX^{2}$ and $\bY^2$ have the same entries in the rows 
	 $\bX\bY$ and $\bY\bX$. 
	
 	The implication $(\Leftarrow)$ of (\ref{point-2-rank4}) follows by Theorem \ref{flat-meas}.
	It  remains to prove the converse. 
	If $\beta$ admits a  nc measure, then the extension $\mc{M}_3$ of $\mc{M}_2$, 
	generated by the nc measure, 
	must satisfy RG relations obtained from (\ref{eq-1-rank4-aljaz}), (\ref{eq-2-rank4-aljaz}), (\ref{eq-3-rank4-aljaz});
	see Theorem \ref{support lemma} (\ref{point-3-support}). 
	On multiplying (\ref{eq-1-rank4-aljaz}) from right (resp.\ left) by $\bX$ (resp.\ $\bY$) we conclude that in $\mc{M}_3$ the columns $\bX^3, \bX^2\bY, \bY\bX^2$ lie in the linear span of the columns 
	$\mds{1},\bX,\bY,\bX\bY$.
	By analogous reasoning it follows from (\ref{eq-3-rank4-aljaz}) that the same is true for $\bY^3, \bY^2\bX, \bX\bY^2$. Finally using these conclusions after multiplying (\ref{eq-2-rank4-aljaz}) 
	by $\bX$ (resp.\ $\bY$), the same applies to $\bX\bY\bX, \bY\bX\bY$. Hence $\mc{M}_3$ is a flat extension of $\mc{M}_2$.
	
	Now we prove the implication $(\Rightarrow)$ of (\ref{point-3-rank4}). Let $\mc{M}_3$ be a flat extension of $\mc{M}_2$. Reasoning in the same way as in (\ref{point-2-rank4}) we must have
		$\bX^3=a_1 \bX+b_1 \bX^2+c_1 \bX\bY$
	in the column space of $\mc{M}_3$.
	Since $\bX^3$ has the same entries in rows $\bX\bY$, $\bY\bX$, it follows that $c_1=0$. Analogously we conclude that $b_3=0$.
	Applying an affine linear transformation $\phi_1(X,Y)=(X-\frac{b_1}{2}, Y-\frac{c_3}{2})$ to $\beta$ we get 
	$\widetilde\beta$ with a psd moment matrix $\widetilde{M}_2$ of rank 4 satisfying the relations
		$$
			\bX^2=\big(a_1+\frac{b_1^2}{4}\big) \mds{1},\quad \bX\bY+\bY\bX=a_4 \mds{1}+b_4 \bX+ c_4\bY, \quad \bY^2=\big(a_3+		
				\frac{c_3^2}{4}\big) \mds{1},
		$$
	where
		$$a_4=a_2+\frac{b_1 c_3}{2},\quad b_4=b_2-c_3,\quad c_4=c_2-b_1.$$ 
		
	\noindent \textbf{Claim 1:} $a_1+\frac{b_1^2}{4}>0$ and $a_3+\frac{c_3^2}{4}>0$. \\
	
	If $a_1+\frac{b_1^2}{4}\leq 0$, then
	$\widetilde{\beta}_{X^4}=(a_1+\frac{b_1^2}{4})\widetilde{\beta}_{X^2}\leq 0$. The case $\widetilde{\beta}_{X^4}<0$ contradicts to $\widetilde{\mc{M}}_2$ being psd, while in the case
	 $\widetilde{\beta}_{X^4}=0$ it follows that $\widetilde{\beta}_{X^2}=\widetilde{\beta}_{X^4}=\widetilde{\beta}_{X^2Y^2}=0$ which contradicts to the rank of $\widetilde{\mc{M}}_2$ being 4.
	 Analogously we conclude that $a_3+\frac{c_3^2}{4}>0$.\\

	 Applying an affine linear transformation 
	 $\phi_2(X,Y)=\Big(\frac{X}{\sqrt{a_1+\frac{b_1^2}{4}}}, \frac{Y}{\sqrt{ a_3+\frac{c_3^2}{4}} }\Big)$
	to $\widetilde{\beta}$ we get $\widehat{\beta}$ with a psd moment 
	matrix $\widehat{\mc{M}}_2$ of rank 4 satisfying the relations
			\begin{equation}\label{new-rel-r4-trans-2}
				\bX^2=\mds{1},\quad \bX\bY+\bY\bX=a \mds{1}+b \bX+ c\bY, \quad \bY^2=\mds{1},
			\end{equation}
	where
		\begin{equation} \label{exp-a-b-c}
			a=\frac{4a_2+2b_1 c_3}{C},\quad 
				b=\frac{4(b_2-c_3)}{C},\quad
			c=\frac{4(c_2-b_1)}{C},
		\end{equation}
	and $C=\sqrt{(4a_1+b_1^2)(4a_3+c_3^2)}.$
	By RG relations it follows from (\ref{new-rel-r4-trans-2}) that the extension $\widehat{\mc{M}}_3$
	of $\widehat{\mc{M}}_2$ satisfies the relations 
		\begin{multicols}{2}
		\begin{equation*}
			\begin{aligned}
				\bX^2\bY&=\bY,\\
				\bX\bY^2+\bY\bX\bY&= a \bY+b \bX\bY+ c\bY^2,
			\end{aligned}
		\end{equation*}
		\vfill
	\columnbreak
		\begin{equation*}
			\begin{aligned}
				\bX^2\bY+\bX\bY\bX&=a \bX+b \bX^2+ c\bX\bY,\\
				\bY^2\bX&=\bX.
			\end{aligned}
		\end{equation*}
	\end{multicols}
	\noindent In particular, 
		\begin{equation}\label{new-rel-r4-after-rg}
			\bX+\bY\bX\bY=a \bY+b \bX\bY+ c\bY^2,\quad
			\bY+\bX\bY\bX=a \bX+b \bX^2+ c\bX\bY.
		\end{equation}
	Observing the rows $\bX\bY$ and $\bY\bX$ on the both sides of the equations in (\ref{new-rel-r4-after-rg}) 
	and noticing that the columns $\bY$, $\bY\bX\bY$, $\bY^2$, $\bX$, $\bX^2$ have the same entries in the rows $\bX\bY,$ $\bY\bX$,
	we must have $b=c=0$.
	Hence, (\ref{exp-a-b-c}) implies that $b_2=c_3$ and $c_2=b_1$.
	
	Now we prove the implication $(\Leftarrow)$ of (\ref{point-3-rank4}).
	As above let 
		$$\phi_1(X,Y)=\Big(X-\frac{b_1}{2}, Y-\frac{c_3}{2}\Big),\quad 
	\phi_2(X,Y)=\Big(\frac{X}{\sqrt{a_1+\frac{b_1^2}{4}}}, \frac{Y}{\sqrt{ a_3+\frac{c_3^2}{4}} }\Big).$$
	Applying transformation $\phi_2\circ\phi_1$ to $\beta$ we get
	$\widehat\beta$ with a psd moment matrix $\widehat{\mc{M}}_2$ of rank 4 satisfying the relations
		\begin{equation} \label{system-rank4}
			\bX^2=\mds{1},\quad \bX\bY+\bY\bX=a\mds 1,\quad \bY^2=\mds{1},
		\end{equation}
	where $a=\frac{4a_2+2b_1 c_3}{\sqrt{(4a_1+b_1^2)(4a_3+c_3^2)}}$.
	We have to prove that $\widehat\beta$ admits a nc measure.
	From the relations (\ref{system-rank4}) we get the following system
		\begin{multicols}{3}
		\begin{equation}\label{rank4-RG}
			\begin{aligned}
				\beta_{X^2}&=1,\\
				\beta_{X^{3}}&=\beta_X,\\
				\beta_{X^2Y}&=\beta_Y,\\
				\beta_{X^{4}}&=\beta_{X^2},\\
				\beta_{X^{3}Y}&=\beta_{XY},\\
				\beta_{X^2Y^{2}}&=\beta_{Y^2},
			\end{aligned}
		\end{equation}
	\vfill
	\columnbreak
		\begin{equation*}
			\begin{aligned}
				2\beta_{XY}&=a,\\
				2\beta_{X^2Y}&=a \beta_X,\\
				2\beta_{XY^2}&=a\beta_{Y},\\
				2\beta_{X^{3}Y}&=a\beta_{X^{2}},\\
				\beta_{X^2Y^{2}}+\beta_{XYXY}&=a\beta_{XY},\\
				2\beta_{XY^{3}}&=a\beta_{Y^{2}},
			\end{aligned}
		\end{equation*}
	\vfill
	\columnbreak
		\begin{equation*}
			\begin{aligned}
				\beta_{Y^2}&=1,\\
				\beta_{XY^2}&=\beta_{X},\\
				\beta_{Y^3}&=\beta_{Y},\\
				\beta_{X^2Y^2}&=\beta_{X^{2}},\\
				\beta_{XY^3}&=\beta_{XY},\\
				\beta_{Y^4}&=\beta_{Y^2}.
			\end{aligned}
		\end{equation*}
	\end{multicols}
	\noindent The solutions to (\ref{rank4-RG}) are given by
	\begin{equation*}
		\begin{aligned}	
			\beta_{Y^4}=\beta_{X^4}=\beta_{X^2Y^2}=\beta_{Y^2}=\beta_{X^2} &=1 \\
			\beta_{XY^3}=\beta_{X^3Y}=\beta_{XY}&=\frac{a}{2},\\
			\beta_{XYXY}&=\frac{a^2}{2}-1,\\
		\end{aligned}
	\end{equation*}
and
	one of the following:
	\begin{enumerate}
		\item[Case 1:] $\beta_{XY^2}=\beta_{X^2Y}=\beta_{Y^3}=\beta_{Y}=\beta_{X^3}=\beta_{X}=0$,
		\item[Case 2:] $a=2$ and $\beta_{X^3}=\beta_{X}=\beta_Y=\beta_{Y^3}=\beta_{XY^2}=\beta_{X^2Y}\in \RR$,
		\item[Case 3:] $a=-2$ and $\beta_{X^3}=\beta_{X}=-\beta_Y=-\beta_{Y^3}=\beta_{XY^2}=-\beta_{X^2Y}\in \RR.$
	\end{enumerate}
	However, in Cases 2 and 3 the submatrices
	$[\widehat{\mc{M}}_2]_{\{\bX,\bY\}}$ are of the form
		$\begin{pmatrix} 1 & \pm 1 \\ \pm 1 & 1 \end{pmatrix}$
	and are not positive definite. Hence we are in Case 1 and $\widehat{\mc{M}}_2$ takes the form 
 	 	\begin{equation}\label{form-of-M(2)-rank4-aljaz}
					\widehat{\mc{M}}_2=
					\begin{mpmatrix}
						1& 0 & 0 & 1 & \frac{a}{2} & \frac{a}{2} & 1\\
						0& 1 & \frac{a}{2} & 0 &0 & 0 & 0\\
						0& \frac{a}{2} & 1 & 0 &0 & 0 & 0\\
						1 & 0 & 0 & 1 & \frac{a}{2} & \frac{a}{2}  & 1\\
						\frac{a}{2} & 0 & 0 & \frac{a}{2} & 1 & \frac{a^2}		
							{2}-1 & \frac{a}{2}\\
						\frac{a}{2}& 0 & 0 & \frac{a}{2} & \frac{a^2}{2}-1 & 1 & 							\frac{a}{2} \\
						1& 0 & 0 & 1 & \frac{a}{2} & \frac{a}{2}  & 1\\
					\end{mpmatrix}.
		 \end{equation}	
	$\widehat{\mc{M}}_2$ is psd if and only if $a\in (-2,2)$.
	Now notice that the representing atom $(\widehat X,\widehat Y)$ for $\widehat{\mc{M}}_2$
	is given by the pair			
			\begin{equation} \label{representing measure-aljaz}
				\widehat X=\begin{pmatrix} 1 & 0 \\ 0 & -1	\end{pmatrix},\quad 
			 	\widehat Y=\begin{pmatrix} \frac{a}{2} & -\frac{1}{2}\sqrt{4-a^2} \\ 
					 -\frac{1}{2}\sqrt{4-a^2} & -\frac{a}{2} \end{pmatrix}.
			\end{equation}
	This proves the implication $(\Leftarrow)$ of (\ref{point-3-rank4}).
	
	It remains to prove part (\ref{point-4-rank4}). 
	Let $\phi_1$, $\phi_2$ and $\widehat{\mc{M}}_2$ be as in the proof of part (\ref{point-3-rank4}). 
	By Proposition \ref{linear transform invariance-nc} (\ref{invariance-point5})
	the measures $\mu$ for $\beta$ are in the bijective correspodence
	with the measures $\widehat{\mu}$ for $\widehat{\beta}$ given by the rule $\mu=\widehat \mu (\phi_2\circ\phi_1)$
	and $\supp(\mu)=(\phi_2\circ \phi_1)^{-1}(\supp(\widehat \mu))$. We have that 
		$$(\phi_2\circ \phi_1)^{-1}(X,Y)=
			\Big(\sqrt{a_1+\frac{b_1^2}{4}} \cdot X+\frac{b_1}{2}, \sqrt{a_3+\frac{c_3^2}{4}} \cdot Y+\frac{c_3}{2}\Big).$$
	Therefore it suffices to prove the following claim.\\
	
	\noindent \textbf{Claim 2:} The atom $(\widehat X,\widehat Y)\in (\mathbb{SR}^{2\times 2})^2$ 
		of the form (\ref{representing measure-aljaz}) is up to orthogonal equivalence the unique atom for
		the measure $\widehat \mu$ of $\widehat{\beta}$.\\
	
	Let  
	$(X,Y)\in (\mathbb{SR}^{2\times 2})^2$ 
	be an atom representing $\widehat{\beta}$.
	Since $X,Y$ do not commute and $X^2=I_2$, the eigenvalues of $X$ must be $1$, $-1$.
	Hence we may assume (after conjugating by a suitable orthogonal matrix) that 
		$$X=\widehat X\quad \text{and} \quad
			Y=\begin{pmatrix} k_1 & k_2\\ k_2 & -k_1	\end{pmatrix}\quad \text{for some }k_1,k_2\in\RR.$$
	The equality $Y^2=I_2$ implies that $k_1^2+k_2^2=1$, while from $\Tr(XY)=\frac{a}{2}$ we get that $k_1=\frac{a}{2}$ and hence $k_2=\pm\frac{1}{2}\sqrt{4-a^2}$. It is now easy to check that the moment matrix $\mc{M}^{(X,Y)}_2$ generated by $(X,Y)$ where $(k_1,k_2)=(\frac{a}{2},\pm\frac{1}{2}\sqrt{4-a^2})$ is equal to $\widehat{\mc{M}}_2$. 
	Since both solutions are unitarily equivalent pairs (the orthogonal equivalence acts on the standard basis vectors $e_1$, $e_2$ as $e_1\mapsto e_1$, $e_2\mapsto -e_2$), this proves Claim 2.
\end{proof}

The following corollary will be very important in the proofs of theorems about the existence of a nc measure in the rank 5 case.

\begin{corollary}\label{nc-TTMM-cor}
	Suppose $\beta\equiv \beta^{(4)}$ has a psd moment matrix $\mc{M}_2$ of rank 4
	satisfying the relations 
		$$\bX^2=a\mds 1, \quad \bX\bY+\bY\bX=b\mds 1, \quad \bY^2=c\mds 1\quad \text{for some }
			a,b>0, c\in \RR.$$
	Then
		\begin{equation}\label{zero-moments}
			\beta_{X}=\beta_Y=\beta_{X^3}=\beta_{X^2Y}=\beta_{XY^2}=\beta_{Y^3}=0.
		\end{equation}
\end{corollary}

\begin{proof}
		Applying an affine linear 
	transformation $\phi(X,Y)=(\frac{X}{\sqrt{a}},\frac{Y}{\sqrt{c}})$ to $\mc{M}_2$
	the relations of the corresponding matrix $\widetilde{M}_2$ become
		$$X^2=\mds 1, \quad \bX\bY+\bY\bX=\frac{b}{\sqrt{ac}}\mds 1,\quad \bY^2=\mds 1.$$
	Then $\widetilde{M}_2$ is of the form (\ref{form-of-M(2)-rank4-aljaz}) (where we replace 
	$a$ with $\frac{b}{\sqrt{ac}}$). In particular, we have
		\begin{equation}\label{zero-moments-1}
			\widetilde{\beta}_{X}=\widetilde{\beta}_Y=\widetilde{\beta}_{X^3}=
			\widetilde{\beta}_{X^2Y}=\widetilde{\beta}_{XY^2}=\widetilde{\beta}_{Y^3}=0.
		\end{equation}
	Since the moments $\beta_w$ and $\widetilde \beta_w$ for $|w|\leq 4$
	are scalar multiples of each other, (\ref{zero-moments}) follows from
	(\ref{zero-moments-1}).
\end{proof}

\section{Ranks 5 and 6 - reductions}\label{reduc}

In this section we establish an essential result for solving a BQTMP with a moment matrix of rank 5 or 6.
Namely, it suffices to solve the BQTMP only for moment matrices satisfying especially nice column relations;
see Proposition \ref{structure-of-rank5-2}.
In the subsequent sections we will analyze each of those cases separately.


\begin{proposition}\label{structure-of-rank5-2}
		Suppose a nc sequence $\beta\equiv \beta^{(4)}$ has a moment matrix $\mc{M}_2$ of rank 5 or 6.
	Let $L_\beta$ be the Riesz functional belonging to $\beta.$
	If $\beta$ admits a nc measure, then there exists an
	affine linear transformation $\phi$ such that a sequence $\widehat \beta$, given by 
		$\widehat \beta_w=L_\beta(w\circ \phi)$ for every $|w(X,Y)|\leq 4$,
	has a moment matrix $\widehat{\mc{M}}_2$ such that:
	\begin{enumerate} 
		\item\label{point-1-str-rank5} If $\mc{M}_2$ is of rank 5, then $\widehat{\mc{M}}_2$ satisfies $\bX\bY+\bY\bX=\mbf 0$ and one of
			the following relations:
	\begin{description}
		\item[\normalfont\textit{Basic case 1}] 	$\bX^2+\bY^2=\mds{1}$, 
		\item[\normalfont\textit{Basic case 2}]		$\bY^2=\mds{1}$,
		\item[\normalfont\textit{Basic case 3}]  	$\bY^2-\bX^2=\mds{1},$
		\item[\normalfont\textit{Basic case 4}]  	$\bY^2=\bX^2$.
	\end{description}
		\item\label{point-2-str-rank5}  If $\mc{M}_2$ is of rank 6, then $\widehat{\mc{M}}_2$ satisfies one of
			the following relations:
	\begin{description}
		\item[\normalfont\textit{Basic relation 1}]  $\bY^2=\mds{1}-\bX^2$,
		\item[\normalfont\textit{Basic relation 2}]  $\bY^2=\mds{1}+\bX^2,$
		\item[\normalfont\textit{Basic relation 3}]  $\bX\bY+\bY\bX=\mbf 0,$
		\item[\normalfont\textit{Basic relation 4}] $\bY^2=\mds{1}.$
	 \end{description}
	\end{enumerate}
\end{proposition}

To prove Proposition \ref{structure-of-rank5-2} we need some lemmas.

\begin{lemma} \label{Y^2-blueuction-lemma}
	Suppose a nc sequence $\beta\equiv \beta^{(4)}$ has a moment matrix $\mc{M}_2$ of rank 5 or 6 satisfying the relation
		\begin{equation}\label{eq-y2}
			\bY^2=a_1\mds 1 + a_2 \bX + a_3 \bY + a_4 \bX^2+ a_5 \bX\bY +a_6 \bY\bX,
		\end{equation}
	where $a_i\in \RR$ for each $i$.
	Let $L_\beta$ be the Riesz funtional belonging to $\beta$. If $\beta$ admits a nc measure, then there
	exists an affine linear transformation $\phi$ of the form
		\begin{equation}\label{form-of-phi}
			\phi(X,Y)=(\alpha_1 X+\alpha_2,\alpha_3 Y+\alpha_4),
		\end{equation}
	where $\alpha_i\in \RR$ for each $i$, $\alpha_1\neq 0$, $\alpha_4\neq 0$,
		 such that the sequence $\widehat\beta$ given by 
	$\widehat\beta_w=L_{\beta}(w\circ \phi)$ for every $|w(X,Y)|\leq 4$, has a moment matrix $\widehat{\mc{M}}(2)$
	satisfying one of the following relations:
	 \begin{description}
		\item[\normalfont\textit{Relation 1}]  $\bY^2=\mds{1}-\bX^2$,
		\item[\normalfont\textit{Relation 2}] $\bY^2=\mds{1},$
		\item[\normalfont\textit{Relation 3}]  $\bY^2=\mds{1}+\bX^2,$
		\item[\normalfont\textit{Relation 4}]  $\bY^2=\bX^2.$
	 \end{description}
	 Moreover, the relation 4 is equivalent to the relation
	 \begin{description}
		\item[\normalfont\textit{Relation 4'}]  $\bX\bY+\bY\bX=\mbf 0.$
	 \end{description}
\end{lemma}

\begin{proof}
	By comparing the rows $\bX\bY$, $\bY \bX$ on both sides of (\ref{eq-y2}) 
	we conclude that 
	$a_5=a_6$.
	We rewrite the relation (\ref{eq-y2}) as
		\begin{equation*}\label{relation-r6-2-new} 
			(\bY-a_5\bX)^2=a_1\mds{1}+a_2\bX+a_3\bY+(a_4+a_5^2)\bX^2.
		\end{equation*}
	Applying an affine linear transformation $\phi_1(X,Y) = (X,Y-a_5 X )$ to $\beta$
	we get $\widetilde \beta$ with the moment matrix $\widetilde{\mc{M}}(2)$
	satisfying the relation 
		\begin{equation}\label{relation-r6-3-new} 
			\bY^2=a_1\mds{1}+(a_2+a_3 a_5)\bX+a_3\bY+a_4 \bX^2.
		\end{equation}
	We separate three possibilities according to the sign of $a_4\in \RR$.\\

	\noindent \textbf{Case 1:} $a_4<0$. The relation (\ref{relation-r6-3-new}) can be rewritten as
		$$\Big(\bY-\frac{a_3}{2}\Big)^2=-\Big(\sqrt{|a_4|} \bX-\frac{a_2+a_3 a_5}{2\sqrt{|a_4|}}\Big)^2+\Big(a_1+
			\frac{a_3^2}{4}+\frac{(a_2+a_3 a_5)^2}{4a_4}\Big)\mds 1.$$
	Applying an affine linear transformation 
		$\phi_2(X,Y)=(\sqrt{|a_4|} X-\frac{a_2+a_3 a_5}{2\sqrt{|a_4|}},Y-\frac{a_3}{2})$ 
	to $\widetilde{\beta}$ we get $\overline\beta$ with $\overline{\mc{M}}(2)$ satisfying the relation
		\begin{equation}\label{y2-trans}
			\bY^2=-\bX^2+\Big(a_1+\frac{a_3^2}{4}+\frac{(a_2+a_3 a_5)^2}{4a_4}\Big)\mds 1.
		\end{equation}
	If $C_1:=a_1+\frac{a_3^2}{4}+\frac{(a_2+a_3 a_5)^2}{4a_4}\leq 0$, then by comparing the row $\bY^2$ on both sides of (\ref{y2-trans})
	 we get
		$$0\leq \beta_{Y^4}+\beta_{X^2Y^2}= C_1\cdot \beta_{Y^2}\leq 0,$$
	where we used that $\beta_{Y^4}\geq 0$, $\beta_{X^2Y^2}\geq 0$, $\beta_{Y^2}\geq 0$.
	But then $\beta_{Y^4}=\beta_{X^2Y^2}=\beta_{Y^2}=0$, 
	which contradicts to the rank of $\widetilde{\mc{M}}(2)$ being 5 or 6. 
	Therefore $C_1>0$.
	Applying an affine linear transformation 
		$\phi_3(X,Y)=(\frac{X}{\sqrt{C_1}},\frac{Y}{\sqrt{C_1}})$ 
	to $\overline{\beta}$ we get $\widehat \beta$ with $\widehat{\mc{M}}(2)$ satisfying
		$$\bY^2=\mds 1-\bX^2,$$
	which is the relation 1.\\

	\noindent \textbf{Case 2:} $a_4=0$. Multiplying (\ref{relation-r6-3-new}) with $\bY$ we get
		\begin{equation}\label{mult-by-y}
			\bY^3=a_1\bY+(a_2+a_3a_5)\bX\bY+a_3\bY^2.
		\end{equation}
	By comparing the rows $\bX\bY$, $\bY \bX$ on both sides of (\ref{mult-by-y}) 
	we conclude that $a_2+a_3 a_5=0$. 
	We can rewrite 	 (\ref{relation-r6-3-new})  as
		$$\Big(\bY-\frac{a_3}{2}\Big)^2=\Big(a_1+\frac{a_3^2}{4}\Big)\mds 1.$$
	Applying an affine linear transformation $\phi_4(X,Y)=(X,Y-\frac{a_3}{2})$ to $\widetilde{\beta}$ we get
	$\overline\beta$ with $\overline{\mc{M}}(2)$ satisfying
		\begin{equation}\label{y2-trans-2}
			\bY^2=\Big(a_1+\frac{a_3^2}{4}\Big)\mds 1.
		\end{equation}
	If $C_2:= a_1+\frac{a_3^2}{4}\leq 0$, then by comparing the row $\bY^2$ on both sides of 
	(\ref{y2-trans-2}) we get
		$$0\leq \beta_{Y^4}=(a_2+\frac{c_2^2}{4})\beta_{Y^2}\leq 0,$$
	where we used that $\beta_{Y^4}\geq 0$, $\beta_{Y^2}\geq 0$.
	But then $\beta_{Y^4}=\beta_{Y^2}=0$ and hence also $\beta_{X^2Y^2}=0$, 
	which contradicts to the rank of $\widetilde{\mc{M}}(2)$ being 5 or 6. 
	Therefore $C_2>0$. Applying an affine linear transformation
	$\phi_5(X,Y)=(X,\frac{Y}{\sqrt{C_2}})$ to $\overline{\beta}$ we get $\widehat\beta$ with
	$\widehat{\mc{M}}(2)$ satisfying 
		$$\bY^2=\mds 1,$$
	which is the relation 2.\\

	\noindent \textbf{Case 3:} $a_4>0$. The relation (\ref{relation-r6-3-new}) can be rewritten as
		$$\Big(\bY-\frac{a_3}{2}\Big)^2=\Big(\sqrt{a_4} \bX+\frac{a_2+a_3 a_5}{2\sqrt{a_4}}\Big)^2+
		\Big(a_1+\frac{a_3^2}{4}-
				\frac{(a_2+a_3 a_5)^2}{4a_4}\Big)\mds 1.$$
	Applying an affine linear transformation 
		$\phi_6(X,Y)=(\sqrt{a_4} X+\frac{a_2+a_3 a_5}{2\sqrt{a_4}},Y-\frac{a_3}{2})$ 
	to $\widetilde{\beta}$ we get $\overline \beta$ with $\overline{\mc{M}}(2)$ satisfying
		\begin{equation}\label{case-1-3-r6}
			\bY^2=\bX^2+\Big(a_1+\frac{a_3^2}{4}-\frac{(a_2+a_3 a_5)^2}{4a_4}\Big)\mds 1.
		\end{equation}
	We separate three possibilities according to the sign of 
	$C_3:=a_1+\frac{a_3^2}{4}-\frac{(a_2+a_3 a_5)^2}{4a_4}$.\\
	
	\noindent \textbf{Case 3.1:} $C_3>0$.
	Applying an affine linear transformation 
		$\phi_7(X,Y)=(\frac{X}{\sqrt{C_3}},\frac{Y}{\sqrt{C_3}})$ 
	to $\overline{\beta}$ we get $\widehat\beta$ with $\widehat{\mc{M}}(2)$ satisfying 
		$$\bY^2=\mds 1+\bX^2,$$
	which is the relation 3.\\
	
	\noindent \textbf{Case 3.2:} $C_3=0$.
	The relation $(\ref{case-1-3-r6})$ is
		$$\bY^2=\bX^2,$$
	which is the relation 4. Applying an affine linear transformation 
		$\phi_8(X,Y)=(X-Y,X+Y)$ 
	to $\widetilde{\beta}$ we get $\overline \beta$ with $\overline{\mc{M}}(2)$ satisfying
		$$\bX\bY+\bY\bX=\mbf 0,$$
	which is the relation 4'.\\
	
	
	\noindent \textbf{Case 3.3:} $C_3<0$. 
	Applying an affine linear transformation 
		$\phi_9(X,Y)=(Y,X)$ 
	to $\overline{\beta}$ we come into Case 3.1. 
\end{proof}

\begin{lemma}\label{structure-of-rank5-prep}
	Suppose a nc sequence $\beta\equiv \beta^{(4)}$
	has a moment matrix $\mc{M}_2$ of rank 5 with linearly independent columns
	$\mds{1}$, $\bX$, $\bY$, $\bX\bY$.
 	Then one of
	the following cases occurs:
		\begin{enumerate}
			\item[Case 1:] The set $\left\{\mds{1}, \bX, \bY, \bX\bY, \bY\bX\right\}$ is the basis for 
				$\mathcal C_{\mc{M}_2}$ and the columns
				$\bX^2,\bY^2$ belong to the $\Span\{\mds{1},\bX,\bY\}$.
			\item[Case 2:] The set $\left\{\mds{1}, \bX, \bY, \bX^2, \bX\bY\right\}$ is the basis for 
				$\mathcal C_{\mc{M}_2}$.
			\item[Case 3:] The set $\left\{\mds{1}, \bX, \bY, \bY^2, \bY\bX\right\}$ is the basis for 
				$\mathcal C_{\mc{M}_2}$.
		\end{enumerate}
\end{lemma}

\begin{proof}
			If  $\bX^2\notin \text{\normalfont span}\{\mds{1},\bX,\bY\}$, it follows by comparing the rows $\bX\bY$ and $\bY\bX$ that $\bX^2\notin \text{\normalfont span}\{\mds{1},\bX,\bY,\bX\bY\}$.
		Hence we are in Case 2. Similarly, if $\bY^2\notin \text{\normalfont span}\{\mds{1},\bX,\bY\}$, then we are in Case 3. Otherwise 
		$\{\bX^2,\bY^2\}\subseteq \text{\normalfont span}\{\mds{1},\bX,\bY\}$ and since $\mc{M}_2$ is of rank 5, 
		$\{\mds{1},\bX,\bY,\bX\bY,\bY\bX\}$ is a basis for $\mathcal C_{\mc{M}_2}$. Hence we are in Case 1.
\end{proof}

\begin{lemma} \label{blue-2-pos-2}
	Suppose a nc sequence $\beta\equiv \beta^{(4)}$
	has a moment matrix $\mc{M}_2$ of rank 6 with linearly independent columns
	$\mds{1}$, $\bX$, $\bY$, $\bX\bY$.
	There exists an affine linear transformation $\phi$ such that a sequence $\widehat \beta$, given by 
		$\widehat \beta_w=L_\beta(w\circ \phi)$ for every $|w(X,Y)|\leq 4$,
	has a moment matrix $\widehat{\mc{M}}_2$ such that:
	\begin{enumerate}
		\item[Case 1:] The set $\left\{\mds{1},\bX,\bY,\bX^2,\bX\bY,\bY\bX\right\}$ is the basis for
			$\mathcal{C}_{\mc{\widehat M}_2}$.
		\item[Case 2:] The set $\left\{\mds{1},\bX,\bY,\bX^2,\bX\bY,\bY^2\right\}$ is the basis for
			$\mathcal{C}_{\mc{\widehat M}_2}$.
	\end{enumerate}
\end{lemma}

\begin{proof}
	If $\bY^2\in \Span\left\{\mds{1},\bX,\bY,\bX^2,\bX\bY,\bY\bX\right\}$, then we are in Case 1 
	(linear independence of the columns $\left\{\mds{1}, \bX, \bY,  \bX^2, \bX\bY, \bY\bX\right\}$ follows from the 
	rank of $\mc{M}_2$ being 6). Otherwise 
	$\bY^2\notin \Span\left\{\mds{1},\bX,\bY,\bX^2,\bX\bY,\bY\bX\right\}$.
	In particular, $\left\{\mds{1},\bX,\bY,\bX\bY,\bY^2\right\}$ is a linearly independent set. Now we have two 
	possibilities. Either $\left\{\mds{1},\bX,\bY,\bX\bY,\bY^2,\bX^2\right\}$ is a linearly independent set
	and we are in Case 2, or 
	$\bX^2\in \Span \left\{\mds{1},\bX,\bY,\bX\bY,\bY^2\right\}$ and $\left\{\mds{1},\bX,\bY,\bX\bY,\bY^2,\bY\bX\right\}$ is a linearly 
	independent set. After applying an affine linear transformation $\phi(X,Y)=(Y,X)$ we are 
	in Case 1.
\end{proof}

\begin{lemma} \label{symmetry-2-aljaz}
		Suppose a nc sequence $\beta\equiv \beta^{(4)}$ has a moment matrix $\mc{M}_2$ satisfying one of the relations 
		$$\bY^2+\bX^2=\mds{1}\quad\text{or}\quad \bY^2-\bX^2=\mds{1}\quad \text{or}\quad 
			\bY^{2}=\bX^{2}.$$
	If $\beta$ admits a nc measure $\mu$, then the extension 
		$\mc{M}_3:=\begin{pmatrix} \mc{M}_2 & B_3 \\ B_3^t & C_3 \end{pmatrix}$ 
	generated by $\mu$ satisfies the relations
		 $$\bX^2\bY=\bY\bX^2\quad\text{and}\quad \bX\bY^2=\bY^2\bX.$$ 
		In particular, the rows $\bX\bY$, $\bY\bX$ are the same in the columns 
	$\bX^2 \bY$, $\bY\bX^2$ and the columns $\bX\bY^2$, $\bY^2\bX$.
\end{lemma}
	
\begin{proof}
	We will give the proof in the case of the relation $\bY^2+\bX^2=\mds{1}$. The other two cases are proved 		
	in the same way.
	Multiplying $\bY^2+\bX^2=\mds{1}$ by $\bY$ (resp.\ $\bX$) from the left (resp.\ right) gives 
		$\bX^2\bY=-\bY^3+\bY=\bY\bX^2$ and $\bY^2\bX=\bX-\bX^3=\bX\bY^2.$
	By Theorem \ref{support lemma} we must have $\bX^2\bY=\bY\bX^2$ and $\bX\bY^2=\bY^2\bX$ in 
	$\mc{M}_3$.
\end{proof}

Finally we give the proof of Proposition  \ref{structure-of-rank5-2} (\ref{point-1-str-rank5}).

\begin{proof}[Proof of Proposition \ref{structure-of-rank5-2} (1)]
	By Proposition \ref{lin-ind-of-4-col} the columns $\mds{1}$, $\bX$, $\bY$, $\bX\bY$ of 
	$\mc{M}_2$ are linearly independent.
	By Lemma \ref{structure-of-rank5-prep} there are three cases to consider.\\
	
	\noindent \textbf{Case 1: The set
		$\left\{\mds{1},\bX,\bY,\bX\bY,\bY\bX\right\}$  is the basis for 
				$\mathcal C_{\mc{M}_2}$ and the columns
				$\bX^2,\bY^2$ belong to the $\Span\left\{\mds{1},\bX,\bY\right\}$.}\\
	
	By assumption there are 
	constants $a_j, b_j, c_j \in \mbb{R}$ for $j=1,2$ such that 
		$$\bX^{2}=a_1\mds{1}+b_1 \bX+c_1 \bY\quad \text{and}\quad 
			\bY^{2}=a_2 \mds{1}+b_2 \bX+c_2 \bY.$$
	By multiplying the first relation with $\bX$ and the second with $\bY$
	it follows that if $\beta$ admits a nc measure, then 
	$c_1=b_2=0$. Let 
		$$\phi_1(X,Y)=\Big(X-\frac{b_1}{2}, Y-\frac{c_3}{2}\Big),\quad
		\phi_2(X,Y)=\Big(\frac{X}{\sqrt{a_1+\frac{b_1^2}{4}}}, \frac{Y}{\sqrt{ a_3+\frac{c_3^2}{4}} }\Big).$$
	Applying an affine linear transformation $\phi_2\circ \phi_1$ to $\beta$ we get $\widetilde\beta$
	with $\widetilde{M}(2)$ satisfying
		$$\bX^2=\bY^2=\mds{1}.$$
	Equivalently, the relations are 
	\begin{equation*}\label{rel-new} 
		\bY^2-\bX^2=\mbf 0,\quad \bY^2=\mds{1}.
	\end{equation*}
	Finally applying an affine linear transformation 
		$\phi_3(X,Y)=(\frac{X+Y}{2},\frac{Y-X}{2})$ to  $\widetilde{\beta}$
	we get $\widehat\beta$ with $\widehat{\mc{M}}(2)$ satisfying 
		$$\bX\bY+\bY\bX=\mbf 0,\quad \bX^2+\bY^2=\mds{1}.$$
	Hence we are in a basic case 1 of Proposition \ref{structure-of-rank5-2}.\\
	
	\noindent{ \textbf{Case 2:  The set $\left\{\mds{1},\bX,\bY,\bX^2,\bX\bY\right\}$ is the basis for 
				$\mathcal C_{\mc{M}_2}$.}}\\
	
	By assumption there are constants $a_j, b_j, c_j,d_j, e_j \in \mbb{R}$ for $j=1,2$ such that 
		$$\bY\bX = a_1\mds{1}+b_1\bX+c_1\bY+d_1\bX^2+e_1\bX\bY,\quad 
			\bY^2 = a_2\mds{1}+b_2\bX+c_2\bY+d_2\bX^2+e_2\bX\bY.$$
	By comparing the rows $\bX\bY$, $\bY\bX$ of the both sides of equations we conclude that $e_1=-1$ and $e_2=0$,
	so that the relation are 
		\begin{equation} \label{rel 2}
			\bX\bY+\bY\bX = a_1\mds{1}+b_1\bX+c_1\bY+d_1\bX^2\quad
			\text{and}\quad \bY^2 = a_2\mds{1}+b_2\bX+c_2\bY+d_2\bX^2.
		\end{equation}
	By Lemma \ref{Y^2-blueuction-lemma} there exists an affine linear transformation $\phi_4$
	of the form (\ref{form-of-phi})
	such that after  applying $\phi_4$ to $\beta$ the second relation in (\ref{rel 2}) of the corresponding matrix $\overline{\mc{M}}(2)$ becomes
	one of the following:
		\begin{equation}\label{second-rel}
			\bY^2=\mds 1\quad\text{or}\quad 
			\bY^2=\mds 1-\bX^2\quad\text{or}\quad 
			\bY^2=\bX^2\quad\text{or}\quad
			\bY^2=\mds 1+\bX^2,
		\end{equation}
	while the first relation in (\ref{rel 2}) becomes 
		\begin{equation}\label{first-rel}
			\bX\bY+\bY\bX = a_3\mds{1}+b_3\bX+c_3\bY+d_3\bX^2,
		\end{equation}
	where $a_3,b_3,c_3,d_3\in \RR$. We separate four possibilities according to the relation
	in (\ref{second-rel}).\\
	
	\noindent \textbf{Case 2.1: $\bY^2=\mds 1$ in (\ref{second-rel}).}
	The relation (\ref{first-rel}) can be rewritten in the form
	\begin{equation*} 
				\bY\big(\bX-\frac{c_3}{2}\big)+\big(\bX-\frac{c_3}{2}\big)\bY=
				a_3\mds{1}+b_3\bX+d_3 \bX^2.
	\end{equation*}
	Applying an affine linear transformation 
		$\phi_5(X,Y)=(X-\frac{c_3}{2},Y)$ to  
	$\overline{\beta}$ we get $\breve{\beta}$ with $\breve{\mc{M}}(2)$ satisfying
		\begin{equation}\label{rel 2.1}
			\bX\bY+\bY\bX = a_4\mds{1}+b_4\bX+d_4\bX^2\quad \text{and}\quad \bY^2 = \mds{1},
		\end{equation}
	where $a_4,b_4,d_4\in \RR$.
	Multiplying the first relation in (\ref{rel 2.1}) with $\bX$ on left (resp.\ right) we get
		$$\bX^2\bY+\bX\bY\bX=a_4 \bX+b_4\bX^2+d_4\bX^3=\bX\bY\bX+\bY\bX^2.$$
	Hence, $\bX^2\bY=\bY\bX^2$.  Multiplying the first relation in (\ref{rel 2.1}) with $\bY$ on right and using the second relation in (\ref{rel 2.1}),
	we get
		\begin{equation} \label{rel 2.1 (1)}
			\bX+\bY\bX\bY= a_4 \bY+b_4 \bX\bY+d_4 \bX^2\bY.
		\end{equation}
	Comparing the rows $\bX\bY$, $\bY\bX$ on both sides of (\ref{rel 2.1 (1)}) gives $b_4=0$.
	We now separate two possibilities depending on $d_4$.\\
	
\noindent \textbf{Case 2.1.1: $d_4= 0$ in (\ref{rel 2.1}).}
The relations (\ref{rel 2.1}) are
	$$ \bX\bY+\bY\bX=a_4\mds{1},\quad \bY^2=\mds{1}.$$
Using the second relation we can rewrite the first relation in the form
	$$\big(\bX-\frac{a_4}{2}\bY\big) \bY + \bY \big(\bX-\frac{a_4}{2}\bY\big)=\mbf 0.$$
Applying an affine linear transformation 
		$\phi_6(\bX,\bY)=(x-\frac{a_4}{2}y,y)$ to  
$\breve{\beta}$ we get $\widehat\beta$ with $\widehat{\mc{M}}(2)$ satisfying
		$$\bX\bY+\bY\bX = \mbf 0,\quad \bY^2 = \mds{1}.$$
Hence we are in the basic case 2 of Proposition \ref{structure-of-rank5-2} (1).\\
	
\noindent \textbf{Case 2.1.2: $d_4\neq 0$ in (\ref{rel 2.1}).}
The relations (\ref{rel 2.1}) are
	$$ \bX^2-\frac{1}{d_4}(\bX\bY+\bY\bX)=-\frac{a_4}{d_4}\mds{1}\quad\text{and}\quad \bY^2=\mds{1}.$$
Summing together the first relation and the second relation multiplied by $\frac{1}{d_4^2}$ we get
	\begin{equation} \label{eq-1111} 
		\frac{1}{d_4^2} \bY^2-\frac{1}{d_4}(\bX\bY+\bY\bX) + \bX^2 = 
			\big(\frac{1}{d_4^2}-\frac{a_4}{d_4}\big)\mds{1}.
	\end{equation}
Now we rewrite (\ref{eq-1111}) in the form
	$$\big(\frac{1}{d_4} \bY-\bX\big)^2= \big(\frac{1}{d_4^2}-\frac{a_4}{d_4}\big)\mds{1}.$$
Applying an affine linear transformation 
		$\phi_7(X,Y)=\big(\frac{1}{d_4} y-X,Y\big)$ to  
	$\breve{\beta}$ we get $\acute\beta$ with $\acute{\mc{M}}(2)$ satisfying
		$$\bX^2 = \big(\frac{1}{d_4^2}-\frac{a_4}{d_4}\big)\mds{1}\quad \text{and}\quad
			\bY^2 = \mds{1}.$$		
	Hence we are in Case 1.\\
	
	\noindent \textbf{Case 2.2: $\bY^2=\mds 1-\bX^2$ in (\ref{second-rel}).} 
	Multiplying the relation $(\ref{first-rel})$ from the left by $\bX$ (resp.\ $\bY$) and comparing  the 	
	rows $\bX\bY$, $\bY\bX$ on both sides using Lemma \ref{symmetry-2-aljaz} we conclude 
	that $c_3=0$ (resp.\ $b_3=0$). Thus the relation of $\overline{\mc{M}}(2)$ are
		$$\bX\bY+\bY\bX = a_3\mds{1}+d_3\bX^2\quad\text{and}\quad \bY^2 + \bX^2 = \mds{1}.$$
	Summing together the first relation and the second relation multiplied by $\alpha$ we get
		\begin{equation} \label{eq-1000} 
			\alpha \bY^2 + (\bX \bY+ \bY \bX) +(\alpha-d_3)\bX^2 =(\alpha+ a_3)\mds 1.
		\end{equation}
	Choosing 
		$$\alpha=\frac{1}{2} \sqrt{4 + d_3^2}+\frac{d_3}{2},$$
	we see that
		\begin{equation*} \label{system-1}
			\alpha>0,\quad \alpha-d_3>0\quad \text{and}\quad \sqrt{(\alpha-d_3)\alpha}=1, 
		\end{equation*}
	and thus (\ref{eq-1000}) can be rewritten in the form
		$$(\sqrt{\alpha- d_3} \bX + \sqrt{\alpha}\bY)^2= (\alpha+a_3)\mds 1.$$
		Applying an affine linear transformation 
		$\phi_8(X,Y)=(X, \sqrt{\alpha- d_3} X + \sqrt{\alpha}Y)$
	to $\overline{\beta}$ we get $\widehat\beta$ with $\widehat{\mc{M}}(2)$ satisfying
		\begin{equation}\label{Y2-case22}
			\bY^2= (\alpha+ a_3)\mds{1}\quad\text{and}\quad \bX\bY+\bY\bX=a_4\mds{1}+d_4\bX^2,
		\end{equation}
	where $a_4,d_4\in \RR$. 
	Since $\widehat{\mc{M}}(2)$  is psd of rank 5, $\alpha+ a_3>0$ and after normalization the relations 
	(\ref{Y2-case22}) become 
		$$\bY^2= \mds{1}\quad\text{and}\quad \bX\bY+\bY\bX=a_5\mds{1}+d_5\bX^2,$$
	where $a_5,d_5\in \RR$. 
	Hence we are in Case 2.1.\\

\noindent \textbf{Case 2.3: $\bY^2=\bX^2$ in (\ref{second-rel}).} As in the first paragraph of Case 2.2
	we conclude that the relations of $\overline{\mc{M}}(2)$ are
		$$\bX\bY+\bY\bX = a_3\mds{1}+d_3\bX^2\quad\text{and}\quad \bY^2=\bX^2.$$
	Applying an affine linear transformation $\phi_9(X,Y)=(X+Y,Y-X)$ to  $\widetilde{\beta}$ 
	we get $\overline\beta$ with $\overline{\mc{M}}(2)$ satisfying
		$$(2-d_3)\bX^2-(2+d_3)\bY^2=4 a_3 \mds 1\quad \text{and}\quad
			\bX\bY+\bY\bX=\mbf 0,\quad .$$
	If $d_3=2$, then after normalization we come into Case 2.1. 
	If $d_3=-2$, then we come into Case 2.1 after we 
	apply a transformation $(X,Y)\mapsto (Y,X)$ to change the roles of $\bX$ and $\bY$ and normalize.
	Otherwise we apply an affine linear transformation 
		$$\phi_{10}(X,Y)=(\sqrt{|2-d_3|}X,\sqrt{|2+d_3|}Y)$$
	to $\widetilde{\beta}$ and get $\breve\beta$ with 
	$\breve{\mc{M}}(2)$ satisfying
		$$\bX\bY+\bY\bX=\mbf 0$$
	and one of the following:
		\begin{equation}\label{rel-3-pos-2-equiv}
			\bX^2+ \bY^2=4 a_3 \mds{1}\quad \text{or}\quad 
			\bX^2- \bY^2=4 a_3 \mds{1}\quad \text{or} \quad -\bX^2- \bY^2= 4a_3 \mds{1}.
		\end{equation}
	The first and the last cases are equivalent, since the third relation can be rewritten as 
	$\bX^2+\bY^2=-4a_3\mds 1$. Thus we separate two possibilities 
	in (\ref{rel-3-pos-2-equiv}).\\
	
	\noindent \textbf{Case 2.3.1: $\bX^2+ \bY^2=4a_3 \mds{1}$ in (\ref{rel-3-pos-2-equiv}).} 
	It is easy to see that $a_3>0$ (by $\breve{\mc{M}}(2)$ being psd of rank 5, since otherwise 
	$\beta_{Y^2}=\beta_{X^2Y^2}=\beta_{Y^4}=0$).
	Thus after the normalization we are in the basic case 1 of Proposition \ref{structure-of-rank5-2}.\\
	
	\noindent \textbf{Case 2.3.2: $\bX^2- \bY^2=4a_3 \mds{1}$ in (\ref{rel-3-pos-2-equiv}).} 
	We may assume that $a_3\leq 0$ (otherwise we change the roles of $\bX$ and $\bY$). If
	$a_3<0$, then after normalization we come into the basic case 3. 
	Otherwise $a_3=0$ and we are in the basic case 4.\\

\noindent \textbf{Case 2.4: $\bY^2=\mds 1+\bX^2$ in (\ref{second-rel}).}
	As in the first paragraph of Case 2.2
	we conclude that the relations of $\overline{\mc{M}}(2)$ are
		$$\bX\bY+\bY\bX = a_3\mds{1}+d_3\bX^2\quad\text{and}\quad \bY^2=\mds 1+\bX^2,$$
	and after applying an affine linear transformation 
		$\phi_9(X,Y)=(X+Y,Y-X)$
	to $\overline{\beta}$ to get $\breve\beta$ with $\breve{\mc{M}}(2)$ 
	satisfying
		$$(2-d_3)\bX^2-(2+d_3)\bY^2=
			(4a_ 3-2d_{3}) \mds{1}\quad \text{and}\quad \bX\bY+\bY\bX=2\cdot \mds{1}.$$
	If $d_3=2$, then after normalization we come into Case 2.1. If $d_3=-2$ then we come into Case 2.1  after we 
	apply a transformation $(X,Y)\mapsto (Y,X)$ to change the roles of $\bX$ and $\bY$ and normalize.
	Otherwise we apply an affine linear transformation 
		$$\phi_{11}(X,Y)=(\sqrt{|2-d_3|}X,\sqrt{|2+d_3|}Y)$$ 
	to $\breve{\beta}$ and get $\acute\beta$ with $\acute{\mc{M}}(2)$ satisfying
		$$\bX\bY+\bY\bX=2\sqrt{|(4-d_3^2|}\mds{1}$$
	and one of the following
		\begin{equation}\label{rel-3-pos-2-equiv-2}
			\bX^2+ \bY^2=\tilde a \mds{1}\quad \text{or}\quad \quad \bX^2- \bY^2=\tilde a \mds{1}\quad 				\text{or}\quad -\bX^2- \bY^2=\tilde a \mds{1},
		\end{equation}
	where $\tilde a=4a_3-2d_3$. 
	The first and the last cases are equivalent, since the third relation can be rewritten as
	$\bX^2+\bY^2=-\tilde a\mds 1$. Thus we separate two
	possibilities in (\ref{rel-3-pos-2-equiv-2}).\\
	
\noindent \textbf{Case 2.4.1: $\bX^2+ \bY^2=\tilde a \mds{1}$.}
	It is easy to see that $\tilde a>0$ (by $\acute{\mc{M}}(2)$ being psd of rank 5, since otherwise 
	$\beta_{Y^2}=\beta_{X^2Y^2}=\beta_{Y^4}=0$).
	Hence after normalization we come into Case 2.2.\\
	
\noindent \textbf{Case 2.4.2:  $\bY^2- \bX^2=\tilde a \mds{1}$.}
	We may assume that $\tilde a\geq 0$ (otherwise we change the roles of $\bX$ and $\bY$). 
	If $\tilde a=0$, we are in Case 2.3.
	Otherwise we apply a transformation 
		$$\phi_{12}(X,Y)=\Big(X,X-\frac{2\sqrt{|(4-d_3^2|}}{\tilde a}Y\Big)$$
	to $\acute{\beta}$ and get $\widehat \beta$ with $\widehat{\mc{M}}(2)$ satisfying
		$$\bY^2+\Big(1-\frac{4(4-d_3^2)^2}{\tilde a^2}\Big)\bX^2=\mbf 0\quad \text{and}\quad
		\bX\bY+\bY\bX=-\tilde a \mds{1} +\tilde a\bX^2.$$
	It is easy to see that $1-\frac{4(4-d_3^2)^2}{\tilde a^2}<0$ (by $\hat{\mc{M}}(2)$ being psd of rank 5, since otherwise $\beta_{Y^4}=\beta_{X^2Y^2}=\beta_{Y^2}=\beta_{X^2}=0$) and
	after a further normalization of $\bX$ 
	 the relations of the corresponding matrix $\widehat{\mc{M}}(2)$ become
		$$\bY^2-\bX^2=\mbf 0\quad\text{and}\quad 
		\bX\bY+\bY\bX=-\hat a \mds{1} -\hat a \bX^2,\quad \text{for some } \hat a\in \RR.$$
	Hence we come into Case 2.3. \\
	
\noindent{ \textbf{Case 3:  The set $\left\{\mds{1},\bX,\bY,\bY^2,\bY\bX\right\}$ is the basis for 
				$\mathcal C_{\mc{M}_2}$.}}\\

Applying an affine linear transformation $(X,Y)\mapsto (Y,X)$ we come into Case 2. 
\end{proof}

Now we prove Proposition \ref{structure-of-rank5-2} (\ref{point-2-str-rank5}).

\begin{proof}[Proof of Proposition \ref{structure-of-rank5-2} (\ref{point-2-str-rank5})]
	By Lemma \ref{blue-2-pos-2} we have to consider 2 different cases.\\

	\noindent\textbf{Case 1: The set $\left\{\mds{1},\bX,\bY,\bX^2,\bX\bY,\bY\bX\right\}$ is the basis for
			$\mathcal{C}_{\mc{M}_2}$.}\\
			
	By assumption there are constants $a_i$, $i=1,\ldots, 6$, such that
		\begin{equation*}\label{relation-r6}
			\bY^2=a_1\mds{1}+a_2\bX+a_3\bY+a_4\bX^2+a_5\bX\bY+a_6\bY\bX.
		\end{equation*}
	By Lemma \ref{Y^2-blueuction-lemma} the statement of Proposition \ref{structure-of-rank5-2}
	follows.\\

	\noindent\textbf{Case 2: The set $\left\{\mds{1},\bX,\bY,\bX^2,\bX\bY,\bY^2\right\}$ is the basis for
			$\mathcal{C}_{\mc{M}_2}$.}\\

	By assumption there are constants $a_i$, $i=1,\ldots, 6$, such that
		\begin{equation}\label{relation-r6-4-2}
			\bY\bX=a_1\mds{1}+a_2\bX+a_3\bY+a_4\bX^2+a_5\bX\bY+a_6\bY^2.
		\end{equation}
	By comparing the rows $\bX\bY$, $\bY \bX$ of the both sides of equation we conclude that 
	$a_5=-1$.
	We separate two cases.\\
	
	\noindent \textbf{Case 2.1: $a_4\neq 0$ or $a_6\neq 0$. }
	By symmetry we may assume that $a_6\neq 0$. We rewrite the relation 
	(\ref{relation-r6-4-2}) as
		\begin{equation*}\label{relation-r6-5}
			\bY^2=-\frac{a_1}{a_6}\mds{1}-\frac{a_2}{a_6}\bX-\frac{a_3}{a_6}\bY 
				-\frac{a_4}{a_6}\bX^2-\frac{a_5}{a_6}\bX\bY+\frac{1}{a_6}\bY\bX.
		\end{equation*}
	By Lemma \ref{Y^2-blueuction-lemma} the statement of Proposition \ref{structure-of-rank5-2}
	follows.\\
	
	\noindent \textbf{Case 2.2:} $a_4=a_6=0$.
	 We rewrite the relation 
	(\ref{relation-r6-4-2}) as
		\begin{equation*}
			(\bX+\bY)\bY+\bY(\bX+\bY)-2\bY^2=a_1\mds{1}+a_2(\bX+\bY)+
				(a_3-a_2) \bY.
		\end{equation*}
	Applying an affine linear transformation $\phi_1(X,Y) = (X+Y,Y)$ to $\beta$ we get $\widetilde\beta$ with 
	$\widetilde{\mc{M}}(2)$ satisfying
		\begin{equation*}
			\bX\bY+\bY\bX-2\bY^2=a_1\mds{1}+a_2 \bX+
				(a_3-a_2) \bY.
		\end{equation*}
	By Lemma \ref{Y^2-blueuction-lemma} the statement of Proposition 
	\ref{structure-of-rank5-2} (2)
	follows.
\end{proof}


\section{Atoms in the minimal measure of ranks 5 and 6}\label{atoms-minim}

In this section we show that every nc sequence $\beta\equiv \beta^{(4)}$ which admits a nc measure with $\mc{M}_2$  in one of the basic cases of rank 5 or one of the first three basic cases of rank 6 given by Proposition \ref{structure-of-rank5-2}, admits a minimal measure with all the atoms of special form; see Proposition \ref{anticommute} below. This form will be crucial in the subsequent sections where we will analyze each basic case separately to show that the atoms of size 2 are sufficient.

\begin{proposition} \label{anticommute}
	Suppose a nc sequence $\beta\equiv \beta^{(4)}$ has a moment matrix $\mc{M}_2$ 
	satisfying one of the column relations
		\begin{equation}\label{rel-rank6}
			\bX\bY+\bY\bX=\mbf 0\quad \text{or}\quad
			\bY^2=\mds 1-\bX^2 \quad\text{or} \quad 
			\bY^2=\mds 1+\bX^2.
		\end{equation}
	If $\beta$ admits a nc measure, then the atoms are of the following two forms:
	\begin{enumerate}
		\item $(x_i,y_i)\in\RR^2.$
		\item $(X_i,Y_i)\in (\mathbb{SR}^{2t_i\times 2t_i})^2$ for some $t_i\in \NN$ such that
			\begin{equation*} 
				X_i=\begin{pmatrix} \gamma_i I_{t_i} & B_i \\ B_i^t & -\gamma_i I_{t_i} 
					\end{pmatrix}			
				\quad  \text{and} \quad 
				Y_i=\begin{pmatrix} \mu_i I_{t_i} & \mbf 0 \\ \mbf 0 & -\mu_i I_{t_i} \end{pmatrix}
			\end{equation*}
	\end{enumerate}
	 where $\gamma_i\geq 0$, $\mu_i>0$ and $B_i$ are $t_i\times t_i$ matrices.
\end{proposition}

\begin{proof}
	Suppose $\mu$ is any nc measure representing $\beta$. By Theorem \ref{support lemma} every 
	atom $(X_i,Y_i)$ in $\mu$ satisfies the relation (\ref{rel-rank6}).\\

	\noindent \textbf{Claim 1:} We may assume that $X_iY_i+Y_iX_i$ and $Y_i$ are diagonal matrices.\\

	Observe that $X_iY_i+Y_iX_i$ is symmetric and commutes with $Y_i$. Therefore after a orthogonal transformation we may assume that
	$X_iY_i+Y_iX_i$ and $Y_i$ are diagonal matrices.\\

	\noindent \textbf{Claim 2:} We may assume that the atoms $(X_i,Y_i)$ of size greater than 1
	 are of the forms 
		\begin{equation} \label{form-claim2}
			X_i=\begin{pmatrix} D_{i1} & B_i \\ B_i^t & D_{i2}\end{pmatrix}\quad\text{and}\quad
			Y_i=\begin{pmatrix} \mu_i I_{n_{i1}} & \mbf 0 \\ \mbf 0 & -\mu_i I_{n_{i2}}\end{pmatrix},
		\end{equation}
	where $\mu_i>0$, $n_{i1},n_{i2}\in \NN$, $D_{i1}\in \RR^{n_{i1}\times n_{i1}}$ and $D_{i2}\in \RR^{n_{i2}\times n_{i2}}$ are diagonal matrices and
	$B_i\in \RR^{n_{i1}\times n_{i2}}$. \\

	By an appropriate permutation we may assume that $Y_i$ is of the form
			$$Y_i=\bigoplus_{j=1}^{\ell_i} 
				\begin{mpmatrix} 
					\mu_j^{(i)} I_{n_{ij}} & \mbf 0 \\
					\mbf 0 & -\mu_{j}^{(i)} I_{m_{ij}} 
				 \end{mpmatrix} \bigoplus \mbf 0_{m\times m},$$			
	\noindent where $\ell_i,  n_{ij},m_{ij}, m\in \NN\cup\{0\}$, $\mu_j^{(i)}>0$ and $\mu_{j_1}^{(i)}\neq \mu_{j_2}^{(i)}$ for 
	$j_1\neq j_2$.
	Let
			$$X_i=(X_{pr}^{(i)})_{pr}$$
	be the corresponding block decomposition of $X_i$. Since $X_i Y_i+Y_iX_i$ is diagonal,
	it follows that 
	\begin{enumerate}
		\item for $1\leq p,r \leq \ell_i$ and $p\neq r$ we have that
			\begin{align*}
		[X_iY_i+Y_iX_i]_{2p-1,2r-1} =  (\mu_p^{(i)}+\mu_r^{(i)}) X_{2p-1,2r-1}^{(i)}=\mbf 0& \quad 
			\Rightarrow \quad X_{2p-1,2r-1}^{(i)}=\mbf{0},\\
		[X_iY_i+Y_iX_i]_{2p-1,2r}    =  (\mu_p^{(i)}-\mu_r^{(i)}) X_{2p-1,2r}^{(i)}=\mbf 0& \quad 
			 \Rightarrow \quad X_{2p-1,2r}^{(i)}=\mbf{0},\\
		[X_iY_i+Y_iX_i]_{2p,2r}	=  -(\mu_p^{(i)}+\mu_r^{(i)}) X_{2p,2r}^{(i)}=\mbf 0& \quad 
			\Rightarrow \quad X_{2p,2r}^{(i)}=\mbf{0}.
			\end{align*}
		\item for $1\leq p\leq \ell_i$ we have that
			\begin{align*}
			[X_iY_i+Y_iX_i]_{2p-1,2\ell_i+1}	=  \mu_p^{(i)} X_{2p-1,2\ell_i+1}^{(i)}=\mbf 0& 
				\quad \Rightarrow \quad X_{2p-1,2\ell_i+1}^{(i)}=\mbf{0},\\
			[X_iY_i+Y_iX_i]_{2p,2\ell_i+1}	=  -\mu_p^{(i)} X_{2p,2\ell_i+1}^{(i)}=\mbf 0& 
				\quad \Rightarrow \quad X_{2p,2\ell_i+1}^{(i)}=\mbf{0},\\
			[X_iY_i+Y_iX_i]_{2\ell_i+1,2p-1}	=  \mu_p^{(i)} X_{2\ell_i+1,2p-1}^{(i)}=\mbf 0& 
				\quad \Rightarrow \quad X_{2\ell_i+1,2p-1}^{(i)}=\mbf{0},\\
			[X_iY_i+Y_iX_i]_{2\ell_i+1,2p}	=  -\mu_p^{(i)} X_{2\ell_i+1,2p}^{(i)}=\mbf 0& 
				\quad \Rightarrow \quad X_{2\ell_i+1,2p}^{(i)}=\mbf{0}.
			\end{align*}
		\item for $1\leq p=r\leq \ell_i$ we have that
			\begin{align*}
			[X_iY_i+Y_iX_i]_{2p-1,2p-1}	=  2\mu_p^{(i)} X_{2p-1,2p-1}^{(i)} \text{ is diagonal}& \quad \Rightarrow\quad X_{2p-1,2p-1}^{(i)} \text{ is diagonal},\\
			[X_iY_i+Y_iX_i]_{2p,2p}	=  -2\mu_p^{(i)} X_{2p,2p}^{(i)} \text{ is diagonal}& \quad 
\Rightarrow\quad X_{2p,2p}^{(i)} \text{ is diagonal}. 
			\end{align*}
	\end{enumerate}
	So $X_i$ is of the form
		$$X_i=\bigoplus_{j=1}^{\ell_i} 
			 \begin{mpmatrix} 
				X_{11}^{(ij)} & X_{12}^{(ij)} \\
				(X_{12}^{(ij)})^t & X_{22}^{(ij)} 
				\end{mpmatrix} \bigoplus X^{(i)}_{\ell_i+1}.$$
	Thus we can replace the atom $(X_i,Y_i)$ with the atoms of the form
		\begin{equation}\label{nc-part}
			\widetilde{X}_{ij}=
				\begin{pmatrix} 
					X^{(ij)}_{11}& X^{(ij)}_{12} \\ 
					(X^{(ij)}_{12})^t & X^{(kij)}_{22} 
				\end{pmatrix}\quad\text{and}\quad
		    	\widetilde{Y}_{ij}=
			    	\begin{pmatrix} 
					\mu_{j}^{(i)} I_{n_{ij}}& \mbf 0 \\ \mbf 0 & -\mu_j^{(i)} I_{m_{ij}} 
				\end{pmatrix},
		\end{equation}
	or
		\begin{equation}\label{cm-part}
			\widetilde{X}_{ij}=X^{(i)}_{\ell_i+1}
			\quad\text{and}\quad
		    	\widetilde{Y}_{ij}=\mbf 0.
		\end{equation}
	By orthogonal transformation the atom (\ref{cm-part}) can be replaced by the atom
		$$\widehat{X}_{ij}=D^{(i)}_{\ell_i+1}
			\quad\text{and}\quad
		    	\widetilde{Y}_{ij}=\mbf 0,$$
	where $D^{(i)}_{\ell_i+1}$ is a diagonal matrix and further on by atoms of size 1 of the form
		$(x,0),$
	where $x$ runs over the diagonal of $D^{(i)}_{\ell_i+1}$. Hence we may assume 
	that the atoms of size greater than 1 in the representing measure for $\beta$
	are of the form (\ref{nc-part}). Further on, by appropriate orthogonal transformation we may 
	assume that they are of the form (\ref{form-claim2}). This proves the claim.\\

	\noindent \textbf{Claim 3:} We may assume that the atoms $(X_i,Y_i)$ of size greater than 1
	are of the forms
		\begin{equation*} 
				X_i=\begin{pmatrix} \gamma_i I_{t_i} & B_i \\ B_i^t & -\gamma_i I_{t_i} 
					\end{pmatrix}			
				\quad  \text{and} \quad 
				Y_i=\begin{pmatrix} \mu_i I_{t_i} & \mbf 0 \\ \mbf 0 & -\mu_i I_{t_i} \end{pmatrix},
			\end{equation*}
	 where $\gamma_i\geq 0$, $\mu_i>0$ and $B_i$ are $t_i\times t_i$ matrices for some
	$ t_i\in \NN$.\\
	 
	 First we prove Claim 3 in case we have $\bX\bY+\bY\bX=\mbf 0$ in (\ref{rel-rank6}). 
	Let us prove that we may assume invertibility of $X_i$. After applying an orthogonal transformation
	to $(X_i,Y_i)$ we have
	 $X_i=\begin{pmatrix} \mbf 0 & \mbf 0 \\ \mbf 0 & \widehat X_i\end{pmatrix}$
	 where $\widehat X$ is invertible and 
	 $Y_i=\begin{pmatrix} Y_{i1} & Y_{i2} \\ Y_{i2}^t & Y_{i3} \end{pmatrix}$.
	 From $X_iY_i+Y_iX_i=\mbf 0$ it follows that $Y_{i2}\widehat X_i=\mbf 0$. Since $\widehat X_i$
	 is invertible, $Y_{i2}=\mbf 0$. Hence we can replace 
	 the atom $(X_i,Y_i)$ with the atoms $(\mbf 0, Y_{i1})$ and $(\widehat X_i,Y_{i3})$.
	 Since the atom $(\mbf 0, Y_{i1})$ can be further replaced with the atoms of size 1, we may assume the $X_i$ is invertible.
	 
	 Observe that
	 in (3) from the proof of Claim 2 we have
	 	\begin{align*}
			\mbf 0=[X_iY_i+Y_iX_i]_{2p-1,2p-1}=2\mu_p^{(i)} X^{(i)}_{2p-1,2p-1}
				&\quad\Rightarrow\quad X^{(i)}_{2p-1,2p-1}=\mbf 0,\\
			\mbf 0=[X_iY_i+Y_iX_i]_{2p,2p}=-2\mu_p^{(i)} X^{(i)}_{2p,2p}
				&\quad\Rightarrow\quad X^{(i)}_{2p,2p}=\mbf 0.
		\end{align*}
	Therefore $X_i$ in (\ref{form-claim2}) is of the form
		$X_i=\begin{pmatrix} \mbf 0 & B_i \\ B_i^t & \mbf 0\end{pmatrix}$
	with $B_i\in \RR^{n_{i1}\times n_{i2}}$ and $n_{i1}=n_{i2}$ by the invertibility of $X_i$.
	This proves Claim 3 in case we have $\bX\bY+\bY\bX=\mbf 0$ in (\ref{rel-rank6}).\\
	
	It remains to prove Claim 3 in case we have $\bY^2=\mds 1\pm \bX^2$ in (\ref{rel-rank6}).
	 By Claim 2 and after an appropriate permutation we may assume that $X_i$, $Y_i$
	 are of the form (\ref{form-claim2}) with
	 	\begin{equation*}
			D_{i1}= \bigoplus_{j=1}^{p_i} \lambda_j^{(i)}I_{s_{ij}}
			\quad\text{and}\quad
			D_{i2}= \bigoplus_{j=1}^{r_i} \gamma_j^{(i)}I_{v_{ij}}, 
		\end{equation*}
	where $p_i,s_{ij}, r_i, v_{ij} \in \NN$ and 
		$$\lambda_1^{(i)}>\lambda_{2}^{(i)}>\ldots>\lambda^{(i)}_{p_i}\quad \text{and}\quad
			\gamma_1^{(i)}>\gamma_{2}^{(i)}>\ldots>\gamma^{(i)}_{r_i}.$$
	Let 
		$$B_i=(B^{(i)}_{pr})_{pr}$$
	be the corresponding block decomposition of $B_i$,
	where 
		$$B^{(i)}_{pr}\in \RR^{s_{ip}\times v_{ir}}$$ 
	for $p=1,\ldots, p_i$, $r=1,\ldots,r_i$.
	Calculating $X_i^2$ we get that
		$$X_i^2=\begin{pmatrix}  D_{i1}^2 + B_i B_i^t & 
							D_{i1}B_i+B_i^tD_{i2} \\
							B_i^t D_{i1}+D_{i2}B_i& 
							B_i^t B_i + D_{i2}^2 \end{pmatrix}.$$
	Since $X_i^2$	is a diagonal matrix, we conclude that
		$$D_{i1}B_i+B_i^tD_{i2}=\mbf 0.$$
	Thus
		$$[D_{i1}B_i+B_i^tD_{i2}]_{pr}=
			(\lambda_{p}^{(i)}+\gamma_{r}^{(i)}) B^{(i)}_{pr}=\mbf{0},$$
 	for $1\leq p\leq p_i$, $1\leq r\leq r_i.$
	We conclude that
		$$\lambda_{p}^{(i)}=-\gamma_{r}^{(i)} \quad \text{or} \quad B^{(i)}_{pr}=\mbf{0}.$$
	So in every row and every column in the block decomposition of $B_i$ at most one block
	$B^{(i)}_{pr}$ is possibly nonzero, i.e., $B^{(i)}_{pr}$ may be nonzero if and only if 
	$\lambda_{p}^{(i)}=-\gamma_{r}^{(i)}$
	So after a suitable permutation $X_i$ has the following block decomposition
		\begin{eqnarray*}
			X_i&=&
		\bigoplus_{\substack{ 1\leq p\leq p_i \\ 1\leq r\leq r_i\\ 
				\lambda_p^{(i)}+\gamma_{r}^{(i)}=0}}
					\begin{pmatrix}  \lambda_p^{(i)} I_{s_{ip}} & B^{(i)}_{pr}\\ 
						(B^{(i)}_{pr})^{t} & \gamma_{r}^{(i)} I_{v_{ir}}\end{pmatrix} 
				\bigoplus
				\bigoplus_{\substack{1\leq p\leq p_k\\
						\lambda_p^{(i)}\neq -\gamma_{r}^{(i)} \;\forall r}} 
					\begin{pmatrix}  \lambda_p^{(i)} I_{s_{ip}} \end{pmatrix} \\
		&&\bigoplus
				\bigoplus_{\substack{ 1\leq r\leq r_k  \\ 
						\lambda_p^{(i)}\neq -\gamma_{r}^{(i)} \;\forall p}}
					\begin{pmatrix}   \gamma_{r}^{(i)} I_{v_{ir}} \end{pmatrix}. 
		\end{eqnarray*}
	The corresponding block decomposition of 
		$Y_i$ is of the form
				\begin{eqnarray*}
				Y_i&=&
					\bigoplus_{\substack{ 1\leq p\leq p_i \\ 1\leq r\leq r_i\\ 
					\lambda_p^{(i)}+\gamma_{r}^{(i)}=0}}					
					\begin{pmatrix}  \mu_i I_{s_{ip}} & 0 \\ 0 & -\mu_i  I_{v_{ir}}\end{pmatrix} 
				\bigoplus
					\bigoplus_{\substack{1\leq p\leq p_k\\
						\lambda_p^{(i)}\neq -\gamma_{r}^{(i)} \;\forall r}}
					\begin{pmatrix}  \mu_i I_{s_{ip}} \end{pmatrix} \\
				&&\bigoplus
				 	\bigoplus_{\substack{ 1\leq r\leq r_k  \\ 
						\lambda_p^{(i)}\neq -\gamma_{r}^{(i)} \;\forall p}}
					\begin{pmatrix}   -\mu_i I_{v_{ir}} \end{pmatrix}. 
				\end{eqnarray*}
	Thus we can replace the atom $(X_i,Y_i)$ with the atoms of the form
		\begin{equation}\label{nc-part-2}
			\widetilde{X}_{ij}=
				\begin{pmatrix}  \lambda_p^{(i)} I_{s_{ip}} & B_{pr}^{(i)}\\ 
					(B_{pr}^{(i)})^{t} & - \lambda_p^{(i)}  I_{v_{ir}}\end{pmatrix} 
				\quad\text{and}\quad
		    	\widetilde{Y}_{ij}=
				\begin{pmatrix}  \mu_i I_{s_{ip}} & 0 \\ 0 & -\mu_i  I_{v_{ir}}\end{pmatrix} 
		\end{equation}
	or
		\begin{equation*}
			\widetilde{X}_{ij}=\lambda_p^{(i)}
			\quad\text{and}\quad
		    	\widetilde{Y}_{ij}=\mu_i
		\end{equation*}			
	or
		\begin{equation*}
			\widetilde{X}_{ij}=\gamma_r^{(i)}
			\quad\text{and}\quad
		    	\widetilde{Y}_{ij}=-\mu_i.
		\end{equation*}		
	Hence we may assume 
	that the atoms $(X_i,Y_i)$ of size greater than 1 in the representing measure for $\mc{M}_2$
	are of the form (\ref{nc-part-2}). 	
	Now
		$$X_{i}^2=
			\begin{pmatrix}
				(\lambda_{p}^{(i)})^2 I_{s_{ip}}+B_{pr}^{(i)}(B_{pr}^{(i)})^t & \mbf 0\\
				\mbf 0 & (B_{pr}^{(i)})^t B_{pr}^{(i)} + (\lambda_{p}^{(i)})^2 I_{v_{ir}}
			\end{pmatrix}$$
	Since 
		$$X_{i}^2=\mds 1 \pm Y_{i}^2=
			\begin{pmatrix}  
				(1\pm \mu_i^2)I_{s_{ip}} & \mbf 0\\
				\mbf 0 & (1\pm \mu_i^2) I_{v_{ir}}
			\end{pmatrix},$$
	it follows that
		\begin{eqnarray}
			B_{pr}^{(i)}(B_{pr}^{(i)})^t
			&=& (1\pm \mu_i^2-(\lambda_{p}^{(i)})^2)I_{s_{ip}} \label{equality1}\\
			(B_{pr}^{(i)})^t B_{pr}^{(i)}
			&=&(1\pm \mu_i^2-(\lambda_{p}^{(i)})^2)I_{v_{ir}}. \label{equality2}
		\end{eqnarray}
	We separate two cases according to the value of $1\pm \mu_i^2-(\lambda_{p}^{(i)})$.\\
	
	\noindent \textbf{Case 1: $1\pm \mu_i^2-(\lambda_{p}^{(i)})=0$.}\\
	
	It follows that $B_{pr}^{(i)}=\mbf 0$. Then $X_{i}$ is diagonal and commutes with 
	$Y_{i}$. Therefore the atom $(X_i,Y_i)$ can be replaced 	
	by the atoms $(\lambda_p^{(i)},\mu_i)$ and $(-\lambda_p^{(i)},-\mu_i)$.\\
	
	\noindent \textbf{Case 2: $1\pm \mu_i^2-(\lambda_{p}^{(i)})\neq 0$.}\\
	
	From (\ref{equality1}) and (\ref{equality2}) it follows that
		\begin{eqnarray}
			s_{ip} &=& \Rank(B_{pr}^{(i)}(B_{pr}^{(i)})^t)\leq \min (\Rank(B_{pr}^{(i)}),\Rank((B_{pr}^{(i)})^t))\leq
					\min(s_{ip},v_{ir}) \label{ineq-1}\\
			v_{ir} &=& \Rank((B_{pr}^{(i)})^tB_{pr}^{(i)})\leq \min (\Rank((B_{pr}^{(i)})^t),\Rank(B_{pr}^{(i)}))\leq
					\min(v_{ir},s_{ip}). \label{ineq-2}
		\end{eqnarray}
	It follows from (\ref{ineq-1}) and (\ref{ineq-2}) that
		$s_{ip}=v_{ir}$
	in (\ref{nc-part-2})
	which proves Claim 3 and concludes the proof of Proposition \ref{anticommute}.
\end{proof}


\section{Solution of the BQTMP for $\mc{M}_2$ of rank 5} \label{rank5-section}

In this section we solve the BQTMP for $\mc{M}_2$ of rank 5. By Proposition \ref{structure-of-rank5-2} it suffices to solve four basic cases. In Subsections
\ref{subsec-1}, \ref{subsec-2}, \ref{subsec-3}, \ref{subsec-4} we study these cases separately. 
We characterize exactly when $\mc{M}_2$ admits a nc measure, see Theorems \ref{M(2)-XY+YX=0-bc1}, \ref{M(2)-XY+YX=0}, \ref{M(2)-XY+YX=0-bc3} and  \ref{M(2)-XY+YX=0-bc4}. 
Moreover, we characterize type and uniqueness of the minimal measures. In particular, the minimal measure is almost always unique (up to orthogonal equivalence), 
except in one subcase for which there are two minimal measures, and there is always exactly one atom from $(\mathbb{SR}^{2\times 2})^2$ in the minimal measure and up to three atoms from $\RR^2$.

Let $(X,Y)\in (\mbb{SR}^{t\times t})^{2}$ where $t\in \NN$. We denote by
$\mc{M}^{(X,Y)}_2$ the moment matrix generated by $(X,Y)$, i.e.,
$\beta_{w(X,Y)}=\Tr(w(X,Y))$ for every $|w(X,Y)|\leq 4$.

 The following proposition will be used in all four basic cases to prove that if $\beta$ admits a nc measure, then it has a representing measure with the atoms of size at most 2.

\begin{proposition}\label{atoms-of-size-2}
		Let us fix a pair $(R_1,R_2)$ of the basic case relations given by Proposition
	\ref{structure-of-rank5-2} (\ref{point-1-str-rank5}).	
	If every sequence $\beta\equiv \beta^{(4)}$ with $\beta_X=\beta_Y=\beta_{X^3}=0$ and 
	a moment matrix $\mc{M}_2(\beta)$ of rank 5 with column relations $R_1$ and $R_2$, 
	admits a nc measure with atoms of size at most 2, then 
	every nc sequence $\widetilde\beta\equiv \widetilde\beta^{(4)}$ which admits a nc measure and
	has a moment matrix $\widetilde{\mc{M}}_2$ of rank 5
	with column relations	 $R_1$ and $R_2$,
	admits a nc measure with atoms of size at most 2.
\end{proposition}

\begin{proof}
	Suppose $\widetilde\beta$ admit a nc measure and
	has a moment matrix $\widetilde{\mc{M}}_2$ of rank 5
	with column relations	 $R_1$ and $R_2$.
	By Proposition \ref{anticommute} we may assume 
	that all the atoms $(X_i,Y_i)\in (\mbb{SR}^{u_i\times u_i})^{2}$ of size $u_i>1$
	are of the form 
		$$X_{i}=\begin{pmatrix} \gamma_i I_{t_i} & B_i \\ B_i^t & -\gamma_i I_{t_i} 
			\end{pmatrix},\quad 
		Y_{i}=\begin{pmatrix} \mu_i I_{t_i} & \mbf 0 \\ \mbf 0 & -\mu_i I_{t_i} \end{pmatrix},$$
	where $\gamma_i\geq 0$, $\mu_i>0$ and $B_i$ are $t_i\times t_i$ matrices.
	Calculating $X_i^3$ we get
		$$X_{i}^3=\begin{pmatrix} 
				\gamma_i(\gamma_i^2 I_{t_i}+B_iB_i^t) & 
				(\gamma_i^2 I_{t_i}+B_iB_i^t)B_i \\ 
				(\gamma_i^2 I_{t_i}+B_i^tB_i)B_i^t & 
				-\gamma_i(\gamma_i^2 I_{t_i}+B_i^tB_i)
			\end{pmatrix}.$$
	Therefore $\mc{M}^{(X_i,Y_i)}_2$ satisfies
	$\beta_X=\beta_{Y}=\beta_{X^3}=0$. By assumption the atom $(X_i,Y_i)$ can be replaced by the atoms of 	size at most 2. 
\end{proof}

\subsection{Pair $\bX\bY+\bY\bX=\mbf 0$ and $\bY^2+\bX^2=\mds{1}$.}  \label{subsec-1}
In this subsection we study a nc sequence $\beta\equiv\beta^{(4)}$ with a moment matrix $\mc{M}_2$ of rank 5 satisfying the relations $\bX\bY+\bY\bX=\mbf 0$ and $\bY^2+\bX^2=\mds{1}$. 
In Theorem \ref{M(2)-XY+YX=0-bc1} we characterize exactly when $\beta$ admits a measure. Moreover, we classify type and uniqueness of the minimal measure. 

The form of $\mc{M}_2$ is given by the
following proposition.

\begin{proposition}
		Suppose $\beta\equiv \beta^{(4)}$ is a nc sequence with a moment matrix $\mc{M}_2$ of rank 5 satisfying the relations 
			\begin{equation} \label{relation-bc1} 
				\bX\bY+\bY\bX=\mbf 0\quad \text{and}\quad \bX^2+\bY^2=\mds{1}.
			\end{equation}
		Then $\mc{M}_2$ is of the form
			\begin{equation} \label{matrix-bc1}
				\begin{mpmatrix}
					 \beta_1 & \beta_{X} & \beta_Y & \beta_{X^2} & 0 & 0 &  \beta_1-\beta_{X^2}			\\
					\beta_{X} & \beta_{X^2} & 0 & \beta_{X} & 0 & 0 & 0 					\\
					\beta_Y & 0 & \beta_1-\beta_{X^2} & 0 & 0 & 0 & \beta_Y					\\
					\beta_{X^2} & \beta_{X} & 0 & \beta_{X^4} & 0 & 0 & \beta_{X^2} -\beta_{X^4}	\\
					0 & 0 & 0 & 0 & \beta_{X^2}-\beta_{X^4} &
						 -\beta_{X^2}+\beta_{X^4} & 0					\\
					0 & 0 & 0 & 0 & -\beta_{X^2}+\beta_{X^4} & 
						\beta_{X^2}-\beta_{X^4} & 0 					\\
					\beta_{1}-\beta_{X^2} & 0 & \beta_{Y} & \beta_{X^2}-\beta_{X^4} & 0 & 
						0 & \beta_1-2\beta_{X^2}+\beta_{X^4}
				\end{mpmatrix}.
			\end{equation}
\end{proposition}

\begin{proof}
	The relations (\ref{relation-bc1}) give us the following system in $\mc{M}_2$
	\begin{multicols}{3}
		\begin{equation}\label{eq-bc1-r5}
			\begin{aligned}
				2\beta_{XY}=0,\\
				2\beta_{X^{2}Y}=0,\\
				2\beta_{XY^{2}}=0,\\
				2\beta_{X^{3}Y}=0,
			\end{aligned}
		\end{equation}
	\hfill
	\columnbreak
		\begin{equation*}
			\begin{aligned}
				\beta_{X^{2}Y^{2}}+\beta_{XYXY}=0,\\
				2\beta_{XY^{3}}=0,\\
				\beta_{Y^{2}}=\beta_{1}-\beta_{X^2},\\
				\beta_{XY^{2}}=\beta_{X}-\beta_{X^3},
			\end{aligned}
		\end{equation*}
	\hfill
	\columnbreak
		\begin{equation*}
			\begin{aligned}
				\beta_{Y^{3}}=\beta_{Y}-\beta_{X^2Y},\\
				\beta_{X^{2}Y^{2}}=\beta_{X^{2}}-\beta_{X^4},\\
				\beta_{XY^{3}}=\beta_{XY}-\beta_{X^3Y},\\
				\beta_{Y^{4}}=\beta_{Y^{2}}-\beta_{X^2Y^2}.
			\end{aligned}
		\end{equation*}
	\end{multicols}
	\noindent The solution to (\ref{eq-bc1-r5}) is given by
\begin{multicols}{2}
	\begin{equation*} 
		\begin{aligned}	
			\beta_{X} = \beta_{X^3},\\
			\beta_{XYXY} = \beta_{X^4}-\beta_{X^2},\\
			\beta_{XY} =\beta_{X^2Y}=\beta_{XY^2} =\beta_{X^3Y} =  \beta_{XY^3} =0,  
		\end{aligned}
	\end{equation*}
\vfill
	\columnbreak
	\begin{equation*}
		\begin{aligned}	
			\beta_{Y^4} = \beta_{1} -2\beta_{X^2}+\beta_{X^4}, \\
			\beta_{Y^3} = \beta_{Y}, \\
			\beta_{X^2Y^2} = \beta_{X^2}-\beta_{X^4},
		\end{aligned}
	\end{equation*}
\end{multicols}
\noindent and thus $\mc{M}_2$ takes the form (\ref{matrix-bc1}).
 \end{proof}

\begin{proposition}\label{form-of-m2-c1}
	Suppose $\beta\equiv \beta^{(4)}$ is a normalized nc sequence with a moment matrix $\mc{M}_2$ 
	of rank 5 satisfying 
	the relations $\bX\bY+\bY\bX=\mbf 0$ and $\bX^2+\bY^2=\mds{1}.$
	Then $\mc{M}_2$ is positive semidefinite if and only if 
			\begin{equation} \label{nec-suf-cond-c1} 
				|\beta_X|<\beta_{X^2}<1,
				\quad |\beta_{Y}|<(1-\beta_{X^2}),
				\quad c<\beta_{X^4}<\beta_{X^2},
			\end{equation}
			where 
$$c:=\frac{-\beta_{X^2}^3+\beta_{X^2}^4-\beta_{X}^2+\beta_{Y}^2\beta_{X}^2+
3\beta_{X^2}\beta_{X}^2-2\beta_{X^2}^2\beta_{X}^2}{-\beta_{X^2}+\beta_{Y}^2\beta_{X^2}+
	\beta_{X^2}^2+\beta_{X}^2-\beta_{X^2}\beta_{X}^2}.$$
\end{proposition}

In the proof of Proposition \ref{form-of-m2-c1} we will need the following lemma.

\begin{lemma}\label{psd-lema}
	Let $A\in \mathbb{SR}^{t\times t}$ be a real symmetric $t\times t$ matrix and
	$A^{\prime} \in \mathbb{SR}^{(t+u)\times (t+u)}$ a real symmetric $(t+u)\times (t+u)$ matrix of the form 
		$$A'=\begin{pmatrix} A& B\\ B^t&C\end{pmatrix}$$ 
	where $B\in\RR^{t\times u},$ $C\in \mathbb{SR}^{u\times u}$. Suppose that 
		$\text{rank}(A)=\text{rank}(A').$
	Then $A$ is positive semidefinite if and only if $A'$ is positive semidefinite.
\end{lemma}

\begin{proof}
	Since $A$ is a principal submatrix of $A'$, $A'$ being psd implies that $A$ is psd. This proves the implication $(\Leftarrow)$. It remains to prove the implication $(\Rightarrow)$.
	The assumption $\text{rank}(A)=\text{rank}(A')$ implies that there is a matrix $W\in \RR^{t\times u}$ such that $B=AW$ and $C=B^tW=W^tAW$.
	By a theorem of Smul'jan \cite{Smu59}, $A$ being psd implies that $A'$ is psd. 
\end{proof}

\begin{proof}[Proof of Proposition \ref{form-of-m2-c1}]
	Write $\cM_2$ as $\cM_2=\begin{pmatrix} A& B\\ B^t&C\end{pmatrix}$ where 
		$$A={{\cM_2}}|_{\{\mds{1},\bX,\bY,\bX^2,\bX\bY\}},\quad 	
			B={{\cM_2}}|_{\{\mds{1},\bX,\bY,\bX^2,\bX\bY\}, \{\bY\bX,\bY^2\}},\quad
			C={{\cM_2}}|_{\{\bY\bX,\bY^2\}}.$$
	Since $\mc{M}_2$ satisfies the relations $\bX\bY+\bY\bX=\mbf 0$ and $\bX^2+\bY^2=\mds{1},$ it follows that
	$\text{rank}(\cM_2)=\text{rank}(A)=5.$
	Now by Lemma \ref{psd-lema}, $\cM_2$ is psd if and only if $A$ is psd.
	Since $\text{rank}(A)=5$, $A$ is psd if and only if $A$ is positive definite (pd). By Sylvester's criterion $A$ is pd if and only if all its principal minors are positive.
	
	First we will prove that the conditions in  (\ref{nec-suf-cond-c1}) are necessary for $A$ being pd. 
	Since 
		$$A|_{\{\bX\}}=\beta_{X^2}>0, \quad A|_{\{\bY\}}=1-\beta_{X^2}>0,\quad A|_{\{\bX\bY\}}=\beta_{X^2}-\beta_{X^4}>0,$$
	we have that $1>\beta_{X^2}>\beta_{X^4}>0$. 
	Now 
		$$0<\det\Big(A|_{\{\bX, \bX^2\}}+\begin{pmatrix}  0 & 0 \\ 0 & \beta_{X^2}-\beta_{X^4}\end{pmatrix}\Big)
		=\beta_{X^2}^2-\beta_{X}^2,$$
	implies that $\beta_{X^2}>|\beta_X|$.
	Further on, 
		$$0<\det\Big(A|_{\{\mds 1, \bX, \bY, \bX^2\}}+
			\begin{mpmatrix} 0 & 0 & 0 & 0\\ 0 & 0 & 0 & 0 \\ 0 & 0 & 0 & 0 \\ 0 & 0 & 0 & \beta_{X^2}-\beta_{X^4} \end{mpmatrix}\Big)
			=(\beta_{X^2}^2-\beta_{X}^2)((1-\beta_{X^2})^2-\beta_{Y}^2),$$
	gives $1-\beta_{X^2}>|\beta_Y|.$ Finally, $\det\left(A|_{\{\mds 1, \bX, \bY, \bX^2\}}\right)>0$ implies that $\beta_{X^4}>c$. 

	It remains to prove that the conditions in (\ref{nec-suf-cond-c1}) are sufficient for $A$ being pd.
	The principal minors $\alpha_i$, $i=1,\ldots, 5$, of $A$ are 
	\begin{eqnarray*}
		\alpha_1 &=& 1,\quad 
		\alpha_2=\beta_{X^2}-\beta_{X}^2, \quad 
		\alpha_3
				=(\beta_{X^2}-\beta_{X}^2)(1-\beta_{X^2})-\beta_{X^2}\beta_{Y}^2,\\
		\alpha_4 &=& d+\alpha_3\beta_{X^4},\quad
		\alpha_5 = \alpha_4(\beta_{X^2}-\beta_{X^4}),
	\end{eqnarray*}
	where
	\begin{eqnarray*}
		d &=& -\beta_{X^2}^3+\beta_{X^2}^4-\beta_{X}^2+\beta_{Y}^2\beta_{X}^2+
				3\beta_{X^2}\beta_{X}^2-2\beta_{X^2}^2\beta_{X}^2,
	\end{eqnarray*}
	Now, if (\ref{nec-suf-cond-c1}) is true, then clearly $\alpha_1,\alpha_2>0$ and $\alpha_5>0$  if $\alpha_4>0$. 
	It remains to prove that $\alpha_3>0$ and $\alpha_4>0$. Now
	$\alpha_3>0$ if and only if
	\begin{equation}\label{cond-pd-c1}
		\beta_Y^2<\frac{(\beta_{X^2}-\beta_{X}^2)(1-\beta_{X^2})}{\beta_{X^2}}=
		(1-\frac{\beta_{X}^2}{\beta_{X^2}})(1-\beta_{X^2}).
	\end{equation}
	Since $(1-\beta_{X^2})^2<(1-\frac{\beta_{X}^2}{\beta_{X^2}})(1-\beta_{X^2})$, the inequality
	$|\beta_Y|<1-\beta_{X^2}$ implies that (\ref{cond-pd-c1}) is true. Finally, note that $\alpha_4>0$ if and only if $\beta_{X^4}>c$. 
\end{proof}

The following theorem characterizes normalized nc sequences $\beta$ with a moment matrix 
$\mc{M}_2$ of rank 5 satisfying the relations $\bX\bY+\bY\bX=\mbf 0$ and $\bX^2+\bY^2=\mds{1}$, 
which admit a nc measure.

\begin{theorem} \label{M(2)-XY+YX=0-bc1}
	Suppose $\beta\equiv \beta^{(4)}$ is a normalized nc sequence with a moment matrix $\mc{M}_2$ 
	of rank 5 satisfying 
	the relations $\bX\bY+\bY\bX=\mbf 0$ and $\bX^2+\bY^2=\mds{1}.$
	Then $\beta$ admits a nc measure if and only if 		
				\begin{equation}\label{conditions-bc1-r5}
					|\beta_{Y}|<1-|\beta_X|,\;
					|\beta_X|<\beta_{X^2}<1-|\beta_Y|,\;
					c \leq\beta_{X^4}<\beta_{X^2},
				\end{equation}
	where
		$$c=\frac{-\beta_{X^2}^2-|\beta_X|+2\beta_{X^2}|\beta_X|+|\beta_Y\beta_X|}
					{-1+|\beta_Y|+|\beta_X|}.$$
	Moreover, the minimal measure is unique (up to orthogonal equivalence) and of type:
	\begin{itemize}
		\item $(1,1)$ if and only if 
			$\beta_X \beta_Y=0$ and $\beta_{X^4}=c$.
	\end{itemize}
	There are two minimal measures (up to orthogonal equivalence) of type:
	\begin{itemize}
		\item $(2,1)$ if and only if $\beta_X=\beta_Y=0$ or 
				($\beta_X\beta_Y\neq 0$ and $\beta_{X^4}=c$).
		\item $(3,1)$ if and only if $\beta_{X} \beta_Y\neq 0$ and 
				$\beta_{X^4}\neq c$.
	\end{itemize}
\end{theorem}

\begin{proof}
	First note that the pairs $(x,y)\in \RR^2$ satisfying the equations
		$xy+yx=0$ and $x^2+y^2=1$ are 
		$$(1,0), (-1,0),(0,1),(0,-1)\in \RR^2.$$ 
	By Lemma \ref{support lemma} these are the only pairs in $\RR^2$ which can be 
	atoms of size 1 in a nc  measure for $\beta$.\\
	
	\noindent \textbf{Claim 1:} $\beta$ with $\beta_X=\beta_Y=0$ and psd $\mc{M}_2$
	admits a measure. Moreover, there are two minimal measures of type (2,1).\\
	
	Using (\ref{matrix-bc1}) we see that $\mc{M}_2$ is of the form
		$$\mc{M}_2=\begin{mpmatrix}
					 1 & 0 & 0 & \beta_{X^2} & 0 & 0 &  1	-\beta_{X^2}			\\
					0 & \beta_{X^2} & 0 & 0 & 0 & 0 & 0 					\\
					0& 0 & 1-\beta_{X^2} & 0 & 0 & 0 & 0				\\
					\beta_{X^2} & 0 & 0 & \beta_{X^4} & 0 & 0 & \beta_{X^2} -\beta_{X^4}	\\
					0 & 0 & 0 & 0 & \beta_{X^2}-\beta_{X^4} & \beta_{X^4}-\beta_{X^2} & 0					\\
					0 & 0 & 0 & 0 & \beta_{X^4}-\beta_{X^2} & \beta_{X^2}-\beta_{X^4} & 0 					\\
					1-\beta_{X^2} & 0 & 0 & \beta_{X^2}-\beta_{X^4} & 0 & 0 & 
					1-2\beta_{X^2}+\beta_{X^4}
				\end{mpmatrix}.$$
	Let $\mc{M}^{(\pm 1,0)}_2$ be the moment matrix generated by the atom $(\pm 1,0)$, i.e., 
		\begin{equation*} \label{com-mat}
			\mc{M}^{(\pm 1,0)}_2=\begin{mpmatrix}
 				1 & \pm 1 & 0 & 1 \\
	 			\pm 1 & 1 & 0 & \pm 1  \\
				 0 & 0 & 0 & 0 \\
				 1 & \pm 1 & 0 & 1
			\end{mpmatrix}\bigoplus \mbf 0_3.
		\end{equation*}
	We define the matrix function
		$$B(\alpha):=\mc{M}_2-\alpha \left(\mc{M}^{(1,0)}_2+\mc{M}^{(-1,0)}_2\right),$$ i.e., 
		$$B(\alpha)=	
			\begin{mpmatrix}
					 1-2\alpha & 0 & 0 & \beta_{X^2}-2\alpha & 0 & 0 &  C	\\			
					0 & \beta_{X^2}-2\alpha  & 0 & 0 & 0 & 0 & 0 					\\
					0& 0 & C& 0 & 0 & 0 & 0				\\
					\beta_{X^2}-2\alpha & 0 & 0 & \beta_{X^4}-2\alpha  & 0 & 0 &  D\\
					0 & 0 & 0 & 0 & D & -D & 0					\\
					0 & 0 & 0 & 0 &- D &  D & 0 					\\
					C & 0 & 0 &  D & 0 & 0 & C-D
				\end{mpmatrix},$$
	where
		$$C=1-\beta_{X^2}, \quad D=\beta_{X^2}-\beta_{X^4}.$$
	We have that
		\begin{equation}\label{det-eq-1-1} 
			\det({B(\alpha)}|_{\{\mds 1, \bX, \bY, \bX^2, \bX\bY\}})=\det({B(\alpha)}|_{\{\mds 1, \bX^2\}})\cdot (\beta_{X^2}-2\alpha)\cdot C\cdot D,
		\end{equation}
	and
		\begin{equation}\label{det-eq-1-2} 
			\det({B(\alpha)}|_{\{\mds 1, \bX^2\}})=\beta_{X^4}-\beta_{X^2}^2+2\alpha(-1+2\beta_{X^2}-\beta_{X^4}).
		\end{equation}
	Let $\alpha_0>0$ be the smallest positive number such that the rank of $B(\alpha_0)$ is smaller than 5. Note that 
	$B(\alpha_0)$ is psd.
	By (\ref{det-eq-1-1}), (\ref{det-eq-1-2}) 
	we get that
		$$\alpha_0=\min\left(\frac{\beta_{X^2}}{2},\frac{\beta_{X^2}^2-\beta_{X^4}}{2(-1+2\beta_{X^2}-\beta_{X^4})}\right).$$
	If $\alpha_0=\frac{\beta_{X^2}}{2}$, then $\beta_{X^4}-2\alpha_0=\beta_{X^4}-\beta_X^2<0$ by Proposition \ref{form-of-m2-c1} and 
	$B(\alpha_0)$ would not be psd which is a contradiction. Hence
		$$\alpha_0=\frac{\beta_{X^2}^2-\beta_{X^4}}{2(-1+2\beta_{X^2}-\beta_{X^4})}.$$
	The matrix $B(\alpha_0)$ is psd of rank 4 and satisfies the column relations
			$$\bX^2=\frac{\beta_{X^2}-\beta_{X^4}}{1-\beta_{X^2}}\mds{1},\quad
			 	\bX\bY+\bY\bX=\mbf 0,\quad
				\bY^2=\frac{1-2\beta_{X^2}+\beta_{X^4}}{1-\beta_{X^2}}\mds{1}.$$
	By Theorem \ref{rank4-soln-aljaz} it has a unique (up to orthogonal equivalence)
	1-atomic nc measure with an atom $(X,Y)\in (\mbb{SR}^{2\times 2})^2$.
	Therefore 
		$\mc{M}_2$
	has a minimal measure of type $(2,1)$. Indeed, minimality follows by the following facts:
	\begin{itemize}
		\item Since $\mc{M}_2$ is a nc moment matrix, 
			there must be at least one atom of size $>1$ in the representing 
			nc
			measure. 
		\item If there is exactly one atom of size 2 in the representing nc measure, 
			then there must be at least one atom of size 1,
			since otherwise  $\mc{M}_2$ would have rank at most 4. 
			Since $\beta_X=\beta_Y=0$, atoms 
			$(1,0), (-1,0)$ (resp.\ $(0,1)$, $(0,-1)$) occur in pairs with the 
			same densities.
	\end{itemize}
	By symmetry there exists a unique minimal measure of type (2,1) involving the atoms $(0,1)$ and $(0,-1)$. 	This concludes the proof of Claim 1.\\

\noindent \textbf{Claim 2:} If $\beta$ admits a nc measure,
	then it has a representing nc measure with the atoms of size at most 2.\\

	Claim 2 follows by Proposition \ref{atoms-of-size-2} and Claim 1.\\
	
\noindent \textbf{Claim 3:} $\beta$ admits a nc measure if and only if (\ref{conditions-bc1-r5}) holds.\\

A special case $\beta_X=\beta_Y=0$ of Claim 3 follows by Claim 1.
Let us assume that $\beta_{X}\neq 0$ or $\beta_{Y}\neq 0$ 
and suppose that $\beta$ admits a nc measure. By Claim 2,
	\begin{equation}\label{oblika-M(2)-bc0}
		\mc{M}_2=\sum_i \lambda_i \mc{M}^{(x_i,y_i)}_2+\sum_j \xi_j \mc{M}^{(X_j,Y_j)}_2.
	\end{equation}
where $(x_i,y_i)\in \RR^2$, $(X_j,Y_j)\in \mathbb{SR}^{2\times 2}$, 
$\lambda_i> 0$, $\xi_j> 0$ and $\sum_i\lambda_i+\sum_j\xi_j=1$.
By Corollary \ref{nc-TTMM-cor},  
	\begin{equation}\label{moments}
		\beta^{(j)}_X=\beta^{(j)}_Y=\beta^{(j)}_{X^3}=\beta^{(j)}_{X^2Y}=\beta^{(j)}_{XY^2}=\beta^{(j)}_{Y^3}=0,
	\end{equation}
where $\beta^{(j)}_{w(X,Y)}$ are the moments of $\mc{M}^{(X_j,Y_j)}_2$.
By the first paragraph in the proof of Theorem \ref{M(2)-XY+YX=0-bc1},
\begin{equation}\label{com-part}
	\sum_i \lambda_i\mc{M}^{(x_i,y_i)}_2=\lambda_1^+ \mc{M}^{(1,0)}_2+\lambda_1^- \mc{M}^{(-1,0)}_2+\lambda_2^+ \mc{M}^{(0,1)}_2+\lambda_2^- \mc{M}^{(0,-1)}_2,
\end{equation}
	where $\lambda_i^\pm\geq 0$ for $i=1,2$.
	Using (\ref{oblika-M(2)-bc0}), (\ref{moments}) and (\ref{com-part}) we conclude that
		$$\sum_i \lambda_i\mc{M}^{(x_i,y_i)}_2=\begin{mpmatrix}
 			a & \beta_{X} & \beta_Y & b & 0 & 0 & c \\
	 		\beta_{X} & b & 0 & \beta_{X} & 0 & 0 & 0 \\
			 \beta_Y & 0 & a & 0 & 0 & 0 & \beta_Y \\
			 b & \beta_{X} & 0 & b & 0 & 0 & 0 \\
			 0 & 0 & 0 & 0 & 0 & 0 & 0 \\
			 0 & 0 & 0 & 0 & 0 & 0 & 0 \\
			 c & 0 & \beta_Y & 0 & 0 & 0 & c \\
		\end{mpmatrix}
		\quad \text{for some } a,b,c\geq 0,$$
	where 
		$$\beta_{X}=\lambda_1^+-\lambda_1^-,\quad \beta_Y=\lambda_2^+-\lambda_2^-,\quad b=\lambda_1^++\lambda_1^-, \quad c=\lambda_2^++\lambda_2^-, \quad
			a=b+c.$$\\
			
	\noindent \textbf{Subclaim 3.1:} 
		We have that 
			\begin{equation} \label{c5-1-ineq}
				\displaystyle \sum_i \lambda_i\mc{M}^{(x_i,y_i)}_2\succeq  |\beta_{X}| \mc{M}^{(\sign(\beta_X)1,0)}_2+
					|\beta_Y| \mc{M}^{(0,\sign(\beta_Y)1)}_2=:A,
			\end{equation}
		where $\sign(x)=+$ if $x\geq 0$ and $-$ otherwise.\\
		
		Since $\beta_{X}=\lambda_1^+-\lambda_1^-$ and $\lambda_1^\pm\geq 0$, it follows that $\lambda_1^{\sign(\beta_X)}\geq |\sign(\beta_X)|$.
		Thus, (\ref{c5-1-ineq}) follows.\\

	$\mc{M}_2-A$ is of the form
		$$
		\begin{mpmatrix}
			1-|\beta_Y|-|\beta_{X}| & 0 & 0 & \beta_{X^2}-|\beta_X| & 0 & 0 & E \\
			0 & \beta_{X^2}-|\beta_X| & 0 & 0 &0 & 0 & 0\\
			0 & 0 & E & 0 & 0 & 0 & 0\\
			\beta_{X^2}-|\beta_X| & 0 & 0 & \beta_{X^4}-|\beta_X| & 0 & 0 & \beta_{X^2}-\beta_{X^4}\\
			0 & 0 & 0 & 0 & \beta_{X^2}-\beta_{X^4} & -\beta_{X^2}+\beta_{X^4} & 0\\
			0 & 0 & 0 & 0 & -\beta_{X^2}+\beta_{X^4} & \beta_{X^2}-\beta_{X^4} & 0\\
			E & 0 & 0 & \beta_{X^2}-\beta_{X^4} & 0 & 0 & F,
		\end{mpmatrix},
		$$
	where 
		$$E=1-\beta_{X^2}-|\beta_Y|\quad \text{and}\quad
			F=E-\beta_{X^2}+\beta_{X^4}.$$
	By Subclaim 3.1, 
		$$\mc{M}_2-\sum_i \lambda_i\mc{M}^{(x_i,y_i)}_2\preceq \mc{M}_2-A.$$
	Necessary conditions for the existence of a nc measure by Proposition \ref{M2-psd} and Corollary \ref{lin-ind-of-4-col} for a nc moment matrix
	$\mc{M}_2-A$ are that $\mc{M}_2-A$ is psd and $(\mc{M}_2-A)|_{\{\mds 1, \bX, \bY, \bX\bY\}}$ is pd.
	The latter is equivalent to
		\begin{equation} \label{c1-meas-cond} 
			1-|\beta_Y|-|\beta_{X}|>0,\quad \beta_{X^2}-|\beta_X|>0,\quad 1-\beta_{X^2}-|\beta_Y|>0,\quad  \beta_{X^2}-\beta_{X^4}>0,\quad
				\det(\mc{M}_2|_{\{\mds 1, \bX^2\}})\geq 0.
		\end{equation}
	Further on, this system is equivalent to the conditions (\ref{conditions-bc1-r5}).
	This proves Claim 3.\\
	
	\noindent \textbf{Claim 4:} Minimal measures are as stated in the theorem.\\
	
	If $\beta_X=\beta_Y=0$, Claim 4 follows by Claim 1. Suppose $\beta_X\neq 0$ or 
	$\beta_Y\neq 0$. 
	Let $A$ be as in the proof of Claim 3.
	The following statements are true:
	\begin{enumerate}
		\item Minimal measure is unique of type (1,1) if and only if 
		the rank of $\mc{M}_2-A$ is 4 and one of the moments 		
		$\beta_{X}$, $\beta_Y$ is 0.
		\item Minimal measure is unique of type (2,1) if and only if
			\begin{enumerate}
				\item the rank of $\mc{M}_2-A$ is 4 and
					$\beta_{X}\beta_Y\neq 0$.
				\item the rank of $\mc{M}_2-A$ is 5 and
					one of the moments 	$\beta_{X}$, $\beta_Y$ is 0 
					in which case we subtract 
					$\alpha\left(\mc{M}^{(1,0)}_2+\mc{M}^{(-1,0)}_2\right)$ or 
					$\alpha\left(\mc{M}^{(0,1)}_2+\mc{M}^{(0,-1)}_2\right)$ 
					with the smallest $\alpha>0$ such that the rank falls to 4.
			\end{enumerate}
		\item There are two minimal measures of type (3,1) if and only if 
			the rank of $\mc{M}_2-A$ is 5 
			and $\beta_{X}\beta_Y\neq 0$ in which case we subtract 
			$\alpha\left(\mc{M}^{(1,0)}_2+\mc{M}^{(-1,0)}_2\right)$ or 
			$\alpha\left(\mc{M}^{(0,1)}_2+\mc{M}^{(0,-1)}_2\right)$ 	
			with the smallest $\alpha>0$ such that the rank falls to 4.
	\end{enumerate}
		The rank of $\mc{M}_2-A$ is 5 exactly when in (\ref{c1-meas-cond}) we have $\det(\mc{M}_2|_{\{\mds 1, \bX^2\}})>0$ which is exactly when $c<\beta_{X^4}$ with $c$ as in
		the statement of the theorem. 
\end{proof}


\subsection{Pair $\bX\bY+\bY\bX=\mbf 0$ and $\bY^2=\mds{1}$.} \label{subsec-2}

In this subsection we study a nc sequence $\beta$ with a moment matrix $\mc{M}_2$ of rank 5 satisfying the relations 
$\bX\bY+\bY\bX=\mbf 0$ and $\bY^2=\mds{1}$. 
In Theorem \ref{M(2)-XY+YX=0} we characterize exactly when $\beta$ admits a nc measure. Moreover, we classify type and uniqueness of the minimal measure. 

The form of $\mc{M}_2$ is given by the following proposition.

\begin{proposition}
		Suppose $\beta\equiv \beta^{(4)}$ is a nc sequence with a moment matrix $\mc{M}_2$ of rank 5
		satisfying the relations 
			\begin{equation} \label{relation-bc2} 
				\bX\bY+\bY\bX=\mbf 0\quad \text{and}\quad \bY^2=\mds{1}.
			\end{equation}
		Then $\mc{M}_2$ is of the form
			\begin{equation} \label{matrix-bc2}
				\begin{pmatrix}
					 \beta_1 & 0 & \beta_{Y} & \beta_{X^2} & 0 & 0 &  \beta_1				\\
					0 & \beta_{X^2} & 0 & \beta_{X^3} & 0 & 0 & 0 					\\
					\beta_Y & 0 & \beta_1 & 0 & 0 & 0 & \beta_Y 						\\
					\beta_{X^2} & \beta_{X^3} & 0 & \beta_{X^4} & 0 & 0 & \beta_{X^2} 		\\
					0 & 0 & 0 & 0 & \beta_{X^2} & -\beta_{X^2} & 0					\\
					0 & 0 & 0 & 0 & -\beta_{X^2} & \beta_{X^2} & 0 					\\
					\beta_{1} & 0 & \beta_{Y} & \beta_{X^2} & 0 & 0 & \beta_1
				\end{pmatrix}.
			\end{equation}
\end{proposition}

\begin{proof}
	The relations (\ref{relation-bc2}) give us the following system in $\mc{M}_2$
	\begin{multicols}{3}
		\begin{equation}\label{eq5-abby}
			\begin{aligned}
				2\beta_{XY}=0,\\
				2\beta_{X^{2}Y}=0,\\
				2\beta_{XY^{2}}=0,\\
				2\beta_{X^{3}Y}=0,
			\end{aligned}
		\end{equation}
	\vfill
	\columnbreak
		\begin{equation*}
			\begin{aligned}
				\beta_{X^{2}Y^{2}}+\beta_{XYXY}=0,\\
				2\beta_{XY^{3}}=0,\\
				\beta_{Y^{2}}=\beta_{1},\\
				\beta_{XY^{2}}=\beta_{X},
			\end{aligned}
		\end{equation*}
	\vfill
	\columnbreak
		\begin{equation*}
			\begin{aligned}
				\beta_{Y^{3}}=\beta_{Y},\\
				\beta_{X^{2}Y^{2}}=\beta_{X^{2}},\\
				\beta_{XY^{3}}=\beta_{XY},\\
				\beta_{Y^{4}}=\beta_{Y^{2}}.
			\end{aligned}
		\end{equation*}
	\end{multicols}
\noindent	The solution to (\ref{eq5-abby}) is given by
\begin{multicols}{2}
	\begin{equation}\label{eq6-abby}
		\begin{aligned}	
			\beta_{XY} =0, \\
			\beta_{X^2Y} =0,\\
			\beta_{XY^2} = \beta_X =  0, \\
			\beta_{X^3Y} =  \beta_{XY^3}=0,  
		\end{aligned}
	\end{equation}
\vfill
	\columnbreak
	\begin{equation*}
		\begin{aligned}	
			\beta_{XYXY} = -\beta_{X^2}, \\
			\beta_{Y^4} =\beta_{Y^2} = \beta_{1},  \\
			\beta_{Y^3} = \beta_{Y}, \\
			\beta_{X^2Y^2} = \beta_{X^2},
		\end{aligned}
	\end{equation*}
\end{multicols}
 \noindent and thus $\mc{M}_2$ takes the form (\ref{matrix-bc2}).
 \end{proof}

\begin{proposition}
	Suppose $\beta\equiv \beta^{(4)}$ is a normalized nc sequence with a moment matrix $\mc{M}_2$ of rank 5 satisfying the relations 
	$\bX\bY+\bY\bX=\mbf 0$ and $\bY^2=\mds{1}$. 
	Then $\mc{M}_2$ is positive semidefinite if and only if
		 	\begin{equation} \label{nec-suf-cond-c2} 
				\beta_{X^3}\in \RR,\quad  \beta_{X^2}>0, \quad \left|\beta_{Y}\right|<1, \quad
				\beta_{X^4}>\frac{\beta_{X^2}^3+\beta_{X^3}^2-\beta_{Y}^2\beta_{X^3}^2}{(1-\beta_{Y}^2)\beta_{X^2}}=:c.
			\end{equation}
\end{proposition}

	\begin{proof}
	Write $\cM_2$ as $\cM_2=\begin{pmatrix} A& B\\ B^t&C\end{pmatrix}$ where 
		$$A={{\cM_2}}|_{\{\mds{1},\bX,\bY,\bX^2,\bX\bY\}},\quad 	
			B={{\cM_2}}|_{\{\mds{1},\bX,\bY,\bX^2,\bX\bY\}, \{\bY\bX,\bY^2\}}\quad
			C={{\cM_2}}|_{\{\bY\bX,\bY^2\}}.$$
	Since $\mc{M}_2$ satisfies the relations $\bX\bY+\bY\bX=\mbf 0$ and $\bY^2=\mds{1},$ it follows that
	$\text{rank}(\cM_2)=\text{rank}(A)=5.$
	Now by Lemma \ref{psd-lema}, $\cM_2$ is psd if and only if $A$ is psd.
	Since $\text{rank}(A)=5$, $A$ is psd if and only if $A$ is positive definite (pd). By Sylvester's criterion $A$ is pd if and only if all its principal minors are positive.
	
	First we will prove that the conditions in  (\ref{nec-suf-cond-c2}) are necessary for $A$ being pd. 
	We have that 
		$A|_{\{\bX\}}=\beta_{X^2}>0$, $A|_{\{\mds 1,\bY\}}=1-\beta_{Y}^2>0$ (if and only if $1>|\beta_{Y}|$).
	Finally, $\det\left(A|_{\{\mds 1, \bX, \bY, \bX^2\}}\right)>0$ implies that $\beta_{X^4}>c$. 

	It remains to prove that the conditions in (\ref{nec-suf-cond-c2}) are sufficient for $A$ being pd.
	The principal minors $\alpha_i$, $i=1,\ldots, 5$, of $A$ are 
	\begin{equation*}
		\alpha_1 = 1,\quad 
		\alpha_2=\beta_{X^2}, \quad 
		\alpha_3
				=\beta_{X^2}(1-\beta_Y^2),\quad
		\alpha_4 = -\beta_{X^2}^3-\beta_{X^3}^2+\beta_Y^2\beta_{X^3}^2+\beta_{X^4}\beta_{X^2}(1-\beta_{Y}^2),\quad
		\alpha_5 = \alpha_4\beta_{X^2}.
	\end{equation*}
	Now, if (\ref{nec-suf-cond-c2}) is true, then clearly $\alpha_1,\alpha_2,\alpha_3>0$ and $\alpha_5>0$  if $\alpha_4>0$. 
	Note that $\alpha_4>0$ if and only if $\beta_{X^4}>c$ which concludes the proof of the proposition.
\end{proof}

The following theorem characterizes normalized nc sequences $\beta$ with a moment matrix $\mc{M}_2$ of rank 5 satisfying the relations $\bX\bY+\bY\bX=\mbf 0$ and $\bY^2=\mds{1}$, which admit a nc measure.

\begin{theorem} \label{M(2)-XY+YX=0}
	Suppose $\beta\equiv \beta^{(4)}$ is a normalized nc sequence with a moment matrix $\mc{M}_2$ of rank 5 satisfying the relations 
	$\bX\bY+\bY\bX=\mbf 0$ and $\bY^2=\mds{1}$. 
	Then $\beta$ admits a nc measure exactly in the following cases:
		\begin{enumerate}
			\item $\mc{M}_2$ is positive semidefinite and $\beta_{X^3}=\beta_Y=0$. The minimal measure is 
			unique (up to orthogonal equivalence) and of type (2,1).
			\item $\mc{M}_2$ is positive semidefinite and
			\begin{equation}\label{conditions-bc2-r5}
				\beta_{X^3}=0, \; \beta_{X^2}>0,\; 0<|\beta_Y|<1, \;
				\beta_{X^4}\geq \frac{\beta_{X^2}^2}{1-|\beta_Y|}.
			\end{equation}			
			Moreover, assume that (\ref{conditions-bc2-r5}) holds. The minimal measure is 	
			unique (up to orthogonal equivalence).
			It is of type (1,1) if and only if 
			$\beta_{X^4}=\frac{\beta_{X^2}^2}{1-|\beta_Y|}$. Otherwise it is of type (2,1).
		\end{enumerate}
\end{theorem}

\begin{proof}
	First note that the pairs $(x,y)\in \RR^2$ satisfying the equations
		$xy+yx=0$ and $y^2=1$ are 
		$$(0,1),(0,-1)\in \RR^2.$$ 
	By Lemma \ref{support lemma} these are the only pairs from $\RR^2$ which can be 
	atoms of size 1 in a nc measure for $\beta$.\\
	
	\noindent\textbf{Claim 1:} $\beta$ with $\beta_{X^3}=\beta_Y=0$ and psd $\mc{M}_2$ admits a nc measure.
		Moreover, the minimal measure is unique 	and of type (2,1).\\
	
	Using (\ref{matrix-bc2}) we see that
	$\mc{M}_2$ is of the form
		$$\mc{M}_2=\begin{mpmatrix}
					 1 & 0 & 0 & \beta_{X^2} & 0 & 0 &  1				\\
					0 & \beta_{X^2} & 0 & 0 & 0 & 0 & 0 					\\
					0& 0 & 1 & 0 & 0 & 0 & 0				\\
					\beta_{X^2} & 0 & 0 & \beta_{X^4} & 0 & 0 & \beta_{X^2} 		\\
					0 & 0 & 0 & 0 & \beta_{X^2} & -\beta_{X^2} & 0					\\
					0 & 0 & 0 & 0 & -\beta_{X^2} & \beta_{X^2} & 0 					\\
					1 & 0 & 0 & \beta_{X^2} & 0 & 0 & 1
				\end{mpmatrix}.$$
	Let $\mc{M}^{(0,\pm 1)}_2$ be the moment matrix generated by the atom $(0,\pm 1)$, i.e., 
		$$\mc{M}^{(0,\pm 1)}_2=\begin{mpmatrix}
 			1 & 0 & \pm 1 & 0 & 0 & 0 & 1 \\
	 		0 & 0 & 0 & 0 & 0 & 0 & 0 \\
			 \pm 1 & 0 & 1 & 0 & 0 & 0 & \pm 1 \\
			 0 & 0 & 0 & 0 & 0 & 0 & 0 \\
			 0 & 0 & 0 & 0 & 0 & 0 & 0 \\
			 0 & 0 & 0 & 0 & 0 & 0 & 0 \\
			 1 & 0 & \pm 1 & 0 & 0 & 0 & 1 \\
		\end{mpmatrix}.$$
	We define the matrix function
		$$B(\alpha):=\mc{M}_2-\alpha (\mc{M}^{(0,1)}_2+\mc{M}^{(0,-1)}_2),$$
	i.e.,
		$$B(\alpha)=
			\begin{mpmatrix}
					 1-2\alpha & 0 & 0 & \beta_{X^2}& 0 & 0 &  1-2\alpha	\\			
					0 & \beta_{X^2} & 0 & 0 & 0 & 0 & 0 					\\
					0& 0 & 1-2\alpha & 0 & 0 & 0 & 0				\\
					\beta_{X^2} & 0 & 0 & \beta_{X^4} & 0 & 0 & \beta_{X^2}		\\
					0 & 0 & 0 & 0 & \beta_{X^2} & -\beta_{X^2} & 0					\\
					0 & 0 & 0 & 0 & -\beta_{X^2} & \beta_{X^2} & 0 					\\
					1-2\alpha & 0 & 0 & \beta_{X^2} & 0 & 0 & 1-2\alpha
				\end{mpmatrix}.$$
	We have 
		\begin{equation}\label{det-eq-2-1} 
			\det(B(\alpha)|_{\{\mds 1, \bX, \bY, \bX^2, \bX\bY\}})=\det(B(\alpha)|_{\{\mds 1, \bX^2\}})\cdot \beta_{X^2}^2\cdot (1-2\alpha),
		\end{equation}
	and
		\begin{equation}\label{det-eq-2-2} 
			\det(B(\alpha)|_{\{\mds 1, \bX^2\}})=\beta_{X^4}-\beta_{X^2}^2-2\alpha\beta_{X^4}.
		\end{equation}
	Let $\alpha_0>0$ be the smallest positive number such that the rank of $B(\alpha_0)$ is smaller than 5.
	By (\ref{det-eq-2-1}) and (\ref{det-eq-2-2}) 
	we get that
		$$\alpha_0=\min\left(\frac{1}{2},\frac{\beta_{X^4} - \beta_{X^2}^2 }{2\beta_{X^4}}\right)=\frac{1}{2}-\frac{\beta_{X^2}^2}{2\beta_{X^4}}.$$
	The matrix $B(\alpha_0)$ is psd matrix of rank 4 and satisfies the relations
		$$\bX^2=\frac{\beta_{X^4}}{\beta_{X^2}}\mds 1,\quad
			 \bX\bY+\bY\bX=\mbf 0, \quad \bY^2=\mds 1.$$
	By Theorem \ref{rank4-soln-aljaz} it has a unique (up to orthogonal equivalence) 1-atomic measure
	with an atom $(X,Y)\in (\mathbb{SR}^{2\times 2})^2$. Therefore $\mc{M}_2$ has a minimal measure of type (2,1). Indeed, minimality follows by the following facts:
	\begin{itemize}
		\item Since $\mc{M}_2$ is a nc moment matrix, 
			there must be at least one atom of size $>1$ in the representing 
			measure. 
		\item If there is exactly one atom of size 2 in the representing measure, 
			then there must be at least one atom of size 1,
			since otherwise  $\mc{M}_2$ would have rank at most 4. 
			Since $\beta_X=\beta_Y=0$, atoms $(0,1)$, $(0,-1)$ occur in pairs with the 
			same densities.\\
	\end{itemize}

\noindent \textbf{Claim 2:} If $\beta$ admits a measure, then it has a representing measure 
	with the atoms of size at most 2.\\

	Claim 2 follows by Proposition \ref{atoms-of-size-2} and Claim 1.\\

\noindent \textbf{Claim 3:} If $\beta_{x^3}\neq 0$ or $\beta_Y\neq 0$, then $\beta$ admits a nc measure if and only if (\ref{conditions-bc2-r5}) 
	holds.\\
	
	Let us assume that $\beta_{x^3}\neq 0$ or $\beta_Y\neq 0$ 
	and suppose that $\beta$ admits a nc measure. 
	By Claim 2,
	\begin{equation}\label{oblika-M(2)}
		\mc{M}_2=\sum_i \lambda_i \mc{M}^{(x_i,y_i)}_2+\sum_j \xi_j \mc{M}^{(X_j,Y_j)}_2.
	\end{equation}
where $(x_i,y_i)\in \RR^2$, $(X_j,Y_j)\in \mathbb{SR}^{2\times 2}$,
$\lambda_i>0$, $\xi_j>0$ and $\sum_i\lambda_i+\sum_j\xi_j=1$.
By Corollary \ref{nc-TTMM-cor},  
	\begin{equation}\label{moments-bc2}
		\beta^{(j)}_X=\beta^{(j)}_Y=\beta^{(j)}_{X^3}=\beta^{(j)}_{X^2Y}=\beta^{(j)}_{XY^2}=\beta^{(j)}_{Y^3}=0,
	\end{equation}
where $\beta^{(j)}_{w(X,Y)}$ are the moments of $\mc{M}^{(X_j,Y_j)}_2$.
By the first paragraph in the proof of Theorem \ref{M(2)-XY+YX=0}, it follows that
\begin{equation}\label{com-part-bc2}
	\sum_i \lambda_i\mc{M}^{(x_i,y_i)}_2=
		\lambda_+ \mc{M}^{(0,1)}_2+\lambda_- \mc{M}^{(0,-1)}_2,
\end{equation}
	where $\lambda_\pm\geq 0$.
	Using (\ref{oblika-M(2)}), (\ref{moments-bc2}) and (\ref{com-part-bc2}) we conclude that
$\sum_i \lambda_i\mc{M}^{(x_i,y_i)}_2$ is of the form
		\begin{equation}\label{eq-at-y2=1}
			\sum_i \lambda_i\mc{M}^{(x_i,y_i)}_2=\begin{mpmatrix}
 			a & 0 & \beta_Y & 0 & 0 & 0 & a \\
	 		0 & 0 & 0 & 0 & 0 & 0 & 0 \\
			 \beta_Y & 0 & a & 0 & 0 & 0 & \beta_Y \\
			 0 & 0 & 0 & 0 & 0 & 0 & 0 \\
			 0 & 0 & 0 & 0 & 0 & 0 & 0 \\
			 0 & 0 & 0 & 0 & 0 & 0 & 0 \\
			 a & 0 & \beta_Y & 0 & 0 & 0 & a \\
		\end{mpmatrix} \quad \text{for some }a\geq 0,
		\end{equation}
	where 
		$$\beta_Y=\lambda_+-\lambda_-, \quad a=\lambda_++\lambda_-.$$
		
	\noindent \textbf{Subclaim 3.1:} If $\beta$ has a measure, then $\beta_{X^3}=0$.\\
	
	Combining (\ref{oblika-M(2)}), (\ref{moments-bc2}) and (\ref{eq-at-y2=1}), the subclaim follows.\\
	
	\noindent \textbf{Subclaim 3.2:} 
		We have that 
		\begin{equation}\label{c5-2-ineq}
			\sum_i \lambda_i\mc{M}^{(x_i,y_i)}_2\succeq |\beta_Y|\mc M^{(0,\sign(\beta_Y)1)}=:A,
		\end{equation} 
	where $\sign(\beta_Y)=+$ if $\beta_Y\geq 0$ and $-1$ otherwise.\\
	
	Since $\beta_{Y}=\lambda_+-\lambda_-$ and $\lambda_\pm\geq 0$, it follows that $\lambda_{\sign(\beta_Y)}\geq |\sign(\beta_Y)|$.
		Thus, (\ref{c5-2-ineq}) follows.\\
	
	$\mc{M}_2-A$ is of the form
		$$\begin{mpmatrix}
		1-|\beta_Y| & 0 & 0 & \beta_{X^2} & 0 & 0 & 1-|\beta_Y| \\
		0 & \beta_{X^2} & 0 & 0 & 0 & 0 & 0 \\
		0 & 0 & 1-|\beta_Y| & 0 & 0 & 0 & 0 \\
		\beta_{X^2} & 0 & 0 & \beta_{X^4} & 0 & 0 & \beta_{X^2}\\
		0 & 0 & 0 & 0 & \beta_{X^2} & -\beta_{X^2} & 0 \\
		0 & 0 & 0 & 0 & -\beta_{X^2} & \beta_{X^2} & 0 \\	
		1-|\beta_Y| & 0 & 0 & \beta_{X^2} & 0 & 0 & 1-|\beta_Y|
		\end{mpmatrix}.$$
	By Subclaim 3.2, $\mc{M}_2-\sum_i \lambda_i\mc{M}^{(x_i,y_i)}_2\succeq \mc{M}_2-A$.
	Necessary conditions for the existence of a nc measure by Proposition \ref{M2-psd} and Corollary \ref{lin-ind-of-4-col} for a nc moment matrix
	$\mc{M}_2-A$ are that $\mc{M}_2-A$ is psd and $(\mc{M}_2-A)|_{\{\mds 1, \bX, \bY, \bX\bY\}}$ is pd.
	The latter is equivalent to
		\begin{equation} \label{c2-meas-cond} 
			1-|\beta_Y|>0,\quad \beta_{X^2}>0,\quad \det(\mc{M}_2|_{\{\mds 1, \bX^2\}})\geq 0.
		\end{equation}
	Further on, this system is equivalent to the conditions (\ref{conditions-bc2-r5}).
	This proves Claim 3.\\
	
	\noindent\textbf{Claim 4:} Minimal measures are as stated in the theorem.\\

	If $\beta_{X^3}=\beta_Y=0$, Claim 4 follows by Claim 1. Suppose $\beta_{Y}\neq 0$. Let
	$A$ be as in Subclaim 3.1. The following statements are true:
	\begin{enumerate}
		\item Minimal measure is unique of type (1,1) if and only if the rank of $\mc{M}_2-A\succeq 0$ is 4.
		\item Minimal measure is unique of type (2,1) if and only if the rank of $\mc{M}_2-A\succeq 0$ is 5 
			in which case we subtract $\alpha\left(\mc{M}^{(0,1)}_2+\mc{M}^{(0,-1)}_2\right)$ with the smallest
			$\alpha>0$ such that the rank falls to 4.
	\end{enumerate}
	The rank of $\mc{M}_2-A$ is 5 exactly when in (\ref{c2-meas-cond}) 
	we have $\det(\mc{M}_2|_{\{\mds 1, \bX^2\}})>0$ which is exactly when $\beta_{X^4}> \frac{\beta_{X^2}^2}{1-|\beta_Y|}$.
\end{proof}


\subsection{Pair $\bX\bY+\bY\bX=\mbf 0$ and $\bY^2-\bX^2=\mds{1}$.} \label{subsec-3}

In this subsection we study a nc  sequence $\beta$ with a moment matrix $\mc{M}_2$ of rank 5 satisfying the relations 
$\bX\bY+\bY\bX=\mbf 0$ and $\bY^2+\bX^2=\mds{1}$. 
In Theorem \ref{M(2)-XY+YX=0-bc3} we characterize exactly when $\beta$ admits a nc measure. Moreover, we classify the type and uniqueness of the minimal measure. 

The form of $\mc{M}_2$ is given by the following proposition.

\begin{proposition}
		Suppose $\beta\equiv \beta^{(4)}$ is a nc sequence with a moment matrix 
		$\mc{M}_2$ of rank 5 satisfying the relations 
			\begin{equation} \label{relation-bc3} 
				\bX\bY+\bY\bX=\mbf 0\quad \text{and}\quad \bY^2-\bX^2=\mds{1}.
			\end{equation}
		Then $\mc{M}_2$ is of the form
			\begin{equation} \label{matrix-bc3}
				\begin{mpmatrix}
					 \beta_1 & \beta_{X} & \beta_Y & \beta_{X^2} & 0 & 0 &  \beta_1+\beta_{X^2}			\\
					\beta_{X} & \beta_{X^2} & 0 & -\beta_{X} & 0 & 0 & 0 					\\
					\beta_Y & 0 & \beta_1+\beta_{X^2} & 0 & 0 & 0 & \beta_Y					\\
					\beta_{X^2} & -\beta_{X} & 0 & \beta_{X^4} & 0 & 0 & \beta_{X^2} +\beta_{X^4}	\\
					0 & 0 & 0 & 0 & \beta_{X^2}+\beta_{X^4} &
						 -\beta_{X^2}-\beta_{X^4} & 0					\\
					0 & 0 & 0 & 0 & -\beta_{X^2}-\beta_{X^4} & 
						\beta_{X^2}+\beta_{X^4} & 0 					\\
					\beta_{1}+\beta_{X^2} & 0 & \beta_{Y} & \beta_{X^2}+\beta_{X^4} & 0 & 
						0 & \beta_1+2\beta_{X^2}+\beta_{X^4}
				\end{mpmatrix}.
			\end{equation}
\end{proposition}

\begin{proof}
The relations (\ref{relation-bc1}) give us the following system in $\mc{M}_2$
	\begin{multicols}{3}
		\begin{equation}\label{eq-bc3-r5}
			\begin{aligned}
				2\beta_{XY}=0,\\
				2\beta_{X^{2}Y}=0,\\
				2\beta_{XY^{2}}=0,\\
				2\beta_{X^{3}Y}=0,
			\end{aligned}
		\end{equation}
	\vfill
	\columnbreak
		\begin{equation*}
			\begin{aligned}
				\beta_{X^{2}Y^{2}}+\beta_{XYXY}=0,\\
				2\beta_{XY^{3}}=0,\\
				\beta_{Y^{2}}=\beta_{1}+\beta_{X^2},\\
				\beta_{XY^{2}}=\beta_{X}+\beta_{X^3},
			\end{aligned}
		\end{equation*}
	\vfill
	\columnbreak
		\begin{equation*}
			\begin{aligned}
				\beta_{Y^{3}}=\beta_{Y}+\beta_{X^2Y},\\
				\beta_{X^{2}Y^{2}}=\beta_{X^{2}}+\beta_{X^4},\\
				\beta_{XY^{3}}=\beta_{XY}+\beta_{X^3Y},\\
				\beta_{Y^{4}}=\beta_{Y^{2}}+\beta_{X^2Y^2}.
			\end{aligned}
		\end{equation*}
	\end{multicols}
	\noindent The solution to (\ref{eq-bc3-r5}) is given by
\begin{multicols}{2}
	\begin{equation*} 
		\begin{aligned}	
			\beta_{X} = -\beta_{X^3},\\
			\beta_{XYXY} = -\beta_{X^4}-\beta_{X^2},\\
			\beta_{XY} =\beta_{X^2Y}=\beta_{XY^2} =\beta_{X^3Y} =  \beta_{XY^3}=0,  
		\end{aligned}
	\end{equation*}
\vfill
	\columnbreak
	\begin{equation*}
		\begin{aligned}	
			\beta_{X^4} = \beta_{1} +2\beta_{X^2}+\beta_{X^4}, \\
			\beta_{Y^3} = \beta_{Y}, \\
			\beta_{X^2Y^2} = \beta_{X^2}+\beta_{X^4},
		\end{aligned}
	\end{equation*}
\end{multicols}
 \noindent and thus $\mc{M}_2$ takes the form (\ref{matrix-bc3}).
 \end{proof}

\begin{proposition}
		Suppose $\beta\equiv \beta^{(4)}$ is a nc sequence with a moment matrix 
	$\mc{M}_2$ of rank 5 satisfying the relations 
		$\bX\bY+\bY\bX=\mbf 0$ and $\bY^2-\bX^2=\mds{1}$.
	Then $\mc{M}_2$ is positive semidefinite if and only if 
			\begin{equation} \label{nec-suf-cond-c3} 
				0\leq |\beta_{X}|<\sqrt{\beta_{X^2}},
				\quad |\beta_{Y}|<c,
				\quad d<\beta_{X^4},
			\end{equation}
			where 
\begin{eqnarray*}
	c&:=&\sqrt{ \frac{(1+\beta_{X^2})(\beta_{X^2}-\beta_{X}^2)}{\beta_{X^2}} },\\ 
	d&:=&\frac{ \beta_{X^2}^3+\beta_{X^2}^4+\beta_{X}^2-\beta_Y^2\beta_{X}^2+3\beta_{X^2}\beta_{X}^2+2\beta_{X^2}^2\beta_{X}^2}{
		(1+\beta_{X^2})(\beta_{X^2}-\beta_X^2)-\beta_Y^2\beta_{X^2}}.
\end{eqnarray*}
\end{proposition}

\begin{proof}
	Write $\cM_2$ as $\cM_2=\begin{pmatrix} A& B\\ B^t&C\end{pmatrix}$ where 
		$$A={\cM_2}|_{\{\mds{1},\bX,\bY,\bX^2,\bX\bY\}},\quad 	
			B={\cM_2}|_{\{\mds{1},\bX,\bY,\bX^2,\bX\bY\}, \{\bY\bX,\bY^2\}}\quad
			C={\cM_2}|_{\{\bY\bX,\bY^2\}}.$$
	Since $\mc{M}_2$ satisfies the relations $\bX\bY+\bY\bX=\mbf 0$ and $\bY^2-\bX^2=\mds{1},$ it follows that
	$\text{rank}(\cM_2)=\text{rank}(A)=5.$
	Now by Lemma \ref{psd-lema}, $\cM_2$ is psd if and only if $A$ is psd.
	Since $\text{rank}(A)=5$, $A$ is psd if and only if $A$ is positive definite (pd). By Sylvester's criterion $A$ is pd if and only if all its principal minors are positive.
	The principal minors $\alpha_i$, $i=1,\ldots, 5$, of $A$ are 
	\begin{eqnarray*}
		\alpha_1 &=& 1,\quad 
		\alpha_2=\beta_{X^2}-\beta_{X}^2, \quad 
		\alpha_3
				=(1+\beta_{X^2})\alpha_2-\beta_Y^2\beta_{X^2},\\
		\alpha_4 &=& \beta_{X^4}\alpha_3- \beta_{X^2}^3+\beta_{X^2}^4+\beta_{X}^2-\beta_Y^2\beta_{X}^2+3\beta_{X^2}\beta_{X}^2+2\beta_{X^2}^2\beta_{X}^2,\quad
		\alpha_5 = \alpha_4(\beta_{X^2}+\beta_{X^4}).
	\end{eqnarray*}
	Thus
		$$\alpha_2>0\; \Leftrightarrow\; 0\leq |\beta_{X}|<\sqrt{\beta_{X^2}}, \quad
			\alpha_3>0\; \Leftrightarrow\;  |\beta_{Y}|<c,\quad
			\alpha_4>0\; \Leftrightarrow\;  d<\beta_{X^4}.$$
	Note that $\alpha_2,\alpha_3 ,\alpha_4>0$ also imply $\alpha_5>0$ (since $\beta_{X^4}>0$ if $A|_{\{\mds{1},\bX,\bY,\bX^2\}}$ is pd).
	This proves the proposition.
\end{proof}

The following theorem characterizes normalized nc sequences $\beta$ with a moment matrix $\mc{M}_2$ of rank 5 satisfying the relations $\bX\bY+\bY\bX=\mbf 0$ and $\bY^2-\bX^2=\mds{1}$, which admit a nc measure.

\begin{theorem} \label{M(2)-XY+YX=0-bc3}
	Suppose $\beta\equiv \beta^{(4)}$ is a nc sequence with a moment matrix 
		$\mc{M}_2$ of rank 5 satisfying the relations 
				$\bX\bY+\bY\bX=\mbf 0$ and $\bY^2-\bX^2=\mds{1}.$
	Then $\beta$ admits a nc measure exactly in the following cases:
		\begin{enumerate}
			\item $\mc{M}_2$ is positive semidefinite and $\beta_{X}=\beta_Y=0$.
				The minimal measure is unique (up to orthogonal equivalence) and of type (2,1).
			\item $\mc{M}_2$ is positive semidefinite and 
			\begin{equation}\label{cond-r5-bc3}
				\beta_{X}=0,\; 0<|\beta_Y|<1,\;
				\beta_{X^4}\geq \frac{\beta_{X^2}^2}{1-|\beta_Y|}.
			\end{equation}
				Moreover, assume that (\ref{cond-r5-bc3}) holds. 
				The minimal measure is unique (up to orthogonal equivalence). It is of type (1,1)
				if and only if $\beta_{X^4}= \frac{\beta_{X^2}}{1-|\beta_Y|}$. Otherwise it is of type (2,1).					
		\end{enumerate}
\end{theorem}

\begin{proof}
	First note that the pairs $(x,y)\in \RR^2$ satisfying the equations
		$xy+yx=0$ and $y^2-x^2=1$ are 
		$$(0,1),(0,-1)\in \RR^2.$$ 
	By Lemma \ref{support lemma} these are the only pairs from $\RR^2$ which can be 
	atoms of size 1 in a nc measure of $\beta$.\\
	
	\noindent \textbf{Claim 1:} $\beta$ with $\beta_X=\beta_Y=0$ and psd $\mc{M}_2$ admits a nc measure. Moreover,
		the minimal measure is unique and of type (2,1).\\
		
	Using (\ref{matrix-bc2}) we see that
	$\mc{M}_2$ is of the form
		$$\mc{M}_2=\begin{mpmatrix}
					 1 & 0 & 0 & \beta_{X^2} & 0 & 0 &  1	+\beta_{X^2}			\\
					0 & \beta_{X^2} & 0 & 0 & 0 & 0 & 0 					\\
					0& 0 & 1+\beta_{X^2} & 0 & 0 & 0 & 0				\\
					\beta_{X^2} & 0 & 0 & \beta_{X^4} & 0 & 0 & \beta_{X^2} +\beta_{X^4}	\\
					0 & 0 & 0 & 0 & \beta_{X^2}+\beta_{X^4} & -\beta_{X^2}-\beta_{X^4} & 0					\\
					0 & 0 & 0 & 0 & -\beta_{X^2}-\beta_{X^4} & \beta_{X^2}+\beta_{X^4} & 0 					\\
					1+\beta_{X^2} & 0 & 0 & \beta_{X^2}+\beta_{X^4} & 0 & 0 & 1+2\beta_{X^2}+\beta_{X^4}
				\end{mpmatrix}.$$
	Let $\mc{M}^{(0,1)}_2$, $\mc{M}^{(0,-1)}_2$ be as in the proof of Theorem \ref{M(2)-XY+YX=0-bc1}.
	We define the matrix function
		$$B(\alpha):=\mc{M}_2-\alpha \left(\mc{M}^{(0,1)}_2+\mc{M}^{(0,-1)}_2\right),$$
	i.e., 
			$$B(\alpha)=
		\begin{mpmatrix}
					1-2\alpha & 0 & 0 & \beta_{X^2} & 0 & 0 &  C(\alpha)		\\
					0 & \beta_{X^2} & 0 & 0 & 0 & 0 & 0 					\\
					0& 0 & C(\alpha) & 0 & 0 & 0 & 0				\\
					\beta_{X^2} & 0 & 0 & \beta_{X^4} & 0 & 0 & \beta_{X^2} +\beta_{X^4}	\\
					0 & 0 & 0 & 0 & \beta_{X^2}+\beta_{X^4} & -\beta_{X^2}-\beta_{X^4} & 0					\\
					0 & 0 & 0 & 0 & -\beta_{X^2}-\beta_{X^4} & \beta_{X^2}+\beta_{X^4} & 0 					\\
					C(\alpha)& 0 & 0 & \beta_{X^2}+\beta_{X^4} & 0 & 0 & D(\alpha)
				\end{mpmatrix},$$
	where 
			$$C(\alpha)=1+\beta_{X^2}-2\alpha,\quad
				D(\alpha)=1+2\beta_{X^2}+\beta_{X^4}-2\alpha.$$
	We have that
		\begin{equation}\label{det-eq-3-1} 
			\det(B(\alpha)|_{\{\mds 1, \bX, \bY, \bX^2, \bX\bY\}})=\det(B(\alpha)|_{\{\mds 1, \bX^2\}})\cdot \beta_{X^2}\cdot C(\alpha)\cdot 
				(\beta_{X^2}+\beta_{X^4}),
		\end{equation}
	and
		\begin{equation}\label{det-eq-3-2} 
			\det(B(\alpha)|_{\{\mds 1, \bX^2\}})=\beta_{X^4}-\beta_{X^2}^2-2\alpha \beta_{X^4}.
		\end{equation}
	Let $\alpha_0>0$ be the smallest positive number such that the rank of $B(\alpha_0)$ is smaller than 5. Not that $B(\alpha_0)$ is psd.
	By (\ref{det-eq-3-1}) and (\ref{det-eq-3-2}) 
	we get that
			$$\alpha_0=\min\left(\frac{1+\beta_{X^2}}{2},\frac{\beta_{X^4}-\beta_{X^2}^2}{2\beta_{X^4}}\right)=\frac{1}{2}-\frac{\beta_{X^2}^2}{2\beta_{X^4}}.$$
	
The matrix $B(\alpha_0)$ is psd matrix of rank 4 and satisfies the relations 
	$$\bX^2=\frac{\beta_{X^4}}{\beta_{X^2}}\mds 1,\quad \bX\bY+\bY\bX=\mbf 0,\quad
		\bY^2=(1+\frac{\beta_{X^4}}{\beta_{X^2}})\mds 1.$$
By Theorem \ref{rank4-soln-aljaz} it has a unique (up to orthogonal equivalence) 1-atomic measure with an atom 
$(X,Y)\in (\mathbb{SR}^{2\times 2})^2$. Therefore $\mc{M}_2$ has a unique minimal measure of type (2,1). 
Indeed, minimality follows by the following facts:
	\begin{itemize}
		\item Since $\mc{M}_2$ is a nc moment matrix, 
			there must be at least one atom of size $>1$ in its representing 
			measure. 
		\item If there is exactly one atom of size 2 in the representing measure for $\mc{M}_2$, 
			then there must be at least one atom of size 1,
			since otherwise  $\mc{M}_2$ would have rank at most 4. 
			Since $\beta_X=\beta_Y=0$, atoms $(0,1)$, $(0,-1)$ occur in pairs with the 
			same densities.
	\end{itemize}

\noindent \textbf{Claim 2:} If $\beta$ admits a nc measure, then it has a representing measure with the atoms of size at most 2.\\

Claim 2 follows by Proposition \ref{atoms-of-size-2} and Theorem \ref{M(2)-XY+YX=0} (1).\\

\noindent \textbf{Claim 3:} If $\beta_X\neq 0$ or $\beta_Y\neq 0$, then $\beta$ admits a nc measure if and only if (\ref{cond-r5-bc3}) holds.\\

Let us assume that $\beta_X\neq 0$ or $\beta_Y\neq 0$
and suppose that $\beta$ admits a nc measure. By Claim 2,
	\begin{equation}\label{oblika-M(2)-bc3}
		\mc{M}_2=\sum_i \lambda_i \mc{M}^{(x_i,y_i)}_2+\sum_j \xi_j \mc{M}^{(X_j,Y_j)}_2.
	\end{equation}
where $(x_i,y_i)\in \RR^2$, $(X_j,Y_j)\in \mathbb{SR}^{2\times 2}$, $\lambda_i> 0$, $\xi_j> 0$ and 
$\sum_i\lambda_i+\sum_j\xi_j=1$.
By Corollary \ref{nc-TTMM-cor},  
	\begin{equation}\label{moments-bc3}
		\beta^{(j)}_X=\beta^{(j)}_Y=\beta^{(j)}_{X^3}=\beta^{(j)}_{X^2Y}=\beta^{(j)}_{XY^2}=\beta^{(j)}_{Y^3}=0,
	\end{equation}
where $\beta^{(j)}_{w(X,Y)}$ are the moments of $\mc{M}^{(X_j,Y_j)}_2$.
By the first paragraph in the proof of Theorem \ref{M(2)-XY+YX=0-bc3},
\begin{equation}\label{com-part-bc3}
	\sum_i \lambda_i\mc{M}^{(x_i,y_i)}_2=
		\lambda_+ \mc{M}^{(0,1)}_2+\lambda_- \mc{M}^{(0,-1)}_2,
\end{equation}
	where $\lambda_\pm\geq 0$.
	Using (\ref{oblika-M(2)-bc3}), (\ref{moments-bc3}) and (\ref{com-part-bc3}) we conclude that
$\sum_i \lambda_i\mc{M}^{(x_i,y_i)}_2$ is of the form
		\begin{equation}\label{form-of-B-bc3-r5}
		\sum_i \lambda_i\mc{M}^{(x_i,y_i)}_2=\begin{mpmatrix}
 			a & 0 & \beta_Y & 0 & 0 & 0 & a \\
	 		0 & 0 & 0 & 0 & 0 & 0 & 0 \\
			 \beta_Y & 0 & a & 0 & 0 & 0 & \beta_Y \\
			 0 & 0 & 0 & 0 & 0 & 0 & 0 \\
			 0 & 0 & 0 & 0 & 0 & 0 & 0 \\
			 0 & 0 & 0 & 0 & 0 & 0 & 0 \\
			 a & 0 & \beta_Y & 0 & 0 & 0 & a \\
		\end{mpmatrix}\quad \text{for some } a\geq 0
		,\end{equation}
	where 
		$$\beta_Y=\lambda_+-\lambda_-, \quad a=\lambda_++\lambda_-.$$
	
	\noindent \textbf{Subclaim 3.1:} If $\beta$ has a nc measure,
		then $\beta_X=0$.\\
	
		Combining (\ref{oblika-M(2)-bc3}), (\ref{moments-bc3}) and (\ref{form-of-B-bc3-r5}), the subclaim follows.\\
	
	\noindent \textbf{Subclaim 3.2:} 
		We have that 
		\begin{equation}\label{c5-3-ineq}
			\sum_i \lambda_i\mc{M}^{(x_i,y_i)}_2\succeq |\beta_Y|\mc M_2^{(0,\sign(\beta_Y)1)}=:A.
		\end{equation}
		where $\sign(\beta_Y)=+$ if $\beta_Y\geq 0$ and $-1$ otherwise.\\
	
	Since $\beta_{Y}=\lambda_+-\lambda_-$ and $\lambda_\pm\geq 0$, it follows that $\lambda_{\sign(\beta_Y)}\geq |\sign(\beta_Y)|$.
		Thus, (\ref{c5-3-ineq}) follows.\\
	
	$\mc{M}_2-A$ is of the form 
	$$\begin{mpmatrix}
		1-|\beta_{Y}| & 0 & 0 & \beta_{X^2} & 0 & 0 & E\\
		0 & \beta_{X^2} & 0 & 0 & 0 & 0 & 0\\
		0 & 0 & E & 0 & 0 & 0 & 0 \\
		\beta_{X^2} & 0 & 0 & \beta_{X^4} & 0 & 0 & \beta_{X^2}+\beta_{X^4}\\
		0 & 0 & 0 & 0 & \beta_{X^2}+\beta_{X^4} & -\beta_{X^2}-\beta_{X^4} & 0 \\
		0 & 0 & 0 & 0 & -\beta_{X^2}-\beta_{X^4} & \beta_{X^2}+\beta_{X^4} & 0 \\		
		E & 0 & 0 & \beta_{X^2}+\beta_{X^4} & 0 & 0 & F,
	\end{mpmatrix},$$
	where 
		$$E=1+\beta_{X^2}-|\beta_{Y}|,\quad F=E+\beta_{X^2}+\beta_{X^4}.$$
	By Subclaim 3.1, $\mc{M}_2-\sum_i \lambda_i\mc{M}^{(x_i,y_i)}_2\preceq \mc{M}_2-A$.
	Necessary conditions for the existence of a measure by Proposition \ref{M2-psd} and Corollary \ref{lin-ind-of-4-col} for a nc moment matrix
	$\mc{M}_2-A$ are that $\mc{M}_2-A$ is psd and $(\mc{M}_2-A)|_{\{\mds 1, \bX, \bY, \bX\bY\}}$ is pd.
	The latter is equivalent to
		\begin{equation} \label{c3-meas-cond} 
			1-|\beta_Y|>0,\quad \beta_{X^2}>0,\quad 1+\beta_{X^2}-|\beta_{Y}|>0,\quad \beta_{X^2}+\beta_{X^4} >0,\quad \beta_{X^4}\geq 0,\quad 
				\det(\mc{M}_2|_{\{\mds 1, \bX^2\}})\geq 0.
		\end{equation}
	Further on, using Subclaim 3.1 this system is equivalent to the conditions (\ref{cond-r5-bc3}).
	This proves Claim 3.\\
	
	\noindent \textbf{Claim 4:} Minimal measures are as stated in the theorem.\\
	
	If $\beta_X=\beta_Y=0$, Claim 4 follows by Claim 1. Suppose $\beta_Y\neq 0$. Let $A$ be as in the proof
	of Claim 3. The following statements are true:
	\begin{enumerate}
		\item Minimal measure is unique of type (1,1) if and only if the rank of $\mc{M}_2-A\succeq 0$ is 4.
		\item Minimal measure is unique of type (2,1) if and only if the rank of $\mc{M}_2-A\succeq 0$ is 5 in 
		which case we subtract  $\alpha\left(\mc{M}^{(0,1)}_2+\mc{M}^{(0,-1)}_2\right)$ with the smallest 
		$\alpha>0$ such that the rank falls to 4.
	\end{enumerate}
	The rank of $\mc{M}_2-A$ is 5 exactly when
	in (\ref{c3-meas-cond}) 
	we have $\det(\mc{M}_2|_{\{\mds 1, \bX^2\}})>0$ which is exactly when in addition $\beta_{X^4}> \frac{\beta_{X^2}^2}{1-|\beta_Y|}$.
\end{proof}


\subsection{Pair $\bX\bY+\bY\bX=\mbf 0$ and $\bY^2=\bX^2$.} \label{subsec-4}

In this subsection we study a nc sequence $\beta$ with a moment matrix $\mc{M}_2$ of rank 5 satisfying the relations 
$\bX\bY+\bY\bX=\mbf 0$ and $\bY^2+\bX^2=\mds{1}$. 
In Theorem \ref{M(2)-XY+YX=0-bc4} we characterize exactly when $\beta$ admits a nc measure. Moreover, 
we classify the type and uniqueness of the minimal measure.

The form of $\mc{M}_2$ is given by the following proposition.

\begin{proposition}
		Suppose $\beta\equiv \beta^{(4)}$ is a nc sequence with a moment matrix $\mc{M}_2$ satisfying the relations 
			\begin{equation} \label{relation-bc4} 
				\bX\bY+\bY\bX=\mbf 0\quad \text{and}\quad \bY^2=\bX^2.
			\end{equation}
		Then $\mc{M}_2$ is of the form
			\begin{equation} \label{matrix-bc4}
				\begin{mpmatrix}
					 \beta_1 & \beta_X & \beta_Y & \beta_{X^2} & 0 & 0 & \beta_{X^2}			\\
					\beta_{X} & \beta_{X^2} & 0 & 0 & 0 & 0 & 0 					\\
					\beta_Y & 0 & \beta_{X^2} & 0 & 0 & 0 & 		0				\\
					\beta_{X^2} & 0 & 0 & \beta_{X^4} & 0 & 0 & \beta_{X^4}\\
					0 & 0 & 0 & 0 & \beta_{X^4} &		 -\beta_{X^4} & 0					\\
					0 & 0 & 0 & 0 & -\beta_{X^4} & 	\beta_{X^4} & 0 					\\
					\beta_{X^2} & 0 & 0 & \beta_{X^4} & 0 & 0 & \beta_{X^4}
				\end{mpmatrix}.
			\end{equation}
\end{proposition}

\begin{proof}
The relations (\ref{relation-bc4}) give us the following system in $\mc{M}_2$
	\begin{multicols}{3}
		\begin{equation}\label{eq-bc4-r5}
			\begin{aligned}
				2\beta_{XY}=0,\\
				2\beta_{X^{2}Y}=0,\\
				2\beta_{XY^{2}}=0,\\
				2\beta_{X^{3}Y}=0,
			\end{aligned}
		\end{equation}
	\vfill
	\columnbreak
		\begin{equation*}
			\begin{aligned}
				\beta_{X^{2}Y^{2}}+\beta_{XYXY}=0,\\
				2\beta_{XY^{3}}=0,\\
				\beta_{Y^{2}}=\beta_{X^2},\\
				\beta_{XY^{2}}=\beta_{X^3},
			\end{aligned}
		\end{equation*}
	\vfill
	\columnbreak
		\begin{equation*}
			\begin{aligned}
				\beta_{Y^{3}}=\beta_{X^2Y},\\
				\beta_{X^{2}Y^{2}}=\beta_{X^4},\\
				\beta_{XY^{3}}=\beta_{X^3Y},\\
				\beta_{Y^{4}}=\beta_{X^2Y^2}.
			\end{aligned}
		\end{equation*}
	\end{multicols}
	\noindent The solution to (\ref{eq-bc4-r5}) is given by
	\begin{eqnarray*}
		\beta_{XY} =\beta_{X^3}=\beta_{X^2Y}=\beta_{XY^2} =\beta_{Y^3}=\beta_{X^3Y} =  \beta_{XY^3}=0,\\
		\beta_{Y^2}=\beta_{X^2},\\
		\beta_{XYXY} = -\beta_{X^2Y^2}=-\beta_{Y^4}=-\beta_{X^4},
	\end{eqnarray*}
 \noindent and thus $\mc{M}_2$ takes the form (\ref{matrix-bc4}).
 \end{proof}

\begin{proposition} 
		Suppose $\beta\equiv \beta^{(4)}$ is a nc sequence with a moment matrix $\mc{M}_2$ of rank 5 satisfying the relations
	$\bX\bY+\bY\bX=\mbf 0$ and $\bY^2=\bX^2$. Then $\mc{M}_2$ is positive semidefinite if and only if 
			\begin{equation}\label{nec-suf-cond-c4}
				0<\beta_{X^2}, \quad |\beta_{X}|<\sqrt{\beta_{X^2}},\quad 
				|\beta_{Y}|<\sqrt{\beta_{X^2}-\beta_X^2},\quad 
				\frac{\beta_{X^2}^3}{\beta_{X^2}-\beta_Y^2-\beta_{X}^2}<\beta_{X^4}.
			\end{equation}
\end{proposition}

\begin{proof}
	Write $\cM_2$ as $\cM_2=\begin{pmatrix} A& B\\ B^t&C\end{pmatrix}$ where 
		$$A={\cM_2}|_{\{\mds{1},\bX,\bY,\bX^2,\bX\bY\}},\quad 	
			B={\cM_2}|_{\{\mds{1},\bX,\bY,\bX^2,\bX\bY\}, \{\bY\bX,\bY^2\}}\quad
			C={\cM_2}|_{\{\bY\bX,\bY^2\}}.$$
	Since $\mc{M}_2$ satisfies the relations $\bX\bY+\bY\bX=\mbf 0$ and $\bX^2=\bY^2,$ it follows that
	$\text{rank}(\cM_2)=\text{rank}(A)=5.$
	Now by Lemma \ref{psd-lema}, $\cM_2$ is psd if and only if $A$ is psd.
	Since $\text{rank}(A)=5$, $A$ is psd if and only if $A$ is positive definite (pd). By Sylvester's criterion $A$ is pd if and only if all its principal minors are positive.
	The principal minors $\alpha_i$, $i=1,\ldots, 5$, of $A$ are 
	\begin{equation*}
		\alpha_1 = 1,\quad 
		\alpha_2=\beta_{X^2}-\beta_{X}^2, \quad 
		\alpha_3=-\beta_{Y}^2\beta_{X^2}-\beta_{X}^2\beta_{X^2}+\beta_{X^2}^2,\quad
		\alpha_4=-\beta_{X^2}^4+\beta_{X^4}\alpha_3,\quad
		\alpha_5 = \alpha_4 \beta_{X^4}.
	\end{equation*}
	Thus
		$$\alpha_2>0\; \Leftrightarrow\; 0\leq |\beta_{X}|<\sqrt{\beta_{X^2}}, \quad
			\alpha_3>0\; \Leftrightarrow\;  |\beta_{Y}|<\sqrt{\beta_{X^2}-\beta_X^2},\quad
			\alpha_4>0\; \Leftrightarrow\;  \frac{\beta_{X^2}^3}{\beta_{X^2}-\beta_Y^2-\beta_{X}^2}<\beta_{X^4}.$$
	Note that $\alpha_2,\alpha_3 ,\alpha_4>0$ also imply $\alpha_5>0$ (since $\beta_{X^4}>0$ if $A|_{\{\mds{1},\bX,\bY,\bX^2\}}$ is pd).
	This proves the proposition.
\end{proof}

The following theorem characterizes normalized nc sequences $\beta$ with a moment matrix $\mc{M}_2$
of rank 5 satisfying the relations $\bX\bY+\bY\bX=\mbf 0$ and $\bX^2=\bY^2$, which admit a nc measure.

\begin{theorem} \label{M(2)-XY+YX=0-bc4}
		Suppose $\beta\equiv \beta^{(4)}$ is a nc sequence with a moment matrix $\mc{M}_2$ of rank 5 satisfying the relations
	 $\bX\bY+\bY\bX=\mbf 0$ and $\bX^2=\bY^2$.
	Then $\beta$ admits a nc measure if and only if $\mc{M}_2$ is positive semidefinite and $\beta_{X}=\beta_Y=0$. 
	The minimal measure is unique (up to orthogonal equivalence)  and of type (1,1).
\end{theorem}

\begin{proof}
	First note that the only pair $(x,y)\in \RR^2$ satisfying the equations
		$xy+yx=0$ and $y^2=x^2$ is 
		$$(0,0)\in \RR^2.$$ 
	By Lemma \ref{support lemma} this is the only pair from $\RR^2$ which can be 
	an atom of size 1 in the nc measure of $\beta$.
	We have
		\begin{equation}\label{com-mat-bc4}
			\mc{M}^{(0,0)}_2=(1)\oplus \mbf 0_{6},
		\end{equation}
	where $\mbf 0_6$ stands for the $6\times 6$ matrix with only zero entries.\\
	
	\noindent\textbf{Claim 1:} If $\beta_X=\beta_Y=0$, then $\beta$ admits a nc measure. Moreover, 
		the minimal measure is unique (up to orthogonal equivalence)  and of type (1,1).\\
	
	Using (\ref{matrix-bc4}) we see that $\mc{M}_2$ is of the form
	$$\mc{M}_2=
		\begin{mpmatrix} 
			1 & 0 & 0 & \beta_{X^2} & 0 & 0 & \beta_{X^2} \\
			0 & \beta_{X^2} & 0 & 0 & 0 & 0 & 0\\
			0 & 0 & \beta_{X^2} & 0 & 0 & 0 & 0\\
			\beta_{X^2} & 0 & 0 & \beta_{X^4} & 0 & 0 & \beta_{X^4}\\
			0 & 0 & 0 & 0 & \beta_{X^4} & -\beta_{X^4} & 0\\
			0 & 0 & 0 & 0 & -\beta_{X^4} & \beta_{X^4} & 0\\
			\beta_{X^2} & 0 & 0 & \beta_{X^4} & 0 & 0 & \beta_{X^4}\\
		\end{mpmatrix}.$$
	We define the matrix function  
		$$B(\alpha):=\mc{M}_2-\alpha \mc{M}^{(0,0)}_2,$$ 
	such that
		$$
		B(\alpha)=
		\begin{mpmatrix}
					1-\alpha & 0 & 0 & \beta_{X^2} & 0 & 0 &  \beta_{X^2}	\\
					0 & \beta_{X^2} & 0 & 0 & 0 & 0 & 0 					\\
					0& 0 & \beta_{X^2} & 0 & 0 & 0 & 0				\\
					\beta_{X^2} & 0 & 0 & \beta_{X^4} & 0 & 0 & \beta_{X^4}	\\
					0 & 0 & 0 & 0 & \beta_{X^4} & -\beta_{X^4} & 0					\\
					0 & 0 & 0 & 0 & -\beta_{X^4} & \beta_{X^4} & 0 					\\
					\beta_{X^2}& 0 & 0 & \beta_{X^4} & 0 & 0 & \beta_{X^4}
				\end{mpmatrix}.
		$$
	We have that
		\begin{equation}\label{det-eq-4-1} 
			\det(B(\alpha)|_{\{\mds 1, \bX, \bY, \bX^2, \bX\bY\}})=\det(B(\alpha)|_{\{\mds 1, \bX^2\}})\cdot \beta_{X^2}^2\cdot \beta_{X^4},
		\end{equation}
	and
		\begin{equation}\label{det-eq-4-2} 
			\det(B(\alpha)|_{\{\mds 1, \bX^2\}})=\beta_{X^4}-\beta_{X^2}^2-\alpha\beta_{X^4}.
		\end{equation}
	Let $\alpha_0>0$ be the smallest positive number such that the rank of $B(\alpha_0)$ is smaller than 5.
	By (\ref{det-eq-4-1}) and (\ref{det-eq-4-2}) 
	we get that
			$$\alpha_0=\frac{\beta_{X^4}-\beta_{X^2}^2}{\beta_{X^4}}.$$
The matrix $B(\alpha_0)$ is psd matrix of rank 4 and satisfies the relations 
	$$\bX^2=\frac{\beta_{X^4}}{\beta_{X^2}}\mds 1,\quad \bX\bY+\bY\bX=\mbf 0,\quad
		\bY^2=\frac{\beta_{X^4}}{\beta_{X^2}}\mds 1.$$
By Theorem \ref{rank4-soln-aljaz} it has a unique (up to orthogonal equivalence) 1-atomic measure with an atom $(X,Y)\in
(\mathbb{SR}^{2\times 2})^2$. Therefore $\mc{M}_2$ has a unique minimal measure of type (1,1). 
Indeed, minimality follows by the following facts:
	\begin{itemize}
		\item Since $\mc{M}_2$ is a nc moment matrix, there must be at least one atom of size $>1$ in its representing 
			measure. 
		\item If there is exactly one atom of size 2 in the representing measure for $\mc{M}_2$, 
			then there must be at least one atom of size 1,
			since otherwise  $\mc{M}_2$ would have rank at most 4.\\
	\end{itemize}

\noindent \textbf{Claim 2:} If $\beta_X\neq 0$ or $\beta_Y\neq 0$, then $\beta$ does not admit a nc measure.\\

\noindent \textbf{Subclaim 2.1:} If $\beta$ admits a nc measure, then it has a representing measure with the atoms of size at 
	most 2.\\

	Subclaim 2.1 follows by Proposition \ref{atoms-of-size-2} and Claim 1.\\

Suppose that $\beta$ admits a nc measure and $\beta_X\neq 0$ or $\beta_Y\neq 0$.
By Subclaim 2.1, 
	\begin{equation}\label{oblika-M(2)-bc4}
		\mc{M}_2=\sum_i \lambda_i \mc{M}^{(x_i,y_i)}_2+\sum_j \xi_j \mc{M}^{(X_j,Y_j)}_2.
	\end{equation}
where $(x_i,y_i)\in \RR^2$, $(X_j,Y_j)\in \mathbb{SR}^{2\times 2}$, $\lambda_i> 0$, $\xi_j> 0$ and 
$\sum_i\lambda_i+\sum_j\xi_j=1$.
By Corollary \ref{nc-TTMM-cor},  
	\begin{equation}\label{moments-bc4}
		\beta^{(j)}_X=\beta^{(j)}_Y=\beta^{(j)}_{X^3}=\beta^{(j)}_{X^2Y}=\beta^{(j)}_{XY^2}=\beta^{(j)}_{Y^3}=0,
	\end{equation}
where $\beta^{(j)}_{w(X,Y)}$ are the moments of $\mc{M}^{(X_j,Y_j)}_2$.
By the first paragraph in the proof of Theorem \ref{M(2)-XY+YX=0-bc4},
\begin{equation}\label{com-part-bc4}
	\sum_i \lambda_i\mc{M}^{(x_i,y_i)}_2=
		\lambda \mc{M}^{(0,0)}_2,
\end{equation}
where $\lambda>0$.
	Using (\ref{oblika-M(2)-bc4}), (\ref{moments-bc4}) and (\ref{com-part-bc4}) it follows that
		$$0=\sum_j \beta^{(j)}_{X}=\beta_X\quad \text{and} 
			\quad 0=\sum_j \beta^{(j)}_{Y}=\beta_{Y}.$$
	This is a contradiction with the assumption $\beta_X\neq 0$ or $\beta_Y\neq 0$, which proves Claim 2.
\end{proof}


\section{BQTMP with $\mc{M}_2$ in the basic cases 1 and 2 of rank 6 }
\label{rank6-section}

In this section we solve the BQTMP for $\mc{M}_2$ in the basic cases
1 and 2 of rank 6 given by Proposition \ref{structure-of-rank5-2}. In Subsections
\ref{r6-subs1} and \ref{r6-subs2} we study each case separately. 
We characterize when $\mc{M}_2$ admits a nc measure, see Theorems \ref{M(2)-bc2-r6-new1}
and  \ref{M(2)-XY+YX=0-bc4-r6}. 
Corollaries \ref{M(2)-bc2-r6-new1-cor} and \ref{r6-bc3} translate the existence of a nc measure into the feasibility problem
of three
linear matrix inequalities and a rank-to-variety condition from Theorem \ref{com-case}.

The following proposition states that if $\beta$ has a moment matrix $\mc{M}_2$ of rank 6
in the basic cases 1, 2 or 3 given by Proposition \ref{structure-of-rank5-2} (\ref{point-2-str-rank5}) and $\beta$ admits a nc measure, 
then it has a representing measure with the atoms of size at most 2.

\begin{proposition}\label{atoms-of-size-2-r6}
		Let us fix a basic case relation 1, 2 or 3 given by 
	Proposition \ref{structure-of-rank5-2} (2) and denote it by $R$.
	If a nc sequence $\beta$ with a moment matrix $\mc{M}_2(\beta)$ of rank 6 satisfying $R$
	admits a nc measure, then it admits a nc measure with atoms of size at most 2.
\end{proposition}

\begin{proof}
	Suppose $\beta$ has a moment matrix $\mc{M}_2(\beta)$ of rank 6 satisfying $R$ and
	admits a nc measure. 
	By Proposition \ref{anticommute} we may assume that all the atoms 
	$(X_i,Y_i)\in (\mbb{SR}^{u_i\times u_i})^{2}$ of size $u_i>1$ are of the form 
		$$X_{i}=\begin{pmatrix} \gamma_i I_{t_i} & B_i \\ B_i^t & -\gamma_i I_{t_i} 
			\end{pmatrix},\quad 
		Y_{i}=\begin{pmatrix} \mu_i I_{t_i} & \mbf 0 \\ \mbf 0 & -\mu_i I_{t_i} \end{pmatrix},$$
	where $\gamma_i> 0$, $\mu_i>0$ and $B_i$ are $t_i\times t_i$ matrices.
	Calculating $Y_i^2$ we get
		\begin{equation} \label{relation-2-r6}
			Y_i^2=\mu_i^2 I_{2t_i}.
		\end{equation}
	Therefore $\mc{M}^{(X_i,Y_i)}_2$ satisfies the relations $R$ and
	$(\ref{relation-2-r6})$ and hence it is of rank at most 5.
	By the results from Sections \ref{Rank4 - soln} and \ref{rank5-section}
	it can be represented by the atoms of size at most 2.
\end{proof}

The following two propositions say more about the minimal measure.

\begin{proposition} \label{1-atom-of-size-2}
		Let us fix a basic case relation 1, 2 or 3 given by Proposition \ref{structure-of-rank5-2} (2) and 
	denote it by $R$.
	If a sequence $\beta$ with a moment matrix $\mc{M}_2$ satisfying $R$ admits a nc measure of type 
			$(k,1)$, then
			\begin{enumerate}
				\item $2\leq k\leq 5$ if $R$ is equal to $\bY^2=\mds 1-\bX^2$ or $\bY^2=\mds 1+\bX^2$.
				\item $2\leq k\leq 6$ if $R$ is equal to $\bX\bY+\bY\bX=\mbf 0$.
			\end{enumerate}
\end{proposition}

\begin{proof}
	By assumption,
		$$\mc{M}_2=\sum_{i=1}^{k} \lambda_i \mc{M}^{(x_i,y_i)}_2+
			\xi \mc{M}^{(X,Y)}_2,$$
	where $(x_i,y_i)\in \RR^2$, $(X,Y)\in (\mathbb{SR}^{2\times 2})^2$, $k\in \NN$,
	 $\lambda_i> 0$, $\xi> 0$ and $\sum_{i=1}^{k}\lambda_i+\xi=1$. 
	Equivalently
		$$\mc{M}_2-\xi \mc{M}^{(X,Y)}_2=
			\sum_{i=1}^{k} \lambda_i \mc{M}^{(x_i,y_i)}_2.$$
	Since $\sum_{i=1}^{k} \lambda_i \mc{M}^{(x_i,y_i)}_2$ is a cm moment matrix of rank at most 6
	satisfying the relation $R$, then	 by Theorem \ref{com-case}
	\begin{enumerate}
		\item $2\leq k\leq 5$ if $R$ is equal to $\bY^2=\mds 1-\bX^2$ or $\bY^2=\mds 1+\bX^2$,
		\item $2\leq k\leq 6$ if $R$ is equal to $\bX\bY+\bY\bX=\mbf 0$,
	\end{enumerate}	
	which proves Proposition \ref{1-atom-of-size-2}.
\end{proof}

\begin{proposition} \label{exactly-1-atom-r6-bc12}
		Let us fix a basic case relation 1 or 2 given by Proposition \ref{structure-of-rank5-2} (2) and 
	denote it by $R$. If every sequence $\beta$ with 
	$\beta_X=\beta_Y=\beta_{X^3}=\beta_{X^2Y}=\beta_Y^3=0$ and 
	a moment matrix $\mc{M}_2(\beta)$ of rank 6 with column relation $R$, admits a nc measure with exactly one
	atom of size 2 and some atoms of size 1, 
	then every sequence $\widetilde\beta$ which admits a nc measure and has a moment matrix
	$\widetilde{\mc{M}}_2$ of rank 6 with the column relation $R$, admits a nc measure with exactly
	one atom of size 2 and some atoms of size 1.
\end{proposition}

\begin{proof}
	Suppose $\widetilde\beta$ admits a nc measure and
	has a moment matrix $\widetilde{\mc{M}}_2$ of rank 6
	with column relation $R$.
	By Proposition \ref{anticommute} we may assume 
	that all the atoms $(X_i,Y_i)\in (\mbb{SR}^{u_i\times u_i})^{2}$ of size $u_i>1$
	are of the form 
		$$X_{i}=\begin{pmatrix} \gamma_i I_{t_i} & B_i \\ B_i^t & -\gamma_i I_{t_i} 
			\end{pmatrix},\quad 
		Y_{i}=\begin{pmatrix} \mu_i I_{t_i} & \mbf 0 \\ \mbf 0 & -\mu_i I_{t_i} \end{pmatrix},$$
	where $\gamma_i\geq 0$, $\mu_i>0$ and $B_i$ are $t_i\times t_i$ matrices.
	Calculating $X_i^3$, $X_i^2Y_i$ and $Y_i^3$ we get
		\begin{eqnarray*}
			X_{i}^3 
				&=&\begin{pmatrix} 
				\gamma_i(\gamma_i^2 I_{t_i}+B_iB_i^t) & 
				(\gamma_i^2 I_{t_i}+B_iB_i^t)B_i \\ 
				(\gamma_i^2 I_{t_i}+B_i^tB_i)B_i^t & 
				-\gamma_i(\gamma_i^2 I_{t_i}+B_i^tB_i)\end{pmatrix},\\
			X_{i}^2 Y_i&=&\begin{pmatrix} 
				\mu_i(\gamma_i^2 I_{t_i}+B_iB_i^t) & 0\\ 
				0 & -\mu_i(\gamma_i^2 I_{t_i}+B_i^tB_i)\end{pmatrix},\\
			Y_{i}^3 
				&=&\begin{pmatrix} \mu_i^3 I_{t_i}& 0 \\ 0 & -\mu_i^3 I_{t_i}\end{pmatrix}.		
		\end{eqnarray*}
	Therefore $\sum_{i}\mc{M}^{(X_i,Y_i)}_2$ satisfies
	$\beta_X=\beta_{Y}=\beta_{X^3}=\beta_{X^2Y}=\beta_{Y^3}=0$. 
	If the rank of $\sum_{i}\mc{M}^{(X_i,Y_i)}_2$ is at most 5, then $\sum_{i}\mc{M}^{(X_i,Y_i)}_2$ can be represented
	by exactly one atom of size 2 and some atoms of size 1 by the results of previous sections.
	Else the rank of $\sum_{i}\mc{M}^{(X_i,Y_i)}_2$ is $6$ and the same conclusion follows by assumption.
\end{proof}

\subsection{Relation $\bY^2=\mds{1}-\bX^2$.} \label{r6-subs1}

In this subsection we study a nc sequence $\beta$ with a moment matrix $\mc{M}_2$ of rank 6 satisfying the relation
$\bY^2=\mds{1}-\bX^2$. 
In Theorem \ref{M(2)-bc2-r6-new1} we characterize when $\beta$ admits a nc measure.
In Corollary \ref{M(2)-bc2-r6-new1-cor} we show that the existence of a nc measure is equivalent to the feasibility problem of three linear matrix inequalities (LMIs) and rank-to-cardinality condition from Theorem \ref{com-case}.

The form of $\mc{M}_2$ is given by the following proposition.

\begin{proposition}
	Let $\beta\equiv \beta^{(4)}$ be a nc sequence with a moment matrix $\mc{M}_2$ satisfying the relation
		\begin{equation}\label{r6-rel-bc2-eq} 
			\bY^2=\mds{1}-\bX^2.
		\end{equation}
	Then $\mc{M}_2$ is of the form
		\begin{equation}\label{bc2-r6}
	\begin{mpmatrix}
 		\beta_{1} & \beta_{X} & \beta_{Y} & \beta_{X^2} & \beta_{XY} & \beta_{XY} & 			
			\beta_{1}-\beta_{X^2} \\
 		\beta_{X} & \beta_{X^2} & \beta_{XY} & \beta_{X^3} & \beta_{X^2Y} & \beta_{X^2Y} & 		
			\beta_{X}-\beta_{X^3} \\
		 \beta_{Y} & \beta_{XY} & \beta_{1}-\beta_{X^2} & \beta_{X^2Y} & \beta_{X}-\beta_{X^3} & 
		 	\beta_{X}-\beta_{X^3} & \beta_{Y}-\beta_{X^2Y} \\
 		\beta_{X^2} & \beta_{X^3} & \beta_{X^2Y} & \beta_{X^4} & \beta_{X^3Y} & 		
			\beta_{X^3Y} & \beta_{X^2}-\beta_{X^4} \\
		\beta_{XY} & \beta_{X^2Y} & \beta_{X}-\beta_{X^3} & \beta_{X^3Y} & 
			\beta_{X^2}-\beta_{X^4} & \beta_{XYXY} & \beta_{XY}-\beta_{X^3Y} \\
		 \beta_{XY} & \beta_{X^2Y} & \beta_{X}-\beta_{X^3} & \beta_{X^3Y} & \beta_{XYXY} & 		
		 	\beta_{X^2}-\beta_{X^4} & \beta_{XY}-\beta_{X^3Y} \\
		 \beta_{1}-\beta_{X^2} & \beta_{X}-\beta_{X^3} & \beta_{Y}-\beta_{X^2Y} & 
			\beta_{X^2}-\beta_{X^4} & 
		\beta_{XY}-\beta_{X^3Y} & \beta_{XY}-\beta_{X^3Y} & \beta_{1}-2 \beta_{X^2}+\beta_{X^4} 
	\end{mpmatrix}.
		\end{equation}
\end{proposition}

\begin{proof}
	The relation (\ref{r6-rel-bc2-eq}) gives us the following system in $\mc{M}_2$
		\begin{multicols}{2}
		\begin{equation}\label{eq-bc2-r6}
			\begin{aligned}
		\beta_{Y^2}	= \beta_{1}-\beta_{X^2},\\
		\beta_{XY^2}	= \beta_{X}-\beta_{X^3},\\
		\beta_{Y^3}	= \beta_{Y}-\beta_{X^2Y},
			\end{aligned}
		\end{equation}
	\vfill
	\columnbreak
		\begin{equation*}
			\begin{aligned}
		\beta_{X^2Y^2}	= \beta_{X^2}-\beta_{X^4},\\
		\beta_{XY^3} 	= \beta_{XY}-\beta_{X^3Y},\\
		\beta_{Y^4} 	= \beta_{Y^2}-\beta_{X^2Y^2}.
			\end{aligned}
		\end{equation*}
	\end{multicols}
	\noindent Plugging in the expressions for $\beta_{Y^2}$ and $\beta_{X^2Y^2}$ in the expression for 
	$\beta_{Y^4}$ gives the form (\ref{bc2-r6}) of $\mc{M}_2$.
\end{proof}

The following theorem characterizes normalized nc sequences $\beta$ with a moment matrix $\mc{M}_2$ of rank 6 
satisfying the relation $\bY^2=\mds 1-\bX^2$, which admit a nc measure.

\begin{theorem} \label{M(2)-bc2-r6-new1}
	Suppose $\beta\equiv \beta^{(4)}$ is a normalized nc sequence with a moment matrix $\mc{M}_2$ of rank 6 satisfying the relation
	$\bY^2=\mds 1-\bX^2$. Then $\beta$ admits a nc measure if and only if $\mc{M}_2$ is positive semidefinite and
	one of the following is true:
		\begin{enumerate}
			\item[(1)] $\beta_{X}=\beta_Y=\beta_{X^3}=\beta_{X^2Y}=0$.
				Moreover, there exists a nc measure of type (4,1).
			\item[(2)] There exist 
				$$a_1\in (0,1),\quad 	a_2\in \left(-2\sqrt{a_1(1-a_1)}, 2\sqrt{a_1(1-a_1)}\right)$$
				such that 
					$$M:=\mc{M}_2-\xi\mc{M}^{(X,Y)}_2$$ 
				is a positive semidefinite cm moment matrix satisfying $\Rank M\leq\Card \mathcal V_M$, where
				$\mathcal V_M$ is the variety associated to $M$ 
				(as in Theorem \ref{com-case}),
					\begin{equation}\label{(X,Y)-form}
						X=\begin{pmatrix} \sqrt{a_1} & 0 \\
										0 & -\sqrt{a_1} \end{pmatrix},\quad
						Y=\sqrt{(1-a_1)} \begin{pmatrix} \frac{a}{2} & \frac{1}{2}\sqrt{4-a^2}\\
							\frac{1}{2}\sqrt{4-a^2} & -\frac{a}{2}
							\end{pmatrix},
					\end{equation}
					$$a=\frac{a_2}{\sqrt{a_1(1-a_1)}},$$
				and $\xi>0$ is the smallest positive number such that the rank of 
				$\mc{M}_2-\xi\mc{M}^{(X,Y)}_2$ is smaller than the rank of $\mc{M}_2$.
		\end{enumerate}
\end{theorem}

\begin{proof}
	First we will prove (1).
	$\mc{M}_2$ is of the form
		$$\begin{mpmatrix}
					 1 & 0 & 0 & \beta_{X^2} & \beta_{XY} & \beta_{XY} &  1-\beta_{X^2}		\\
					0 & \beta_{X^2} & \beta_{XY} & 0 & 0 & 0 & 0 					\\
					0& \beta_{XY} & 1-\beta_{X^2} & 0 & 0 & 0 & 0			\\
					\beta_{X^2} & 0 & 0 & \beta_{X^4} & \beta_{X^3Y} & \beta_{X^3Y} & 
						\beta_{X^2}-\beta_{X^4}	\\
					\beta_{XY} & 0 & 0 & \beta_{X^3Y} & \beta_{X^2}-\beta_{X^4} & \beta_{XYXY} & 
						\beta_{XY}-\beta_{X^3Y}					\\
					\beta_{XY} & 0 & 0 & \beta_{X^3Y} & \beta_{XYXY} & \beta_{X^2}-\beta_{X^4} & 
						\beta_{XY}-\beta_{X^3Y}				\\
					1-\beta_{X^2} & 0 & 0 & \beta_{X^2}-\beta_{X^4} & \beta_{XY}-\beta_{X^3Y}	 & 
						\beta_{XY}-\beta_{X^3Y}	 & 1-2\beta_{X^2}+\beta_{X^4}
				\end{mpmatrix}.$$
	We define the matrix function
		$$B(\alpha):=\mc{M}_2-\alpha \big(\mc{M}_2^{(1,0)}+\mc{M}_2^{(-1,0)}\big).$$
	We have that
		$$
			B(\alpha)=\begin{mpmatrix}
					 1-2\alpha & 0 & 0 & \beta_{X^2}-2\alpha & \beta_{XY} & \beta_{XY} & D	\\			
					0 & \beta_{X^2}-\alpha & \beta_{XY} & 0 & 0 & 0 & 0 					\\
					0& \beta_{XY} & D & 0 & 0 & 0 & 0				\\
					\beta_{X^2}-2\alpha & 0 & 0 & \beta_{X^4}-2\alpha & \beta_{X^3Y} & \beta_{X^3Y} &  C
						\\
					\beta_{XY} & 0 & 0 & \beta_{X^3Y} & C & E & \beta_{XY}-\beta_{X^3Y}				\\
					\beta_{XY} & 0 & 0 & \beta_{X^3Y} & E &  
						C & \beta_{XY}-\beta_{X^3Y}			\\
					D & 0 & 0 & C & \beta_{XY}-\beta_{X^3Y}	 & \beta_{XY}-\beta_{X^3Y}	 & D-C
				\end{mpmatrix},$$
	where 
		$$C=\beta_{X^2}-\beta_{X^4},\quad D=1-\beta_{X^2},\quad E=\beta_{XYXY}.$$
	We have that
		\begin{equation}\label{det-eq-r6-c1}
		\det\big(B(\alpha)|_{\{\mds 1, \bX, \bY, \bX^2, \bX\bY,\bY\bX\}}\big)=
		-(\beta_{XYXY}-\beta_{X^2}+\beta_{X^4})(-\beta_{XY}^2+\beta_{X^2}-\beta_{X^2}^2+2\alpha(-1+\beta_{X}^2))(-F+2\alpha G),
		\end{equation}
	where 
		\begin{align*}
			F &= \beta_{XYXY}(\beta_{X^2}^2-\beta_{X^4})+\beta_{X^2}
				(\beta_{X^2}^2-4\beta_{XY}\beta_{X^3Y}-\beta_{X^4}(1+\beta_{X^2}))+ 
					2\beta_{X^3Y}^2+ \beta_{X^4}(\beta_{X^4}+ 2\beta_{XY}^2),\\
			G &= 2\beta_{XY}(\beta_{XY}-2\beta_{X^3Y})+\beta_{XYXY}(2\beta_{X^2}-1-\beta_{X^4})+
				\beta_{X^2}(2\beta_{X^2}-1-3\beta_{X^4})
				+ 2\beta_{X^3Y}^2+\beta_{X^4}(1+\beta_{X^4}).
		\end{align*}
	Let $\alpha_0>0$ be the smallest positive number such that the rank of $B(\alpha_0)$ is smaller than 6.
	By (\ref{det-eq-r6-c1})
	we get
		$$\alpha_0=\min\Big(\frac{\beta_{XY}^2-\beta_{X^2}+\beta_{X^2}^2}{2(-1+\beta_{X^2})}, \frac{F}{2G}\Big).$$
		
	\noindent \textbf{Claim 1:} $\alpha_0=\frac{F}{2G}<\alpha_1$.\\
	
	Since 
		\begin{eqnarray*}
			\det\big(B(\alpha)|_{\{\mds 1, \bX^2\}}\big)&=& \beta_{X^4}-\beta_{X^2}^2+2\alpha(-1+2\beta_{X^2}-\beta_{X^4}),\\
			\det\big(B(\alpha)|_{\{\mds 1, \bX\bY\}}\big)&=& C-\beta^2_{XY}-2\alpha\cdot C,\\
			\det\big(B(\alpha)|_{\{\mds 1, \bX\bY,\bY\bX\}}\big)&=& (E-C)(-E-C+2\beta_{X^2Y}+2\alpha\cdot (E+C)),
		\end{eqnarray*}
	the system 
		$$\det\big(B(\alpha_2)|_{\{\mds 1, \bX^2\}}\big)=0,\quad 
			\det\big(B(\alpha_3)|_{\{\mds 1, \bX\bY\}}\big)=0,\quad
			\det\big(B(\alpha_4)|_{\{\mds 1, \bX\bY,\bY\bX\}}\big)=0$$
	has a solution
		\begin{align*}
			\alpha_2=\frac{\beta_{X^2}^2-\beta_{X^4}}{2(-1+2\beta_{X^2}-\beta_{X^4})},\quad
			\alpha_3=\frac{-\beta_{XY}^2+\beta_{X^2}-\beta_{X^4}}{2(\beta_{X^2}-\beta_{X^4})},\quad
			\alpha_4=
			\frac{-2\beta_{XY}^2+\beta_{XYXY}+\beta_{X^2}-\beta_{X^4}}{2(\beta_{XYXY}+
				\beta_{X^2}-\beta_{X^4})}.
		\end{align*}
	If $\alpha_1\leq  \frac{F}{2G}$,
	then since $B(\alpha_{1})\succeq 0$, it follows that $\alpha_1\leq \min(\alpha_2,\alpha_3,\alpha_4)$.
	Using \textit{Mathematica}, the system
		\begin{align}
			\alpha_1\leq \min(\alpha_2,\alpha_3,\alpha_4),\quad 
		\det\big({\mc{M}_2}|_{\{\bY\}}\big)>0,\quad \det\big({\mc{M}_2}|_{\{\bX\bY\}}\big)>0,\\
		\det\big({\mc{M}_2}|_{\{\bX,\bY\}}\big)>0,\quad \det\big({\mc{M}_2}|_{\{1,X^2\}}\big)>0,\quad
		\det\big({\mc{M}_2}|_{\{1,\bX\bY,\bY\bX\}}\big)>0,
		\end{align}
	does not have solutions (see \url{https://github.com/Abhishek-B/TTMP} for the Mathematica file). Hence $\alpha_0=\frac{F}{2G}< \alpha_1$.\\
	
	Using \textit{Mathematica} to calculate the kernel of $B(\frac{F}{2G})$ we conclude that 
	$B(\frac{F}{2G})$ satisfies the relations
		$$\bX\bY+\bY\bX=a\mds 1+ d\bX^2, \quad \bY^2+\bX^2=\mds 1$$
	for some $a,d\in \RR$. We also have
		$$\beta^{(B)}_{X}=\beta^{(B)}_Y=\beta^{(B)}_{X^3}=\beta^{(B)}_{X^2Y}=		
								\beta^{(B)}_{XY^2}=\beta^{(B)}_{Y^3}=0,$$
	where $\beta^{(B)}_{w(X,Y)}$ are the moments of $B(\frac{F}{2G})$.
	This is a special case in the proof of Proposition \ref{structure-of-rank5-2}, i.e., Case 2.2. Following the proof
	we see that after using only transformations of type	
		$$(x,y)\mapsto (\alpha_1 x+ \beta_1 y, \alpha_2 x+ \beta_2 y)$$
	for some $\alpha_1,\alpha_2,\beta_1,\beta_2\in \RR$,
	we come into the basic case 1 or 2 of rank 5 with
		$\widetilde\beta_X=\widetilde\beta_Y=\widetilde\beta_{X^3}=0$. 
	But every such sequence admits a measure of type (2,1)
	by Theorems \ref{M(2)-XY+YX=0-bc1} and \ref{M(2)-XY+YX=0}.
	Hence $\beta$ admits a measure of type $(4,1)$. \\	
		
	It remains to prove (2). Suppose that $\beta$ admits a nc measure.
	By Proposition \ref{exactly-1-atom-r6-bc12} and Theorem 
	\ref{M(2)-bc2-r6-new1} (1),
		\begin{equation} \label{r6-M(2)-with-rank4-bc4-new22} 
			\mc{M}_2=\sum_i \lambda_i \mc{M}^{(x_i,y_i)}(2) +  \xi \mc{M}^{(X,Y)}(2),
		\end{equation}
	where $(x_i,y_i)\in \RR^2$, $(X,Y)\in (\mathbb{SR}^{2\times 2})^2$,
	$\lambda_i> 0$, $\xi> 0$ and $\sum_i \lambda_i+\xi=1$.
	Therefore 
		$$\mc{M}_2-\xi \mc{M}^{(X,Y)}_2$$
	is a cm moment matrix satisfying the relations
		$$\bY^2=\mds 1-\bX^2\quad\text{and}\quad \bX\bY=\bY\bX.$$
	By Theorem \ref{com-case}, $M$ admits a measure if and only if $M$ is psd and satisfies
	$\Rank M\leq\Card \mathcal V_M$.
	To conclude the proof it only 
	remains to prove that $X, Y$ are of the form (\ref{(X,Y)-form}).  $\mc{M}^{(X,Y)}_2$ is a nc moment matrix 	
	rank 4. Therefore the columns
	$\{\mds 1,\bX,\bY,\bX\bY\}$ are linearly independent and hence
		\begin{equation*}
			\bX^2=a_1 \mds 1+ b_1 \bX+ c_1 \bY+d_1 \bX\bY,\quad\text{and}\quad
			\bY^2=a_3 \mds1+ b_3 \bX+ c_3 \bY+d_3 \bX\bY,
		\end{equation*}
	where $a_j,b_j,c_j,d_j\in \RR$ for $j=1, 3$. By Theorem  \ref{rank4-soln-aljaz} (\ref{point-1-rank4}),
	$d_1=d_3=0$. 
	By Theorem \ref{rank4-soln-aljaz} (\ref{point-3-rank4}), 
	$c_1=b_3=0$. Since $\bX^2+\bY^2=\mds 1$ it follows that $b_1=c_3=0$ and $a_3=1-a_1$.
	By Theorem \ref{rank4-soln-aljaz} (\ref{point-4-rank4}), $X$ and $Y$ are of the form 
	 (\ref{(X,Y)-form}).
\end{proof}
	
The following theorem translates the BQTMP for $\beta$ with $\mc{M}_2$ of rank 6 satisfying $\bY^2=\mds 1-\bX^2$ into the feasibility problem of some LMIs and a rank-to-cardinality condition from Theorem \ref{com-case}.

\begin{corollary} \label{M(2)-bc2-r6-new1-cor}
	Suppose $\beta\equiv \beta^{(4)}$ is a normalized nc sequence with a moment matrix $\mc{M}_2$ of rank 6 satisfying the relation
	$\bY^2=\mds 1-\bX^2$. 
	Let $L(a,b,c,d,e)$ be the following linear matrix polynomial 
		$$
				\begin{mpmatrix}
					a & \beta_X & \beta_Y & b & c & c &  a-b		\\
					\beta_X & b & c & \beta_{X^3} & \beta_{X^2Y} & \beta_{X^2Y} & \beta_X-\beta_{X^3} 					\\
					\beta_{Y}& c & a-b & \beta_{X^2Y} & \beta_X-\beta_{X^3} & \beta_X-\beta_{X^3} & 
					\beta_Y-\beta_{X^2Y}			\\
					b & \beta_{X^3} & \beta_{X^2Y} & d & e & e & b-d	\\
					c & \beta_{X^2Y} & \beta_X-\beta_{X^3} & e & b-d & b-d & c-e					\\
					c & \beta_{X^2Y} & \beta_X- \beta_{X^3}& e & b-d & b-d & c-e 					\\
					a-b & \beta_X-\beta_{X^3} & \beta_{Y}-\beta_{X^2Y} & b-d & c-e & c-e & a-2b+d
				\end{mpmatrix},$$
	where $a,b,c,d,e\in \RR$.
	Then $\beta$ admits a nc measure if and only if there exist $a,b,c,d,e\in \RR$ such that
	\begin{enumerate}
		\item\label{point1-bc1-r6} $L(a,b,c,d,e)\succeq 0$,
		\item\label{point2-bc1-r6} $\mc{M}_2-L(a,b,c,d,e)\succeq 0$,
		\item\label{point3-bc1-r6} 
			$(\mc{M}_2-L(a,b,c,d,e))_{\{\mds 1, \bX, \bY, \bX\bY\}} \succ 0$,
		\item\label{point4-bc1-r6} $\Rank(L(a,b,c,d,e))\leq \Card \mathcal V_L$, where 
			$\mathcal V_L$ is the variety associated to the moment matrix $L(a,b,c,d,e)$ 
			(see Theorem \ref{com-case}).
	\end{enumerate}
\end{corollary}
	
\begin{proof}
	By Theorem \ref{M(2)-bc2-r6-new1}, $\beta$ admits a nc measure if and only 
	if 
		\begin{equation}\label{with-2-times-2-matrices-bc1}
			\mc{M}_2=\sum_{i=1}^k\lambda_i \mc{M}^{(x_i,y_i)}_2+
			\xi\mc{M}^{(X,Y)}_2,
		\end{equation}
	where $(x_i,y_i)\in \RR^2$, $(X,Y)\in (\mathbb{SR}^{2\times 2})^2$,
	$\lambda_i> 0$, $\xi>0$ and $\sum_i \lambda_i+\xi=1$.
	By Corollary \ref{nc-TTMM-cor},
		\begin{equation} \label{r6-form-of-x2-2-bc4-new1-cor} 
			\beta^{(X,Y)}_X=\beta^{(X,Y)}_Y=\beta^{(X,Y)}_{X^3}=
			\beta^{(X,Y)}_{X^2Y}=\beta^{(X,Y)}_{XY^2}=\beta^{(X,Y)}_{Y^3}=0,
		\end{equation}
	where $\beta^{(X,Y)}_{w(X,Y)}$ are the moments of $\mc{M}^{(X,Y)}_2$.
	Using (\ref{with-2-times-2-matrices-bc1}) and (\ref{r6-form-of-x2-2-bc4-new1-cor}), we conclude that
	$\sum_i \lambda_i \mc{M}^{(x_i,y_i)}_2$ and $\xi \mc{M}^{(X,Y)}_2$ are of the forms
			\begin{align}
				\begin{mpmatrix}
					a & \beta_X & \beta_Y & b & c & c &  a-b		\\
					\beta_X & b & c & \beta_{X^3} & \beta_{X^2Y} & \beta_{X^2Y} & \beta_X-\beta_{X^3} 					\\
					\beta_{Y}& c & a-b & \beta_{X^2Y} & \beta_X-\beta_{X^3} & \beta_X-\beta_{X^3} & \beta_Y-\beta_{X^2Y}			\\
					b & \beta_{X^3} & \beta_{X^2Y} & d & e & e & b-d	\\
					c & \beta_{X^2Y} & \beta_X-\beta_{X^3} & e & b-d & b-d & c-e					\\
					c & \beta_{X^2Y} & \beta_X- \beta_{X^3}& e & b-d & b-d & c-e 					\\
					a-b & \beta_X-\beta_{X^3} & \beta_{Y}-\beta_{X^2Y} & b-d & c-e & c-e & a-2b+d
				\end{mpmatrix}, \label{matrix-1-bc1}\\
				\begin{mpmatrix}
					1-a & 0 & 0 & \beta_{X^2}-b & A_1(c) & A_1(c) &  A_2(a,b)	\\
					0 & \beta_{X^2}-b & A_1(c) & 0 & 0 & 0 & 0 					\\
					0 & A_1(c) & A_2(a,b) & 0 & 0 & 0 & 0			\\
					\beta_{X^2}-b & 0 & 0 & \beta_{X^4}-d & A_3(e) & A_3(e) & A_4(b,d)\\
					A_1(c) & 0 & 0 & A_3(e) & A_4(b,d) & \beta_{XYXY}-(b-d)& 					A_5(c,e)\\
					A_1(c) & 0 & 0 & A_3(e) & \beta_{XYXY}-(b-d) & A_4(b,d) & A_5(c,e)					\\
					A_2(a,b) & 0 & 0 & A_4(b,d) & A_5(c,e) &  A_5(c,e)  & A_6(a,b,d)
				\end{mpmatrix},\label{matrix-2-bc1}
			\end{align}
	where 
		\begin{eqnarray*}
		A_1(c)&=& \beta_{XY}-c, \quad A_2(a,b)=1-\beta_{X^2}-(a-b),\\ 
		A_3(e) &=& \beta_{X^3Y}-e,\quad A_4(b,d)=\beta_{X^2}-\beta_{X^4}-(b-d),\\
		A_5(c,e)&=& \beta_{XY}-\beta_{X^3Y}-(c-e),\quad  A_6(a,b,d)=1-2\beta_{X^2}+\beta_{X^4}-(a-2b+d),
		\end{eqnarray*}
	for some $a,b,c,d,e\in \RR$.
	Notice that the matrix (\ref{matrix-1-bc1}) equals to $L(a,b,c,d,e)$
	and the matrix (\ref{matrix-2-bc1}) to $\mc{M}_2-L(a,b,c,d,e)$.
	Since $L(a,b,c,d,e)$ is a cm moment matrix, it admits a nc measure by Theorem \ref{com-case} if and only if 
	(\ref{point1-bc1-r6}) and (\ref{point4-bc1-r6}) of Theorem \ref{M(2)-bc2-r6-new1-cor} are
	true. Since $\mc{M}_2-L(a,b,c,d,e)$ is a nc moment matrix satisfying 
	$\bY^2=\mds 1-\bX^2$ and 
	$\widetilde\beta_X=\widetilde\beta_Y=\widetilde\beta_{X^3}=\widetilde\beta_{X^2Y}=\widetilde\beta_{XY^2}=
	\widetilde\beta_{Y^3}=0$, it admits a nc measure by the results of rank 4 and 5 cases
	and Theorem \ref{M(2)-bc2-r6-new1} (1)
	if and only if (\ref{point2-bc1-r6}) and 
	(\ref{point3-bc1-r6}) of Theorem \ref{M(2)-bc2-r6-new1-cor} are true. 
\end{proof}


\subsection{ Relation $\bX\bY+\bY\bX=\mbf{0}$.} \label{r6-subs2}

In this subsection we study a nc sequence $\beta\equiv\beta^{(4)}$ with a moment matrix $\mc{M}_2$ of rank 6 satisying the relation 
$\bX\bY+\bY\bX=\mbf 0$. In Theorem \ref{M(2)-XY+YX=0-bc4-r6} we characterize when 
$\beta$ admits a nc measure. In Corollary \ref{r6-bc3} we show that the existence of a nc measure is equivalent to the feasibility problem of three LMIs and a rank-to-variety condition from Theorem \ref{com-case}.

The form of $\mc{M}_2$ is given by the following proposition.

\begin{proposition}\label{r6-rel-bc4}
	Let $\beta\equiv \beta^{(4)}$ be a nc sequence with a moment matrix
	$\mc{M}_2$ of rank 6 satisfying the relation
		\begin{equation}\label{r6-rel-bc4-eq} 
			\bX\bY+\bY\bX=\mbf{0}
		\end{equation}
	Then $\mc{M}_2$ is of the form
		\begin{equation}\label{bc4-r6}
		\begin{mpmatrix}
 		\beta_{1} & \beta_{X} & \beta_{Y} & \beta_{X^2} & 0 & 0 & 			\beta_{Y^2} 	\\
 		\beta_{X} & \beta_{X^2} & 0 & \beta_{X^3} & 0 & 0 & 		0 	\\
		 \beta_{Y} & 0 & \beta_{Y^2} & 0 & 0 & 0 & \beta_{Y^3}\\
		 \beta_{X^2} & \beta_{X^3} & 0 & \beta_{X^4} & 0 & 0 & \beta_{X^2Y^2} \\
		0 & 0 & 0  & 0 & \beta_{X^2Y^2} & -\beta_{X^2Y^2} & 0 \\
		0 & 0 & 0 & 0 & -\beta_{X^2Y^2}  & \beta_{X^2Y^2} & 0 \\
		\beta_{Y^2} & 0 & \beta_{Y^3} & \beta_{X^2Y^2} & 0 & 0 & \beta_{Y^4}
	\end{mpmatrix}.
	\end{equation}
\end{proposition}

\begin{proof}
	The relation (\ref{r6-rel-bc4-eq}) gives us the following system in $\mc{M}_2$
		\begin{multicols}{2}
		\begin{equation}\label{eq-bc4-r6}
			\begin{aligned}
		2\beta_{XY}	= 0,\\
		2\beta_{X^2Y}	= 0,\\
		2\beta_{XY^2}	= 0,
			\end{aligned}
		\end{equation}
	\vfill
	\columnbreak
		\begin{equation*}
			\begin{aligned}
		2\beta_{X^3Y}	= 0,\\
		\beta_{X^2Y^2}+\beta_{XYXY}	= \beta_{XY},\\
		2\beta_{XY^3}	=0.
			\end{aligned}
		\end{equation*}
	\end{multicols}
	\noindent Thus the solution of the system (\ref{eq-bc4-r6}) is given by the statement of the proposition.
\end{proof}

The following theorem characterizes normalized sequences $\beta$ with a moment matrix 
$\mc{M}_2$ of rank 6 satisfying $\bX\bY+\bY\bX=\mbf 0$, which admit a nc measure.

\begin{theorem} \label{M(2)-XY+YX=0-bc4-r6}
	Suppose $\beta\equiv \beta^{(4)}$ is a normalized nc sequence with a moment matrix $\mc{M}_2$ of rank 6 satisfying the relation
	$\bX\bY+\bY\bX=\mbf 0$. 
	Then $\beta$ admits a nc measure if and only if $\mc{M}_2$ is positive semidefinite and one of the following is true:
		\begin{enumerate}
			\item $\beta_{X}=\beta_Y=\beta_{X^3}=\beta_{Y^3}=0$. There exists a nc measure of type (2,1) or (3,1).
			\item There exist
				$$a_1>0,\quad a_3>0,$$
				such that
					$$M:=\mc{M}_2-\xi\mc{M}^{(X,Y)}_2$$ 
				is a positive semidefinite cm moment matrix satisfying $\Rank M\leq\Card \mathcal V_M$, where
				$\mathcal V_M$ is the variety associated to $M$ 
					(as in Theorem \ref{com-case}),
					\begin{equation}\label{(X,Y)-form-bc2}
						X=\begin{pmatrix} \sqrt{a_1} & 0 \\
										0 & -\sqrt{a_1} \end{pmatrix},\quad
						Y=\begin{pmatrix} 0 & \sqrt{a_3} \\\
							\sqrt{a_3}  & 0
							\end{pmatrix}
					\end{equation}
				and $\xi>0$ is the smallest positive number such that rank of 
				$\mc{M}_2-\xi\mc{M}^{(X,Y)}_2$ is smaller than the rank of $\mc{M}_2$. 
		\end{enumerate}
\end{theorem}

\begin{proof}
	First we will prove (1).
	$\mc{M}_2$ is of the form
		$$\mc{M}_2=\begin{mpmatrix}
					 1 & 0 & 0 & \beta_{X^2} & 0 & 0 &  \beta_{Y^2}		\\
					0 & \beta_{X^2} & 0 & 0 & 0 & 0 & 0 					\\
					0& 0 & \beta_{Y^2} & 0 & 0 & 0 & 0			\\
					\beta_{X^2} & 0 & 0 & \beta_{X^4} & 0 & 0 & \beta_{X^2Y^2}	\\
					0 & 0 & 0 & 0 & \beta_{X^2Y^2} &  -\beta_{X^2Y^2} & 0					\\
					0 & 0 & 0 & 0 & -\beta_{X^2Y^2} & \beta_{X^2Y^2} & 0 					\\
					\beta_{Y^2} & 0 & 0 & \beta_{X^2Y^2} & 0 & 0 & \beta_{Y^4}
				\end{mpmatrix}.$$
	We define the matrix function
		$$B(\alpha):=\mc{M}_2-\alpha \mc{M}^{(0,0)}_2.$$
	We have that
		\begin{equation}\label{det-eq-r6-c2}
		\det\big(B(\alpha)|_{\{\mds 1, \bX, \bY, \bX^2, \bX\bY,\bY^2\}}\big)=
		-\beta_{Y^2}\beta_{X^2}\beta_{X^2Y^2}(F-\alpha G),
		\end{equation}
	where 
		\begin{equation*}
			F = \beta_{Y^4}\beta_{X^2}^2-2\beta_{Y^2}\beta_{X^2}\beta_{X^2Y^2}+
			\beta_{X^2Y^2}+\beta_{Y^2}^2\beta_{X^4}-\beta_{X^4}\beta_{Y^4},\quad 
			G = \beta_{X^2Y^2}^2-\beta_{X^4}\beta_{Y^4}.
		\end{equation*}
	Let $\alpha_0>0$ be the smallest positive number such that the rank of $B(\alpha_0)$ is smaller than 6.
	By (\ref{det-eq-r6-c2})
	we get
		$\alpha_0=\frac{F}{G}.$
	The kernel of $B(\alpha_0)$ satisfies the relations
		$$\bY^2=\frac{\beta_{X^4}\beta_{Y^4}-\beta_{X^2Y^2}^2}{\beta_{Y^2}\beta_{X^4}-\beta_{X^2}\beta_{X^2Y^2}}\mds 1+\frac{\beta_{Y^2}\beta_{X^2Y^2}-\beta_{Y^4}\beta_{X^2}}{\beta_{Y^2}\beta_{X^4}-\beta_{X^2}\beta_{X^2Y^2}}\bX^2,\quad \bX\bY+\bY\bX=\mbf 0.$$
	It also satisfies
		$$\beta_X^{(B)}=\beta_Y^{(B)}=\beta_{XY}^{(B)}=
			\beta_{X^3}^{(B)}=\beta_{X^2Y}^{(B)}=\beta_{XY^2}^{(B)}=
			\beta_{Y^3}^{(B)}=0,$$
	where $\beta_{w(X,Y)}^{(B)}$ are the moments of $B(\alpha_0).$
	This is a special case in the proof of Proposition \ref{structure-of-rank5-2} 
	(\ref{point-1-str-rank5}) , i.e.,
	Case 2.3. Following the proof we see that after using only transformations of type
		$$(x,y)\mapsto (\alpha_1 x+\beta_1 y,\alpha_2x+\beta_2 y)$$
	for some $\alpha_1,\alpha_2,\beta_1,\beta_2\in \RR$, we come into one of 
	the basic pairs 1 or 4 of rank 5 with $\widetilde\beta_X=\widetilde\beta_Y=\widetilde\beta_{X^3}=0$.
	But every such moment matrix admits a measure of type (1,1) or (2,1) by Theorems 
	\ref{M(2)-XY+YX=0-bc1} and \ref{M(2)-XY+YX=0-bc4}. 
	Hence $\mc{M}_2$ admits
	a nc measure of type (2,1) or (3,1).
	
	It remains to prove (2). Suppose that $\beta$ admits a nc measure.
	By Propositions \ref{exactly-1-atom-r6-bc12} and Theorem
	\ref{M(2)-XY+YX=0-bc4-r6} (1), 
		\begin{equation*}
			\mc{M}_2=\sum_i \lambda_i \mc{M}^{(x_i,y_i)}_2 +  \xi \mc{M}^{(X,Y)}_2,
		\end{equation*}
	where $(x_i,y_i)\in \RR^2$, $(X,Y)\in (\mathbb{SR}^{2\times 2})^2$,
	$\lambda_i> 0$, $\xi> 0$ and $\sum_i \lambda_i+\xi=1$.
	Therefore 
		$$M:=\mc{M}_2-\xi \mc{M}^{(X,Y)}_2$$
	is a cm moment matrix satisfying the relations
		$$\bX\bY+\bY\bX=\mbf 0\quad\text{and}\quad \bX\bY=\bY\bX.$$
	By Theorem \ref{com-case}, $M$ admits a nc measure if and only if $M$ is psd and satisfies
	$\Rank M\leq\Card \mathcal V_M$.
	To conclude the proof it only 
	remains to prove that $X, Y$ are of the form (\ref{(X,Y)-form-bc2}).  
	$\mc{M}^{(X,Y)}_2$ is a nc moment matrix of rank 4. 
	 Therefore the columns
	$\{\mds 1,\bX,\bY,\bX\bY\}$ are linearly independent and hence
		$$\bX^2=a_1 \mds 1+ b_1 \bX+ c_1 \bY+d_1 \bX\bY\quad\text{and}\quad
			\bY^2=a_3 \mds 1+ b_3 \bX+ c_3 \bY+d_3 \bX\bY.$$
	where $a_j,b_j,c_j,d_j\in \RR$ for $j=1,2, 3$. By Theorem  \ref{rank4-soln-aljaz} (\ref{point-1-rank4}),
	$d_1=d_3=0$. 
	By Theorem \ref{rank4-soln-aljaz} (\ref{point-3-rank4}), 
	$c_1=b_3=0$. Since $\bX\bY+\bY\bX=\mbf 0$ it follows that $b_1=c_3=0$.
	By Theorem \ref{rank4-soln-aljaz} (\ref{point-4-rank4}), $X$ and $Y$ are of the form 
	 (\ref{(X,Y)-form-bc2}).
\end{proof}

The following corollary translates the BQTMP for $\beta$ with $\mc{M}_2$ of rank 6 satisying 
$\bX\bY+\bY\bX=\mbf 0$ into the feasibility problem of some LMIs and a rank-to-variety condition from Theorem 
\ref{com-case}.

\begin{corollary}\label{r6-bc3}
	Suppose $\beta\equiv \beta^{(4)}$ is a normalized nc sequence with a moment matrix $\mc{M}_2$ of rank 6 satisfying the relation
	$\bX\bY+\bY\bX=\mbf 0$. 
	Let us define a linear matrix polynomial
				$$
				L(a,b,c,d,e)=\begin{mpmatrix}
					a & \beta_X & \beta_Y & b & 0 & 0 &  c		\\
					\beta_X & b & 0 & \beta_{X^3} & 0 & 0 & 0 					\\
					\beta_{Y}& 0 & c & 0  & 0 & 0 & \beta_{Y^3}		\\
					b & \beta_{X^3} & 0 & d & 0 & 0 & 0	\\
					0 & 0 & 0 & 0 & 0 & 0 & 0				\\
					0 & 0 & 0 & 0 & 0 & 0 & 0				\\
					c & 0 & \beta_{Y^3} & 0 & 0 & 0 & e
				\end{mpmatrix},$$
	where $a,b,c,d,e\in \RR$. 
	Then $\beta$ admits a nc measure if and only there exist
		\begin{equation}\label{par-a,b,c,d,e}
			a\in (0,1), \quad b\in (0,\beta_{X^2}),\quad  c\in (0,\beta_{Y^2}), \quad 
			d\in (0,\beta_{X^4}),\quad
			e\in (0,\beta_{Y^4}),
		\end{equation}
	such that
		\begin{enumerate}
			\item\label{point1-bc2-r6} $L(a,b,c,d,e)\succeq 0$,
			\item\label{point2-bc2-r6} $\mc{M}_2-
				L(a,b,c,d,e)\succeq 0,$
			\item\label{point3-bc2-r6} $\Rank(L(a,b,c,d,e))\leq \Card \mathcal V_L$, where 
			$\mathcal V_L$ is the variety associated to the moment matrix $L(a,b,c,d,e)$ 
			(see Theorem \ref{com-case}).
		\end{enumerate}
\end{corollary}
	
\begin{proof}
	By Theorem \ref{M(2)-XY+YX=0-bc4-r6}, $\beta$ admits a nc measure if and only 
	if 
		\begin{equation}\label{with-2-times-2-matrices-BC3}
			\mc{M}_2=\sum_{i=1}^k\lambda_i \mc{M}^{(x_i,y_i)}_2+
			\xi\mc{M}^{(X,Y)}_2,
		\end{equation}
	where $(x_i,y_i)\in \RR^2$, $(X,Y)\in (\mathbb{SR}^{2\times 2})^2$,
	$\lambda_i>0$, $\xi>0$ and $\sum_i \lambda_i+\xi=1$.	
	By Corollary \ref{nc-TTMM-cor},
		\begin{equation} \label{r6-form-of-x2-2-bc4-new1-cor-BC3} 
			\beta^{(X,Y)}_X=\beta^{(X,Y)}_Y=\beta^{(X,Y)}_{X^3}=
			\beta^{(X,Y)}_{X^2Y}=\beta^{(X,Y)}_{XY^2}=\beta^{(X,Y)}_{Y^3}=0,
		\end{equation}
	where $\beta^{(X,Y)}_{w(X,Y)}$ are the moments of $\mc{M}^{(X,Y)}_2$.
	Since $\bX\bY+\bY\bX=\mbf 0$, we also have 
		$\beta^{(X,Y)}_{XY}=0.$
	Using (\ref{with-2-times-2-matrices-BC3}) and (\ref{r6-form-of-x2-2-bc4-new1-cor-BC3}), we conclude that
	$\sum_i \lambda_i \mc{M}^{(x_i,y_i)}_2$ and $\xi \mc{M}^{(X,Y)}_2$ are of the forms
			\begin{align}
				\begin{mpmatrix}
					a & \beta_X & \beta_Y & b & 0 & 0 &  c		\\
					\beta_X & b & 0 & \beta_{X^3} & 0 & 0 & 0 					\\
					\beta_{Y}& 0 & c & 0  & 0 & 0 & \beta_{Y^3}		\\
					b & \beta_{X^3} & 0 & d & 0 & 0 & 0	\\
					0 & 0 & 0 & 0 & 0 & 0 & 0				\\
					0 & 0 & 0 & 0 & 0 & 0 & 0				\\
					c & 0 & \beta_{Y^3} & 0 & 0 & 0 & e
				\end{mpmatrix}, \label{matrix-1-BC3}\\
				\begin{mpmatrix}
					1-a & 0 & 0 & \beta_{X^2}-b & 0 & 0 &  \beta_{Y^2}-c	\\
					0 & \beta_{X^2}-b & 0 & 0 & 0 & 0 & 0 					\\
					0 & 0 & \beta_{Y^2}-c & 0 & 0 & 0 & 0			\\
					\beta_{X^2}-b & 0 & 0 & \beta_{X^4}-d & 0 & 0 & \beta_{X^2Y^2}\\
					0 & 0 & 0 & 0 & \beta_{X^2Y^2} & -\beta_{X^2Y^2} & 	0\\
					0 & 0 & 0 & 0 & -\beta_{X^2Y^2} & \beta_{X^2Y^2} & 0					\\
					\beta_{Y^2}-c & 0 & 0 & \beta_{X^2Y^2} & 0 &  0  & \beta_{Y^4}-e
				\end{mpmatrix},\label{matrix-2-BC3}
			\end{align}
	for some $a,b,c,d,e\in \RR$.
	Notice that the matrix (\ref{matrix-1-BC3}) is $L(a,b,c,d,e)$ and the matrix (\ref{matrix-2-BC3}) is
	$\mc{M}_2-L(a,b,c,d,e)$.
	Since $L(a,b,c,d,e)$ is a cm moment matrix, it admits a nc measure by Theorem \ref{com-case} 
	if and only if (\ref{point1-bc2-r6}) and (\ref{point3-bc2-r6}) of Corollary \ref{r6-bc3}
	are true. Since $\mc{M}_2-L(a,b,c,d,e)$ is a nc moment matrix satisfying
	$$\bX\bY+\bY\bX=0\quad \text{and}\quad
		\widetilde \beta_X=\widetilde\beta_Y=\widetilde\beta_{X^3}=\widetilde\beta_{Y^3}=0,$$
	it admits a nc measure by the results of rank 4 and 5 cases and 
	Theorem \ref{M(2)-XY+YX=0-bc4-r6} (1)
	if and only if (\ref{point2-bc2-r6}) of Corollary \ref{r6-bc3} is true.
\end{proof}

\section{ Flat extensions $\mc M_3$ for the BQTMP with $\mc{M}_2$ of rank 6} \label{flat-rank6}

In this section we characterize when a nc sequence $\beta\equiv\beta^{(4)}$ with a moment matrix $\mc{M}_2$ of rank 6 satisfying one of the basic relations of Proposition \ref{structure-of-rank5-2} (\ref{point-2-str-rank5}), admits a flat extension to a moment matrix $\mc M_3$. By Theorem \ref{flat-meas}, this is a sufficient
condition for the existence of a nc measure. We demonstrate with examples that this is not a necessary condition for none of the basic relations of Proposition \ref{structure-of-rank5-2} (\ref{point-2-str-rank5}). This shows a difference with the cm case since 
by Theorem \ref{com-case} (\ref{pt3}), if $\beta$ is a cm sequence with a psd moment matrix $\mc{M}_2$ satisfying exactly one of the relations $x^2+y^2=1$,   $x^2-y^2=1$ 
(since it is equivalent to $xy=1$) or $y^2=1$, then $\mc{M}_2$ admits a flat extension to a moment matrix $\mc M_3$.
We can also conclude that $\mc M_2$ being psd does not imply that there exist $a,b,c,d,e\in \RR$ in Corollary \ref{M(2)-bc2-r6-new1-cor} such that 
(\ref{point1-bc1-r6})-(\ref{point4-bc1-r6}) hold.

\subsection{Preliminaries} In this subsection we introduce some preliminaries needed in solving the flat extension question. 
First we establish the form of an extension $\mc{M}$ of a moment matrix $\mc{M}_2$ 
to be a moment matrix of degree 3. 

\begin{proposition} \label{B3,C3,structure}
	Suppose $\mc{M}_3=\begin{pmatrix} \mc{M}_2 & B_3 \\ B_3^t & C_3 \end{pmatrix}$ is a moment matrix of degree 3, where $B_{3}\in \RR^{7\times 8}$ and $C_3\in \RR^{8\times 8}$
and let the rows and columns be order lexicographically.
Then $B_3$ and $C_3$ are of the forms

\begin{equation*}
\begin{blockarray}{ccccccccc}
	 & \bX^3 & \bX^2 \bY & \bX\bY\bX &\bX \bY^2 & \bY\bX^2 &\bY\bX\bY &\bY^{2}\bX & \bY^3\\
	 \begin{block}{c(cccccccc)}
	\mds 1 & \beta_{X^3} & \beta_{X^2Y} & \beta_{X^2Y} & \beta_{XY^2} & \beta_{X^2 Y} & \beta_{XY^2} & 
		\beta_{XY^2} & \beta_{Y^3}\\
	\bX & \beta_{X^4} & \beta_{X^3Y} & \beta_{X^3Y} & \beta_{X^2Y^2} & \beta_{X^3 Y} & \beta_{XYXY} & 
		\beta_{X^2Y^2} & \beta_{XY^3}\\
	\bY & \beta_{X^3Y} & \beta_{X^2Y^2} & \beta_{XYXY} & \beta_{XY^3} & \beta_{X^2 Y^2} & \beta_{XY^3} & 
		\beta_{XY^3} & \beta_{Y^4}\\
	\bX^2 & \beta_{X^5} & \beta_{X^4Y} & \beta_{X^4Y} & \beta_{X^3Y^2} & \beta_{X^4 Y} & \beta_{X^2YXY} & 
		\beta_{X^3Y^2} & \beta_{X^2Y^3}\\
	\bX\bY & \beta_{X^4Y} & \beta_{X^3Y^2} & \beta_{X^2YXY} & \beta_{X^2Y^3} & \beta_{X^2 YXY} & \beta_{XY^2XY} & 
		\beta_{XY^2XY} & \beta_{XY^4}\\
	\bY\bX & \beta_{X^4Y} & \beta_{X^2YXY} & \beta_{X^2YXY} & \beta_{XY^2XY} & \beta_{X^3 Y^2} & \beta_{XY^2XY} & 
		\beta_{X^2Y^3} & \beta_{XY^4}\\
	\bY^2 & \beta_{X^3Y^2} & \beta_{X^2Y^3} & \beta_{XY^2XY} & \beta_{XY^4} & \beta_{X^2 Y^3} & \beta_{XY^4} & \beta_{XY^4} & \beta_{Y^5}\\
	\end{block}
\end{blockarray},
\end{equation*}

\begin{equation*}
\begin{blockarray}{ccccccccc}
	 & \bX^3 &\bX^2 \bY & \bX\bY\bX & \bX\bY^2 & \bY\bX^2 & \bY\bX\bY & \bY^{2}\bX & \bY^3\\
	 \begin{block}{c(cccccccc)}
	\bX^3 & \beta_{X^6} & \beta_{X^5Y} & \beta_{X^5Y} & \beta_{X^4Y^2} & \beta_{X^5 Y} & \beta_{X^3YXY} &
		\beta_{X^4Y^2} & \beta_{X^3 Y^3}\\ 
	\bX^2\bY & \beta_{X^5 Y} & \beta_{X^4Y^2} & \beta_{X^3YXY} & \beta_{X^3Y^3} & \beta_{X^2 YX^2Y} & 
		\beta_{X^2Y^2XY} & \beta_{X^2Y^2XY} & \beta_{X^2 Y^4}\\ 
	\bX\bY\bX & \beta_{X^4YX} & \beta_{X^3YXY} & \beta_{X^2YX^2Y} & \beta_{X^2Y^2XY} & \beta_{X^3 YXY} & 					\beta_{XYXYXY} & \beta_{X^2Y^2XY} & \beta_{XY^3XY}\\
	\bX\bY^2 & \beta_{X^4Y^2} & \beta_{X^3Y^3} & \beta_{X^2Y^2XY} & \beta_{X^2Y^4} & \beta_{X^2 Y^2XY} & 
		\beta_{X Y^3XY} & \beta_{X Y^2X Y^2} & \beta_{XY^5}\\  
	\bY\bX^2 & \beta_{X^5Y} & \beta_{X^2YX^2Y} & \beta_{X^3YXY} & \beta_{X^2Y^2XY} & \beta_{X^4 Y^2} & 
		\beta_{X^2Y^2XY} & \beta_{X^3Y^3} & \beta_{X^2 Y^4}\\
	\bY\bX\bY & \beta_{X^3YXY} & \beta_{X^2Y^2XY} & \beta_{XYXYXY} & \beta_{XY^3XY} & \beta_{X^2Y^2XY} & 
		\beta_{XY^2XY^2} & \beta_{XY^3XY} & \beta_{XY^5}\\
	\bY^2\bX & \beta_{X^4Y^2} & \beta_{X^2Y^2XY} & \beta_{X^2Y^2XY} & \beta_{XY^2XY^2} & \beta_{X^3Y^3} &
		\beta_{XY^3XY} & \beta_{X^2Y^4} & \beta_{XY^5}\\
	\bY^3 & \beta_{X^3Y^3} & \beta_{X^2Y^4} & \beta_{XY^3XY} & \beta_{XY^5} & \beta_{X^2Y^4} & 
		\beta_{XY^5} & \beta_{XY^5} & \beta_{Y^6}.   \\
	\end{block}
\end{blockarray},
\end{equation*}
respectively.
\end{proposition}

\begin{proof}
	This follows by definition of moment matrices.
\end{proof}

If $B_3$ and $C_3$ are of the form given in Proposition \ref{B3,C3,structure}, then we say they have a \textbf{moment structure}.\\

The moment structure of  $C_3$ implies the system given by the following proposition is satisfied.

\begin{proposition} \label{Hankel-system}
	If $\mc M_3$ is a moment matrix, then $C_3:=(C_{ij})_{ij}$ satisfies the following system

\begin{multicols}{2}
\begin{equation}\label{hankel-system-2}
\begin{aligned}
C_{47}=C_{66},\\
C_{25}=C_{33},\\
C_{12}=C_{13}=C_{15},\\
C_{18}=C_{24}=C_{57},\\
C_{16}=C_{23}=C_{35},
\end{aligned}
\end{equation}
\columnbreak
\begin{equation*}
\begin{aligned}
C_{38}=C_{46}=C_{67},\\
C_{48}=C_{68}=C_{78},\\
C_{14}=C_{17}=C_{22}=C_{55},\\
C_{28}=C_{58}=C_{44}=C_{77},\\
C_{26}=C_{27}=C_{34}=C_{37}=C_{45}=C_{56}.
\end{aligned}
\end{equation*}
\end{multicols}
\end{proposition}

Recall from Subsection \ref{general-ttmm} that for a polynomial 
$p\in\RR\!\left\langle X,Y\right\rangle_{\leq 2k}$, $\hat p = (a_w )_w$ denotes the 
coefficient vector with respect to the lexicographically-ordered basis 
	$$\{1, X, Y,X^2, XY, YX, Y^2,\ldots,X^{2k},\ldots,Y^{2k} \}$$
of $\RR\!\left\langle X,Y\right\rangle_{\leq 2k}$
We will use the following proposition to show that some of the equations in
(\ref{hankel-system-2}) are automatically satisfied for any flat extension $\mc M$ of a moment
matrix $\mc M_n$, i.e., $\mc M$ does not need to have the moment structure.

\begin{proposition}
\label{starsymmetry}
Suppose $\mc{M}=\begin{pmatrix} \mc{M}_n & B_{n+1} \\ B_{n+1}^t & C_{n+1}\end{pmatrix}$ is a flat extension of a moment matrix $\mc{M}_n$ with rows and columns indexed by monomials of degree at most $n+1$.
Let $\left\langle\cdot,\cdot\right\rangle_{\mc M}$ be a bilinear form on  
$\RR\!\left\langle X,Y\right\rangle_{\leq n+1}$ defined by
$\left\langle w_1,w_2 \right\rangle_{\mc M}:=\left\langle \mc M \widehat{w_1},\widehat{w_2}\right\rangle.$
For polynomials $p,q\in \RR\!\left\langle X,Y\right\rangle_{\leq n+1}$ we have
	\begin{equation} \label{bf-sym}
		\langle p,q\rangle_{\mc M} = \langle q, p\rangle_{\mc M},
	\end{equation}
	and
	\begin{equation}\label{bf-star-sym}
		\langle p,q\rangle_{\mc M} =\langle q^\ast,p^\ast\rangle_{\mc M}.
	\end{equation}
\end{proposition}

\begin{proof}
If $p,q$ are polynomials of degree at most $n$, then (\ref{bf-sym}), (\ref{bf-star-sym}) are
true due to the moment structure of $\mc{M}_n$. 
Suppose $p$ and $q$ be polynomials of degree at most $n+1$. The equality 
$\langle p,q\rangle_{\mc M} =\langle q,p\rangle_{\mc M}$ is true since $\mc{M}$ is symmetric. 
From $\Rank(\mc{M})=\Rank(\mc{M}_n)$, 
it follows that in $\mc M$ we have
	$$p(\bX,\bY)=\sum_{|u|\leq n}a_{u} u(\bX,\bY) 
		\quad\text{and}\quad 
		q(\bX,\bY)=\sum_{|t|\leq n} b_{t} t(\bX,\bY),$$
for some $a_u,b_t\in \RR$. Now by the properties of bilinear forms and the moment structure
of $\mc M_n$, we have
\begin{align*}
\langle  p, q\rangle_{\mc M}&=
\langle \mc M \widehat{p},\widehat{q}\rangle=
\Big\langle \mc M \big(\sum_{|u|\leq n}a_{u} \widehat u\big) ,\sum_{|t|\leq n} b_{t} \widehat t
\Big\rangle=
\sum_{|u|\leq n} \sum_{|t|\leq n}a_{u} b_{t} \big\langle \mc M \widehat u , \widehat t \big\rangle\\
&=\sum_{|u|\leq n} \sum_{|t|\leq n}a_{u} b_{t} \big\langle \mc M \widehat {t^\ast} , \widehat {u^\ast} \big\rangle=
\Big\langle \mc M \big(\sum_{|t|\leq n} b_{t} \widehat {t^\ast} \big),\sum_{|u|\leq n}a_{u} \widehat{u^\ast}\Big\rangle=
\langle \mc M \widehat{q^\ast},\widehat{p^\ast}\rangle\\
&= \langle  q^\ast , p^\ast \rangle_{\mc M}.
\end{align*}
This establishes Proposition \ref{starsymmetry}.
\end{proof}

\begin{corollary} \label{cor-Hank-sys}
Suppose $\mc{M}=\begin{pmatrix} \mc{M}_2 & B_3 \\ B_3^t & C_3\end{pmatrix}$ is a flat extension of a moment matrix $\mc{M}_2$ with rows and columns indexed by monomials of degree at most $3$. 
We write $C_3=(C_{ij})_{ij}$. Then we have:
			\begin{multicols}{4}
		\begin{equation}\label{aut-sat-hank}
			\begin{aligned}
				C_{12}= C_{15}, \\
				C_{24}= C_{57}, \\
				C_{23}= C_{35}, 
			\end{aligned}
		\end{equation}
\vfill
\columnbreak
		\begin{equation*}
			\begin{aligned}
			C_{46}= C_{67}, \\
			C_{48}= C_{78}, \\ 
			C_{14}= C_{17}, 
			\end{aligned}
		\end{equation*}
\vfill
\columnbreak
		\begin{equation*}
			\begin{aligned}
			C_{22}= C_{55}, \\
			C_{28}= C_{58}, \\
			C_{44}= C_{77}, 
			\end{aligned}
		\end{equation*}
\vfill
\columnbreak
		\begin{equation*}
			\begin{aligned}
			C_{26}&=& C_{56}, \\
			C_{27}&=& C_{45}, \\
			C_{34}&=& C_{37}.
			\end{aligned}
		\end{equation*}
\end{multicols}
	Therefore assuming $B_3$ has a moment structure, 
	$\mc M$ is a moment matrix of degree 3 if and only if
	\begin{multicols}{3}
\begin{equation}\label{hankel-system}
\begin{aligned}
	C_{47}=C_{66},\\	
	C_{25}=C_{33},\\
	C_{12}=C_{13},\\
	C_{18}=C_{24}, \\
\end{aligned}
\end{equation}
\columnbreak
\begin{equation*}
\begin{aligned}
	C_{16}=C_{23},\\
	C_{38}=C_{46},\\
	C_{48}=C_{68},\\
	C_{14}=C_{22},\\
\end{aligned}
\end{equation*}
\columnbreak
\begin{equation*}
\begin{aligned}
	C_{28}=C_{44},\\
	C_{26}=C_{27}=C_{34}.
\end{aligned}
\end{equation*}
\end{multicols}
\end{corollary}

\begin{proof}
	\noindent \textbf{Claim 1:} $C_{12}=C_{15}$.
	
		\begin{eqnarray*}
			C_{12}=\langle X^2Y, X^3\rangle_{\mc M}&=&
				\langle X^3,YX^2\rangle_{\mc M}	\quad (\text{by (\ref{bf-star-sym})}) \\
									 &=& \langle YX^2,X^3\rangle_{\mc M}=C_{15} 
									\quad (\text{by (\ref{bf-sym})).}
		\end{eqnarray*}
		
	\noindent \textbf{Claim 2:} $C_{24}=C_{57}$.
	
		\begin{eqnarray*}
		C_{24}= \langle XY^2,X^2Y\rangle_{\mc M}
			&=& \langle YX^2,Y^2X\rangle_{\mc M} \quad (\text{by (\ref{bf-star-sym})})\\
			&=& \langle Y^2X,YX^2\rangle_{\mc M}=C_{57}  \quad (\text{by (\ref{bf-sym})).}
		\end{eqnarray*}
		
	\noindent \textbf{Claim 3:} $C_{23}=C_{35}$.
		
		\begin{eqnarray*}
			C_{23}=\langle XYX, X^2Y\rangle_{\mc M}&=&
			\langle YX^2,XYX\rangle_{\mc M} \quad (\text{by (\ref{bf-star-sym})})\\
			&=&\langle XYX, YX^2\rangle_{\mc M}=C_{35} \quad (\text{by (\ref{bf-sym})).}
		\end{eqnarray*}
		
	\noindent \textbf{Claim 4:} $C_{46}=C_{67}$.
	
		\begin{eqnarray*}
			C_{46}=\langle YXY, XY^2\rangle_{\mc M}&=&
			\langle Y^2X,YXY\rangle_{\mc M} \quad (\text{by (\ref{bf-star-sym})})\\
			&=& \langle YXY, Y^2X\rangle_{\mc M}=C_{67} \quad (\text{by (\ref{bf-sym})).}
		\end{eqnarray*}
	
	\noindent \textbf{Claim 5:} $C_{48}=C_{78}$.
	
		\begin{eqnarray*}
			C_{48}=\langle Y^3,XY^2\rangle_{\mc M}&=&
			\langle Y^2X,Y^3\rangle_{\mc M} \quad (\text{by (\ref{bf-star-sym})})\\
			&=&\langle Y^3, Y^2X\rangle_{\mc M}=C_{78}  \quad (\text{by (\ref{bf-sym})).}
		\end{eqnarray*}
	
	\noindent \textbf{Claim 6:} $C_{14}=C_{17}$.
	
		\begin{eqnarray*}
			C_{14}=\langle XY^2,X^3\rangle_{\mc M}&=&
			\langle X^3,Y^2X\rangle_{\mc M}  \quad (\text{by (\ref{bf-star-sym})})\\
			&=&\langle Y^2X, X^3\rangle_{\mc M}=C_{17} \quad (\text{by (\ref{bf-sym})).}
		\end{eqnarray*}
		
	\noindent \textbf{Claim 7:} $C_{22}=C_{55}$.
	
		\begin{eqnarray*}
			C_{22}=\langle X^2Y,X^2Y\rangle_{\mc M}&=&
			\langle YX^2,YX^2\rangle_{\mc M}=C_{55}  \quad (\text{by (\ref{bf-star-sym}))}.
		\end{eqnarray*}

	 \noindent \textbf{Claim 8:} $C_{28}=C_{58}$.
	
		\begin{eqnarray*}
			C_{28}=\langle Y^3,X^2Y\rangle_{\mc M}&=&
			\langle YX^2,Y^3\rangle_{\mc M} \quad (\text{by (\ref{bf-star-sym}))}\\
			&=&\langle Y^3,YX^2\rangle_{\mc M}=C_{58}\quad (\text{by (\ref{bf-sym})).}
		\end{eqnarray*}
		
	 \noindent \textbf{Claim 9:} $C_{44}=C_{77}$.
	
		\begin{eqnarray*}
			C_{44}=\langle XY^2,XY^2\rangle_{\mc M} &=&
			\langle Y^2X,Y^2X\rangle_{\mc M} =C_{77}  \quad (\text{by (\ref{bf-star-sym}))}.
		\end{eqnarray*}

	\noindent \textbf{Claim 10:} $C_{26}=C_{56}.$
	
		\begin{eqnarray*}
			C_{26}=\langle YXY,X^2Y\rangle_{\mc M} &=&
			\langle YX^2,YXY\rangle_{\mc M} \quad (\text{by (\ref{bf-star-sym}))}\\
			&=&\langle YXY, YX^2\rangle_{\mc M} =C_{56}  \quad (\text{by (\ref{bf-sym})).}
		\end{eqnarray*}
	 
	 \noindent \textbf{Claim 11:} $C_{27}=C_{45}$.
	 
	 	\begin{eqnarray*}
			C_{27}=\langle Y^2X,X^2Y\rangle_{\mc M} &=&
			\langle YX^2,XY^2\rangle_{\mc M} \quad (\text{by (\ref{bf-star-sym}))}\\
			&=&\langle  XY^2, X^2Y \rangle_{\mc M} =C_{45} \quad (\text{by (\ref{bf-sym})).}
		\end{eqnarray*}
	 
	 \noindent \textbf{Claim 12:} $C_{34}=C_{37}$.
	 	\begin{eqnarray*}
			C_{34}=\langle XY^2,XYX\rangle_{\mc M} &=&
			\langle XYX,Y^2X\rangle_{\mc M} \quad (\text{by (\ref{bf-star-sym}))}\\
			&=&\langle  Y^2X, XYX \rangle_{\mc M}=C_{37} \quad (\text{by (\ref{bf-sym})).}
		\end{eqnarray*}
	This proves the first statement of Corollary \ref{cor-Hank-sys}. The second statement follows by observing that if $\mc M$ is a moment matrix, then the entries of $C_3$ are independent from the other entries of $\mc M$ and combining (\ref{hankel-system-2}) with 
	(\ref{aut-sat-hank}).
\end{proof}

If $\mc{M}=\begin{pmatrix} \mc{M}_n & B_{n+1} \\ B_{n+1}^t & C_{n+1}\end{pmatrix}$
is a flat extension of $\mc{M}_n$, then there is a matrix $W$ such that 
	$$B_{n+1}=\mc M_n W\quad \text{and}\quad C_{n+1}=W^t \mc M_n W.$$ 
By the following lemma $C_{n+1}$ is independent of the choice of $W$ satisfying $B_{n+1}=\mc M_n W$.

\begin{lemma} \label{lemma-about-symmetric}
	Let $A\in \mathbb{SR}^{m\times m}$ be a symmetric matrix and 
	$W_1, W_2\in \RR^{m\times p}$ matrices satisfying $AW_1=AW_2$. 
	Then $W_1^t AW_1=W_2^t AW_2.$
\end{lemma}

\begin{proof}
		Since $W_j^t AW_j$ are symmetric matrices, we have
	\begin{eqnarray*}
		W_1^t AW_1=W_2^t AW_2 
			&\Leftrightarrow&
			\big\langle (W_1^t AW_1-W_2^t AW_2) v,v\big\rangle\quad \text{for every }v\in \RR^p\\
			&\Leftrightarrow&
			\big\langle AW_1v,W_1v\rangle=\langle AW_2v,W_2v\big\rangle\quad \text{for every } v\in \RR^p.
	\end{eqnarray*}
	Let us write $v_1:=W_1 v$ and $v_2:=W_2 v$.
	By assumption $AW_1=AW_2$ it follows that $Av_1=Av_2$.
	The following calculation holds:
	\begin{eqnarray*}
		0&=&\left\langle A(v_1-v_2),(v_1+v_2) \right\rangle=
		\left\langle Av_1,v_1 \right\rangle+\left\langle Av_1,v_2 \right\rangle-\left\langle Av_2,v_1 \right\rangle-\left\langle Av_2,v_2 \right\rangle\\
		&=&\left\langle Av_1,v_1 \right\rangle+\left\langle v_1,Av_2 \right\rangle-\left\langle Av_2,v_1 \right\rangle-\left\langle Av_2,v_2 \right\rangle\\
		&=&\left\langle Av_1,v_1 \right\rangle+\left\langle Av_2,v_1 \right\rangle-\left\langle Av_2,v_1 \right\rangle-\left\langle A v_2,v_2 \right\rangle
		=\left\langle Av_1,v_1 \right\rangle-\left\langle Av_2,v_2 \right\rangle.
	\end{eqnarray*}
	This concludes the proof of the lemma.
\end{proof}


\subsection{Relation $\bY^2=\mds{1}-\bX^2$.}

The candidate for $B_3$ in a moment matrix $\mc{M}_3$ generated by the measure
for $\mc{M}_2$ is given by the following.

\begin{proposition}\label{bc2-r6-B(3)-1}
	Let $\beta\equiv \beta^{(4)}$ be a nc sequence with a moment matrix $\mc{M}_2$ of rank 6 satisfying the relation 
	$\bY^2=\mds{1}-\bX^2$.
	Suppose $\beta$ admits a nc measure $\mu$. If
		$\mc{M}_3 =\begin{pmatrix} \mc{M}_2 & B_3 \\ B_3^t & C_3 \end{pmatrix}$
	is a moment matrix generated by the measure $\mu$, then $B_3$ satisfies
	\begin{multicols}{2}
	\begin{equation}\label{order-5-1}
	\begin{aligned}
		\beta_{X^2Y^3}=\beta_{XY^2XY} = \beta_{X^2Y}-q,\\
		\beta_{X^3Y^2}=\beta_{X^2YXY} = \beta_{X^3}-p,\\
		\beta_{XY^4} = \beta_X-2\beta_{X^3}+p,\\
	\end{aligned}
	\end{equation}
\vfill
\columnbreak
	\begin{equation*}
		\begin{aligned}
			\beta_{Y^5} = \beta_Y-2\beta_{X^2Y}+q,\\
			\beta_{X^5} = p,\\
			\beta_{X^4Y}=q,
		\end{aligned}
	\end{equation*}
\end{multicols}
	\noindent where $p,q\in \RR$ are parameters.
\end{proposition}

\begin{proof}
	The RG relations which must hold in $\mc{M}_3$ are
	\begin{multicols}{2}
	\begin{eqnarray*}
		\bY^3	&=& \bY-\bX^2\bY,\\
		\bY^3	&=& \bY-\bY\bX^2,
	\end{eqnarray*}
	\columnbreak
	\begin{eqnarray*}
		\bX\bY^2  &=& \bX-\bX^3,\\
		\bY^2\bX 	&=& \bX-\bX^3.
	\end{eqnarray*}
	\end{multicols}
	\noindent From these relations we get the following system:
	\begin{multicols}{2}
	\begin{equation*}\label{y2=1-x2 rg-system-b}
	\begin{aligned}
		\beta_{XY^4} &=& (\beta_X-\beta_{X^3})-\beta_{X^3Y^2},\\
		\beta_{XY^4} &=& (\beta_X-\beta_{X^3})-\beta_{X^2YXY},\\
		\beta_{Y^5} &=& (\beta_Y-\beta_{X^2Y})-\beta_{X^2Y^3},
	\end{aligned}
	\end{equation*}
	\columnbreak
	\begin{equation*}
	\begin{aligned}
		\beta_{X^2Y^3} &=& \beta_{X^2Y}-\beta_{X^4Y},\\
		\beta_{X^3Y^2} &=& \beta_{X^3}-\beta_{X^5},\\
		\beta_{XY^2XY} &=& \beta_{X^2Y}-\beta_{X^4Y}.\\
	\end{aligned}
	\end{equation*}
	\end{multicols}
	\noindent Now the solution of this system is given by the statement of 
	the proposition.
\end{proof}

\begin{theorem}\label{flat-bc2-r6}
	Suppose $\beta\equiv \beta^{(4)}$ is a nc sequence with a moment matrix $\mc{M}_2$ of rank 6 satisfying the relation 
	$\bY^2=\mds{1}-\bX^2$. Let us define the moments
	of degree 5 by (\ref{order-5-1}) and
	$B_3$ as in Proposition \ref{B3,C3,structure}. 
	Then the following are true:
	\begin{enumerate} 
		\item There exists a matrix $W\in \RR^{7\times 10}$ such that
			$$B_3 = \mc{M}_2W.$$
	\item \label{point4-bc2-r6}
		We write $M_1 =\{\mds 1,\bX,\bY,\bX^2, \bX\bY,\bY\bX\}$ and  
		$M_2 =\{\bX^3, \bX^2 \bY, \bX\bY\bX, \bX\bY^2, \bY\bX^2, \bY\bX\bY, \bY^{2}\bX, \bY^3\}$. Let $W_1\in \RR^6\times \RR^{10}$
		be the matrix 	
			$$W_1=({\mc{M}_2}|_{M_1})^{-1} B_3|_{M_1,M_2}.$$
		If $\mc{M} =\begin{pmatrix} \mc{M}_2 & B_3 \\ B_3^t & C_3 \end{pmatrix}$
		is a flat extension of $\mc M_2$, then
		$C_3=(C_{ij})_{ij}$ is equal to
		$W_1^t {\mc{M}_2}|_{M_1} W_1$ and
		$\mc{M}$ has a moment structure if and only if 
		\begin{multicols}{3}
		\begin{equation}\label{hank-sys-bc1}
			\begin{aligned}
				C_{47}= C_{66},	\\
				C_{25}= C_{33},	\\
				C_{12}=C_{13},
			\end{aligned} 
		\end{equation}
\vfill
\columnbreak
		\begin{equation*}
			\begin{aligned}
				C_{16}= C_{23},\\
				C_{48}= C_{68},\\
				C_{14}= C_{22},	
			\end{aligned}
		\end{equation*}
\vfill
\columnbreak
		\begin{equation*}
			\begin{aligned}
				C_{28}= C_{44},\\
				C_{26}= C_{27}.
			\end{aligned}
		\end{equation*}
\end{multicols}
	\end{enumerate}
\end{theorem}

\begin{proof}
		To prove (1) we have to show that every column of $B_3$ belongs to the linear 
	span of the columns of $\mc{M}_2$. 
		Since the proofs are analogous, we will establish this only for the column $\bX^3$.
	Since ${\mc{M}_2}|_{M_1}$ is positive definite, it follows that
		\begin{equation}\label{rel-bc2-rank6}
			\bX^3|_{M_1}=a_1 \mds{1}|_{M_1}+a_2 \bX|_{M_1} +
				a_3 \bY|_{M_1} +a_4\bX^2|_{M_1} +a_5\bX\bY|_{M_1}+a_6\bY\bX|_{M_1},
		\end{equation}
	for some $a_i\in \RR$ where $\ast|_{M_1}$ denotes the restriction of the column $\ast$ to the rows from $M_1$. Notice that $a_5=a_6$.
	Using the relation $\bY^2=\mds{1}-\bX^2$ we calculate 
	\begin{align*}
		a_1\beta_{Y^2}+a_2\beta_{XY^2}+a_3\beta_{Y^3}+
			a_4\beta_{X^2Y2}+2a_5\beta_{XY^3}\\
		= a_1(\beta_1-\beta_{X^2})+a_2(\beta_X-\beta_{X^3})+
			a_3(\beta_Y-\beta_{X^2Y})+
			a_4(\beta_{X^2}-\beta_{X^4})\\+2a_5(\beta_{XY}-\beta_{X^3Y})\\
		= \beta_{X^3}-p.
	\end{align*}
	By the form of $B_3$ it follows that 
		\begin{equation*}
			\bX^3=a_1 \mds{1}+a_2 \bX+a_3 \bY +a_4\bX^2 +a_5\bX\bY+a_6\bY\bX 
			\quad\text{in }\mc M_3.
		\end{equation*}
	This proves part (1).
	
	If $\mc{M}$ is a flat extension of $\mc{M}_2$, then in particular
	 $\Rank(\begin{pmatrix} \mc{M}_2 & B_3\end{pmatrix})=\Rank \mc{M}_2$. Since the columns from $M_1$ are the
	basis for the column space of $\mc{M}_2$, we have 
		$$B_3=\mc{M}_2 \begin{pmatrix}W_{1} \\ 0_{1\times 10} \end{pmatrix}.$$
	By Lemma \ref{lemma-about-symmetric} it follows that
		$$C_3=\begin{pmatrix}W_{1}^t & 0_{10\times 1}\end{pmatrix}
			\mc{M}_2 \begin{pmatrix}W_{1} \\ 0_{1\times 10} \end{pmatrix}=
				W_1^t {\mc{M}_2}|_{M_1} W_1.$$
	Now we will establish the relations which will prove that the system
	(\ref{hankel-system}) holds if and only if 
	the system (\ref{hank-sys-bc1}) holds.\\
	
	\noindent \textbf{Claim 1:} $C_{18}=C_{24}$.
	
		\begin{eqnarray*}
			C_{18}=\langle Y^3, X^3\rangle_{\mc M_3}&=&
			\langle Y-X^2Y,X^3\rangle_{\mc M_3}\quad \text{(by RG relations)}\\
			&=& \langle Y,X^3\rangle_{\mc M_3}-\langle X^3,X^2Y\rangle_{\mc M_3} 
				\quad\text{(by (\ref{bf-sym}))}\\
			&=& \langle X,X^2Y\rangle_{\mc M_3}-\langle X^3,X^2Y\rangle_{\mc M_3} 
				\quad \text{(by the moment structure of }B_3)\\
			&=& \langle XY^2,X^2Y\rangle_{\mc M_3}=C_{24} 
				\quad \text{(by RG relations)}.
		\end{eqnarray*}

	\noindent \textbf{Claim 2:} $C_{16}-C_{23}=C_{38}-C_{46}$.
		\begin{eqnarray*}
			C_{16}-C_{23}&=&\langle YXY, X^3\rangle_{\mc M_3}-\langle XYX, X^2Y\rangle_{\mc M_3}\\
				&=&\langle YXY, X-XY^2\rangle_{\mc M_3}-\langle X^2Y, XYX\rangle_{\mc M_3} 
				\quad \text{(by RG relations and (\ref{bf-sym}))}\\
				&=& (\langle X, YXY\rangle_{\mc M_3}-\langle X^2Y, XYX\rangle_{\mc M_3})-\langle YXY, XY^2\rangle_{\mc M_3}\\
				&& \quad \text{(by the moment structure of }B_3)\\
				&=& \langle Y^3, XYX\rangle_{\mc M_3} -\langle YXY, XY^2\rangle_{\mc M_3} \quad \text{(by RG relations)}\\
				&=& C_{38}-C_{46}.
		\end{eqnarray*}

	 \noindent \textbf{Claim 3:} $C_{37}-C_{27}=C_{12}-C_{13}$.
			\begin{eqnarray*}
			C_{37}-C_{27}&=&\langle Y^2X, XYX\rangle_{\mc M_3}-\langle Y^2X,X^2Y\rangle_{\mc M_3}\\
				&=&\langle X-X^3, XYX\rangle_{\mc M_3}-\langle Y^2X, X^2Y\rangle _{\mc M_3}\quad \text{(by RG relations)}\\
				&=& (\langle X,X^2Y\rangle_{\mc M_3}-\langle Y^2X, X^2Y\rangle_{\mc M_3})-
					\langle X^3, XYX\rangle_{\mc M_3}\\
				&& \quad \text{(by the moment structure of }B_3)\\
				&=& \langle X^3,X^2Y\rangle_{\mc M_3}-\langle X^3, XYX\rangle_{\mc M_3}\\
				&=& C_{12}-C_{13}.
			\end{eqnarray*}

	\noindent Using Claims 1-3 proves Theorem \ref{flat-bc2-r6} (\ref{point4-bc2-r6}).
\end{proof}

\begin{remark}
		Assume the notation as in Theorem \ref{flat-bc2-r6}. If $\mc{M}$ is a flat extension
	and has a moment structure, then we must also have 
		$$C_{26}=C_{18}\quad \text{and}\quad C_{14}=C_{25},$$ 
	which follows by the following:

			\begin{eqnarray*}
			C_{48}-C_{68}&=&\langle Y^3, XY^2\rangle_{\mc M_3}
					-\langle Y^3, YXY\rangle_{\mc M_3} \\
				&=&\langle Y^3, XY^2\rangle_{\mc M_3}-
					\langle Y-X^2Y, YXY\rangle_{\mc M_3}
				\quad\text{(by RG relations)}\\
				&=&\langle Y^3, XY^2\rangle_{\mc M_3}-
					\langle Y^3,X\rangle_{\mc M_3}+
					\langle X^2Y, YXY\rangle_{\mc M_3} \\
				&& \quad \text{(by the moment structure of }B_3)\\
				&=& \langle X^2Y, YXY\rangle_{\mc M_3} - 
					\langle Y^3, X^3\rangle_{\mc M_3} 
				\quad\text{(by RG relations)}\\
				&=& \langle YXY, X^2Y\rangle_{\mc M_3} - 
					\langle X^3, Y^3\rangle_{\mc M_3} 
				\quad\text{(by (\ref{bf-sym}))}\\
				&=& C_{26}-C_{18},
		\end{eqnarray*}
and
	
		\begin{eqnarray*}
			C_{58}-C_{44}&=&\langle Y^3, YX^2\rangle_{\mc M_3}-
				\langle XY^2,XY^2\rangle_{\mc M_3}\\
				&=&\langle Y-X^2Y, YX^2\rangle_{\mc M_3}-
					\langle XY^2, XY^2\rangle_{\mc M_3} \quad 
					\text{(by RG relations)}\\
				&=& (\langle Y,YX^2\rangle_{\mc M_3}-
					\langle XY^2, XY^2\rangle_{\mc M_3})-
					\langle YX^2, X^2Y\rangle_{\mc M_3} \quad \text{(by (\ref{bf-sym}))}\\
				&=& (\langle X,XY^2\rangle_{\mc M_3}-
					\langle XY^2, XY^2\rangle_{\mc M_3})-
					\langle YX^2, X^2Y\rangle_{\mc M_3} \\
				&& \quad \text{(by the moment structure of }B_3)\\
				&=& \langle XY^2,X^3\rangle_{\mc M_3}-
					\langle YX^2, X^2Y\rangle_{\mc M_3} \quad 
						\text{(by RG relations and (\ref{bf-sym}))}\\
				&=& C_{14}-C_{25}.
		\end{eqnarray*}
\end{remark}

We present now a special case which highlights the difference between the classical commutative and the tracial bivariate quartic moment problems with a moment matrix $\mc M_2$ satisfying exactly $\bY^2=\mds{1}-\bX^2$ . 
By Theorem \ref{com-case} (\ref{pt3}), in the cm case $\mc M_2$ always admits a flat extension to the moment matrix $\mc M_3$.
The following example shows that this is not true in the nc case. However, a nc measure in this example still exists.

\begin{example}\label{ex-bc2-r6} 
	For $\beta_{X^4}\in \big(\frac{1}{4},\frac{1}{2}\big)$, the following matrices are psd moment matrices 
	of rank 6 satisfying the relation $\bY^2=\mds{1}-\bX^2$,
$$\mc{M}_2(\beta_{X^4})=
	\begin{mpmatrix}
 		1 & 0 & 0	 & \frac{1}{2} & 0 & 0 & \frac{1}{2} \\
 		0 & \frac{1}{2} & 0 & 0 & 0 & 0 & 0 \\
 		0 & 0 & \frac{1}{2} & 0 & 0 & 0 & 0 \\
		 \frac{1}{2} & 0 & 0 & \beta_{X^4} & 0 & 0 &
 		  \frac{1}{2}-\beta_{X^4} \\
 		0 & 0 & 0 & 0 & \frac{1}{2}-\beta_{X^4} & 0 & 0 \\
 		0 & 0 & 0 & 0 & 0 & \frac{1}{2}-\beta_{X^4} & 0 \\
 		\frac{1}{2} & 0 & 0 & \frac{1}{2}-\beta_{X^4} & 0 & 0 & \beta_{X^4} \\
	\end{mpmatrix}.$$
\noindent
Let us define the moments
	of degree 5 by (\ref{order-5-1}),
	$B_3$ as in Proposition \ref{B3,C3,structure} and $M_1, M_2, C_3$ as in Theorem \ref{flat-bc2-r6} (\ref{point4-bc2-r6}).
	 It is easy to check that
	$$({\mc{M}_2}|_{M_1})^{-1}=
		\begin{pmatrix} 4\beta_{X^4}&0&0&\frac{2}{1-4\beta_{X^4}}\\0&2&0&0\\0&0&2&0\\\frac{2}{1-4\beta_{X^4}}&0&0&\frac{4}{4\beta_{X^4}-1} \end{pmatrix}
		\bigoplus \begin{pmatrix} \frac{2}{1-2\beta_{X^4}} & 0 \\ 0 &  \frac{2}{1-2\beta_{X^4}}\end{pmatrix}.
	$$
	By a straightworward calculation of $\left({B_3}|_{M_1,M_2}\right)^t ({\mc{M}_2}|_{M_1})^{-1} \left({B_3}|_{M_1,M_2}\right)$
	we get that
	\begin{eqnarray*}
		C_{47}(p,q,\beta_{X^4}) &=&
			\frac{1}{2} - 2 \beta_{X^4}+ 2 \beta_{X^4}^2 + \frac{4 q^2}{1 - 2 \beta_{X^4}} + 
			\frac{4 p^2}{-1 + 4 \beta_{X^4}},\\
		C_{66}(p,q,\beta_{X^4}) &=&
			\frac{4q^2}{1-2\beta_{X^4}}+\frac{4p^2}{-1+4\beta_{X^4}},
	\end{eqnarray*}
	and hence 
		$$(C_{47}-C_{66})(p,q,\beta_{X^4}) = \frac{1}{2}(1-2\beta_{X^4})^2\neq 0.$$
	Therefore the system from Theorem \ref{flat-bc2-r6} does not have a solution and 
 	$\mc{M}_2(\beta_{X^4})$ does not admit a flat extension with a moment structure $\mc{M}_3(\beta_{X^4})$.
However, for every $\beta_{X^4}\in (\frac{1}{4},\frac{1}{2})$, 
$\mc{M}_2(\beta_{X^4})$ admits a nc measure by Theorem \ref{M(2)-bc2-r6-new1} (1).
\end{example}

\subsection{Relation $\bX\bY+\bY\bX=\mbf 0$.}

The candidate for $B_3$ in a moment matrix $\mc{M}_3$ generated by the nc measure
for $\mc{M}_2$ is given by the following.

\begin{proposition}
	Let $\beta\equiv \beta^{(4)}$ be a sequence with a moment matrix $\mc{M}_2$ of rank 6 satisfying the relation 
	$\bX\bY+\bY\bX=\mbf 0$.
	Suppose $\beta$ admits a nc measure $\mu$. If
		$\mc{M}_3 =\begin{pmatrix} \mc{M}_2 & B_3 \\ B_3^t & C_3 \end{pmatrix}$
	is a moment matrix generated by the nc measure $\mu$, then $B_3$ satisfies
		\begin{equation}\label{b3-sys-bc2}
			\begin{aligned}
				\beta_{X^4Y}=\beta_{X^2YXY}=\beta_{X^3Y^2}=
				\beta_{X^2Y^3}=\beta_{XY^2XY}=\beta_{XY^4}=0,\\
				\beta_{X^5}=p\quad\text{and}\quad \beta_{Y^5}=q,
			\end{aligned}
		\end{equation}
	\noindent where $p, q\in \RR$ are parameters.
\end{proposition}

\begin{proof}
	The RG relations which must hold in $\mc{M}_3$ are 
	\begin{multicols}{2}
	\begin{equation*}\label{eq-bc4-block-B(3)}
	\begin{aligned}
		\bX^2\bY+\bX\bY\bX	&=& \mbf 0,\\
		\bX\bY\bX+\bY\bX^2	&=& \mbf 0,
	\end{aligned}
	\end{equation*}
	\columnbreak
	\begin{eqnarray*}
		\bY\bX\bY+\bY^2\bX  &=& \mbf 0,\\
		\bX\bY^2+\bY\bX\bY 	&=& \mbf 0.
	\end{eqnarray*}
	\end{multicols}
	\noindent From this relations we get the following system:
	\begin{multicols}{2}
	\begin{equation*}\label{xy+yx=0 rg-system-b}
	\begin{aligned}
		2\beta_{X^4Y} &=& 0,\\
		\beta_{X^3Y^2}+\beta_{X^2YXY} &=& 0,\\
		2\beta_{X^2YXY}&=& 0,
	\end{aligned}
	\end{equation*}
	\columnbreak
	\begin{equation*}
	\begin{aligned}
		\beta_{X^2Y^3}+\beta_{XY^2XY} &=& 0,\\
		2\beta_{XY^2XY}&=& 0,\\
		2\beta_{XY^4} &=& 0.
	\end{aligned}
	\end{equation*}
	\end{multicols}
	\noindent Now the solution of this system is given by the statement of the proposition.
\end{proof}

\begin{theorem}\label{flat-bc44-r6}
	Suppose $\beta\equiv \beta^{(4)}$ is a nc sequence with a moment matrix $\mc{M}_2$ of rank 6 satisfying the relation 
	$\bX\bY+\bY\bX=\mbf 0$. Let us define the moments of degree 5 by (\ref{b3-sys-bc2}) and
	$B_3$ as in Proposition \ref{B3,C3,structure}. 
	Then the following are true:
	\begin{enumerate} 
		\item There exists a matrix $W\in \RR^{7\times 10}$ satisfying 
			$$B_3 = \mc{M}_2W.$$
		\item \label{point4-bc4-r6}
		We write
		$M_1 =\{\mds 1,\bX,\bY,\bX^2, \bX\bY,\bY^2\}$ and  
		$M_2 =\{\bX^3, \bX^2 \bY, \bX\bY\bX, \bX\bY^2, \bY\bX^2, \bY\bX\bY, \bY^{2}\bX, \bY^3\}$. 
		Let $W_1\in \RR^{6\times 10}$ 
		be the matrix	
			$$W_1=({\mc{M}_2}|_{M_1})^{-1} {B_3}|_{M_1,M_2}.$$
		If $\mc{M}=\begin{pmatrix} \mc{M}_2 & B_3 \\ B_3^t & C_3 \end{pmatrix}$
		is a flat extension of $\mc M_2$, then
		$C_3=(C_{ij})_{ij}$ is equal to
		$W_1^t {\mc{M}_2}|_{M_1} W_1$ and
		$\mc{M}$ has a moment structure if and only if 
		\begin{multicols}{3}
		\begin{equation}\label{hank-sys-bc2}
			\begin{aligned}
				C_{12}= C_{13}=0,	\\
				C_{18}= C_{24},\\ 
				C_{16}= C_{23},
			\end{aligned}
		\end{equation}
\vfill
\columnbreak
		\begin{equation*}
			\begin{aligned}
				C_{38}= C_{46},\\
				C_{48}= C_{68}=0,\\
				C_{14}=C_{22}, 
			\end{aligned}
		\end{equation*}
\vfill
\columnbreak
		\begin{equation*}
			\begin{aligned}
				C_{28}= C_{44},\\
				C_{26}= C_{27}=C_{34}.
			\end{aligned}
		\end{equation*}
\end{multicols}
	\end{enumerate}
\end{theorem}

\begin{proof}
	The proof is analogous to the proof of Theorem \ref{flat-bc2-r6}. Using similar arguments we prove the following claim:\\
	
	\noindent\textbf{Claim:}
		$C_{47}= C_{66}$,
			 $C_{25}=C_{33}$, 
		$C_{12}=-C_{13},$
			$C_{24}= C_{57},$
		$C_{48}=-C_{68}.$	\\
		
	By Claim the system (\ref{hankel-system}) holds if and only if the system (\ref{hank-sys-bc2}) holds.
\end{proof}

We present now a special case which shows that for a nc sequence $\beta$ with a psd moment matrix $\mc M_2$ satisfying exactly $\bX\bY+\bY\bX=\mbf 0$, admitting a flat extension to the moment matrix $\mc M_3$ is not equivalent to admitting a nc measure.

\begin{example}
	For $\beta_{X^4}>1$, the following matrices are psd moment matrices 
	of rank 6 satisfying the relation $\bX\bY+\bY\bX=\mbf 0$,
$$\mc{M}_2(\beta_{X^4})=
	\begin{mpmatrix}
 		1 & 0 & 0	 & 1 & 0 & 0 & 1 \\
 		0 & 1 & 0 & 0 & 0 & 0 & 0 \\
 		0 & 0 & 1 & 0 & 0 & 0 & 0 \\
		1 & 0 & 0 & \beta_{X^4} & 0 & 0 & 1\\
 		0 & 0 & 0 & 0 & 1 & -1 & 0 \\
 		0 & 0 & 0 & 0 & -1 & 1 & 0 \\
 		1 & 0 & 0 & 1 & 0 & 0 & 2 
	\end{mpmatrix}.$$
\noindent
Let us define the moments
	of degree 5 by (\ref{b3-sys-bc2}),
	$B_3$ as in Proposition \ref{B3,C3,structure} and $M_1, M_2, C_3$ as in Theorem \ref{flat-bc44-r6} (2).
	It is easy to check that
	$$({\mc{M}_2}|_{M_1})^{-1}=
\begin{pmatrix} 
 \frac{2 \beta_{X^4}-1}{\beta_{X^4}-1} & 0 & 0 & \frac{1}{1-\beta_{X^4}} & 0 & -1 \\
 0 & 1 & 0 & 0 & 0 & 0 \\
 0 & 0 & 1 & 0 & 0 & 0 \\
 \frac{1}{1-\beta_{X^4}} & 0 & 0 & \frac{1}{\beta_{X^4}-1} & 0 & 0 \\
 0 & 0 & 0 & 0 & 1 & 0 \\
 -1 & 0 & 0 & 0 & 0 & 1 \\
\end{pmatrix}.$$
	By a straightworward calculation of $\left({B_3}|_{M_1,M_2}\right)^t ({\mc{M}_2}|_{M_1})^{-1} \left({B_3}|_{M_1,M_2}\right)$
	we get that
	$$C_3(p,q)=\begin{mpmatrix}
 		\frac{p^2}{\beta_{X^4}-1}+4 & 0 & 0 & 2 & 0 & -2 & 2 & 0 \\
		 0 & 1 & -1 & 0 & 1 & 0 & 0 & 2 \\
		 0 & -1 & 1 & 0 & -1 & 0 & 0 & -2 \\
		 2 & 0 & 0 & 1 & 0 & -1 & 1 & 0 \\
		 0 & 1 & -1 & 0 & 1 & 0 & 0 & 2 \\
 		-2 & 0 & 0 & -1 & 0 & 1 & -1 & 0 \\
		 2 & 0 & 0 & 1 & 0 & -1 & 1 & 0 \\
		 0 & 2 & -2 & 0 & 2 & 0 & 0 & q^2+4.
 	\end{mpmatrix}$$
	The system (\ref{hank-sys-bc2}) does not have a solution, e.g., 
		$-2=C_{16}\neq C_{23}=-1.$
Hence $\mc{M}_2(\beta_{X^4})$ does not admit a flat extension with a moment structure $\mc M_3(\beta_{X^4})$.
However, for every $\beta_{X^4}>1$, 
	$\mc{M}_2(\beta_{X^4})$ admits a nc measure by Theorem \ref{M(2)-XY+YX=0-bc4-r6} (1).
\end{example}

\subsection{Relation $\bY^2=\mds 1+\bX^2$.}

The form of $\mc{M}_2$ is given by the following proposition.

\begin{proposition}\label{r6-rel-bc2}
	Let $\beta\equiv \beta^{(4)}$ be a nc sequence with a moment matrix $\mc{M}_2$ satisfying the relation
		\begin{equation}\label{r6-rel-bc3-eq} 
			\bY^2=\mds{1}+\bX^2.
		\end{equation}
	Then $\mc{M}_2$ if of the form
		\begin{equation}\label{bc3-r6}
		\begin{mpmatrix}
 		\beta_{1} & \beta_{X} & \beta_{Y} & \beta_{X^2} & \beta_{XY} & \beta_{XY} & 		
			\beta_{1}+\beta_{X^2} \\
 		\beta_{X} & \beta_{X^2} & \beta_{XY} & \beta_{X^3} & \beta_{X^2Y} & \beta_{X^2Y} & 		
			\beta_{X}+\beta_{X^3} \\
		 \beta_{Y} & \beta_{XY} & \beta_{1}+\beta_{X^2} & \beta_{X^2Y} & \beta_{X}+	\beta_{X^3} & 
		 	\beta_{X}+\beta_{X^3} &   \beta_{Y}+\beta_{X^2Y} \\
		 \beta_{X^2} & \beta_{X^3} & \beta_{X^2Y} & \beta_{X^4} & \beta_{X^3Y} & 		
			 \beta_{X^3Y} & \beta_{X^2}+\beta_{X^4} \\
		 \beta_{XY} & \beta_{X^2Y} & \beta_{X}+\beta_{X^3} & \beta_{X^3Y} & 
			\beta_{X^2}+\beta_{X^4} & \beta_{XYXY} & \beta_{XY}+\beta_{X^3Y} \\
 		\beta_{XY} & \beta_{X^2Y} & \beta_{X}+\beta_{X^3} & \beta_{X^3Y} & \beta_{XYXY} & 		
			\beta_{X^2}+\beta_{X^4} & \beta_{XY}+\beta_{X^3Y} \\
		 \beta_{1}+\beta_{X^2} & \beta_{X}+\beta_{X^3} & \beta_{Y}+\beta_{X^2Y} & 
			\beta_{X^2}+\beta_{X^4} & 
		\beta_{XY}+\beta_{X^3Y} & \beta_{XY}+\beta_{X^3Y} & \beta_{1}+2 \beta_{X^2}+	
			\beta_{X^4} 
		\end{mpmatrix}
		\end{equation}
\end{proposition}

\begin{proof}
	The relation (\ref{r6-rel-bc3-eq}) gives us the following system in $\mc{M}_2$
		\begin{multicols}{2}
		\begin{equation}\label{eq-bc3-r6}
			\begin{aligned}
		\beta_{Y^2}	= \beta_{1}+\beta_{X^2},\\
		\beta_{XY^2}	= \beta_{X}+\beta_{X^3},\\
		\beta_{Y^3}	= \beta_{Y}+\beta_{X^2Y},
			\end{aligned}
		\end{equation}
	\vfill
	\columnbreak
		\begin{equation*}
			\begin{aligned}
		\beta_{X^2Y^2}	= \beta_{X^2}+\beta_{X^4},\\
		\beta_{XY^3} 	= \beta_{XY}+\beta_{X^3Y},\\
		\beta_{Y^4} 	= \beta_{Y^2}+\beta_{X^2Y^2}.
			\end{aligned}
		\end{equation*}
	\end{multicols}
	\noindent Plugging in the expressions for $\beta_{Y^2}$ and $\beta_{X^2Y^2}$ in the expression for $	
	\beta_{Y^4}$
	gives the form (\ref{bc3-r6}) of $\mc{M}_2$.
\end{proof}

The candidate for $B_3$ in a moment matrix $\mc{M}_3$ generated by the measure
for $\mc{M}_2$ is given by the following.

\begin{proposition}\label{bc33-r6-B(3)}
	Let $\beta\equiv \beta^{(4)}$ be a nc sequence with a moment matrix $\mc{M}_2$ of rank 6 satisfying the relation 
	$\bY^2=\mds{1}+\bX^2$.
	Suppose $\beta$ admits a nc measure $\mu$. If
		$\mc{M}_3 =\begin{pmatrix} \mc{M}_2 & B_3 \\ B_3^t & C_3 \end{pmatrix}$
	is a moment matrix generated by the measure $\mu$, then $B_3$ satisfies
		\begin{multicols}{2}
	\begin{equation}\label{B3-bc3}
		\begin{aligned}
			\beta_{X^2Y^3}=\beta_{XY^2XY} = \beta_{X^2Y}+q,\\
			\beta_{X^3Y^2}=\beta_{X^2YXY} = \beta_{X^3}+p,\\
			\beta_{XY^4} = \beta_X+2\beta_{X^3}+p,	
		\end{aligned}
	\end{equation}
\vfill
\columnbreak
	\begin{equation*}
		\begin{aligned}
			\beta_{Y^5} = \beta_{Y}+2\beta_{X^2Y}+q,\\
			\beta_{X^5}=p,\\
			\beta_{X^4Y}=q,
		\end{aligned}
	\end{equation*}
\end{multicols}
	where $p,q\in \RR$ are parameters.
\end{proposition}

\begin{proof}
	The RG relations which must hold in $\mc{M}_3$ are
	\begin{multicols}{2}
	\begin{equation*}\label{eq-bc3-block-B(3)}
	\begin{aligned}
		\bY^3	&=& \bY+\bX^2\bY,\\
		\bY^3	&=& \bY+\bY\bX^2,
	\end{aligned}
	\end{equation*}
	\columnbreak
	\begin{eqnarray*}
		\bX\bY^2  &=& \bX+\bX^3,\\
		\bY^2\bX 	&=& \bX+\bX^3.
	\end{eqnarray*}
	\end{multicols}
	From these relations we get the following system:
	\begin{multicols}{2}
	\begin{equation*}\label{y2=1+x2 rg-system-b}
	\begin{aligned}
		\beta_{XY^4} &=& (\beta_X+\beta_{X^3})+\beta_{X^3Y^2},\\
		\beta_{XY^4} &=& (\beta_X+\beta_{X^3})+\beta_{X^2YXY},\\
		\beta_{Y^5} &=& (\beta_Y+\beta_{X^2Y})+\beta_{X^2Y^3},
	\end{aligned}
	\end{equation*}
	\columnbreak
	\begin{equation*}
	\begin{aligned}
		\beta_{X^2Y^3} &=& \beta_{X^2Y}+\beta_{X^4Y},\\
		\beta_{X^3Y^2} &=& \beta_{X^3}+\beta_{X^5},\\
		\beta_{XY^2XY} &=& \beta_{X^2Y}+\beta_{X^4Y}.
	\end{aligned}
	\end{equation*}
	\end{multicols}
	\noindent The solution of this system is given by the statement of the proposition.
\end{proof}

\begin{theorem}\label{flat-bc3-r6}
	Suppose $\beta\equiv \beta^{(4)}$ is a nc sequence with a moment matrix $\mc{M}_2$ of rank 6 satisfying the relation 
	$\bY^2=\mds{1}+\bX^2$. Let us define the moments of degree 5 by (\ref{B3-bc3})
	and $B_3$ as in Proposition \ref{B3,C3,structure}.
	Then the following are true:
	\begin{enumerate} 
		\item There exists a matrix $W\in \RR^{7\times 10}$ satisfying 
			$$B_3 = \mc{M}_2W.$$
	\item \label{point4-bc3-r6}
		We write $M_1=\{\mds 1,\bX,\bY,\bX^2, \bX\bY,\bY\bX\}$ and $M_2 =\{\bX^3, \bX^2 \bY, \bX\bY\bX, \bX\bY^2, \bY\bX^2, \bY\bX\bY, \bY^{2}\bX, \bY^3\}$. 
		Let $W_1\in \RR^{6\times 10}$ 
		be the matrix 	
			$$W_1=({\mc{M}_2}|_{M_1})^{-1} {B_3}|_{M_1,M_2}.$$
		If $\mc{M} =\begin{pmatrix} \mc{M}_2 & B_3 \\ B_3^t & C_3 \end{pmatrix}$
		is a flat extension of $\mc M_2$, then
		$C_3=(C_{ij})_{ij}$ is equal to
		$W_1^t {\mc{M}_2}|_{M_1} W_1$ and
		$\mc{M}$ has a moment structure if and only if
		\begin{multicols}{3}
		\begin{equation}\label{C3-sys-bc3}
			\begin{aligned}
				C_{47}= C_{66},\\
				C_{25}= C_{33},\\
				C_{12}= C_{13},
			\end{aligned}
		\end{equation}
\vfill
\columnbreak
		\begin{equation*}
			\begin{aligned}
				C_{16}= C_{23},\\
				C_{48}= C_{68},\\
				C_{14}= C_{22},	
			\end{aligned}
		\end{equation*}
\vfill
\columnbreak		
		\begin{equation*}
			\begin{aligned}
				C_{28}= C_{44},\\
				C_{26}= C_{27},
			\end{aligned}
		\end{equation*}
\end{multicols}
	\end{enumerate}
\end{theorem}

\begin{proof}
	The proof is analogous to the proof of Theorem \ref{flat-bc2-r6}.
	Using similar arguments we prove the following claim:\\
	
	\noindent \textbf{Claim:}
		$C_{18}=C_{24}$, $C_{16}-C_{23}=C_{46}-C_{38}$, 
		$C_{37}-C_{27}=C_{13}-C_{12}.$\\
		
	By Claim the system (\ref{hankel-system}) holds if and only if the system (\ref{C3-sys-bc3})
	holds.
\end{proof}

We present now a special case which highlights the difference between the classical commutative and the tracial bivariate quartic moment problems with a moment matrix $\mc M_2$ satisfying exactly $\bY^2=\mds{1}+\bX^2$.
Note that in the case of a cm sequence $\beta$, using a transformation $x\mapsto \frac{x-y}{2}$, $y\mapsto \frac{x+y}{2}$ gives us a cm sequence $\widetilde \beta$ with the moment matrix $\widetilde{\mc M_2}$ satisfying $\bX\bY=1.$
Therefore, by Theorem \ref{com-case} (\ref{pt3}), in the cm case $\mc M_2$ always admits a flat extension to the moment matrix $\mc M_3$.
The following example shows that this is not true in the nc case. However, a nc measure in this example still exists.

\begin{example}\label{ex-bc3-r6} 
	For $\beta_{X^4}> \frac{1}{4}$, the following matrices are psd moment matrices 
	of rank 6 satisfying the relation $\bY^2=\mds{1}+\bX^2$,
$$\mc{M}_2(\beta_{X^4})=
	\begin{mpmatrix}
 		1 & 0 & 0	 & \frac{1}{2} & 0 & 0 & \frac{3}{2} \\
 		0 & \frac{1}{2} & 0 & 0 & 0 & 0 & 0 \\
 		0 & 0 & \frac{3}{2} & 0 & 0 & 0 & 0 \\
		 \frac{1}{2} & 0 & 0 & \beta_{X^4} & 0 & 0 &
 		  \frac{1}{2}+\beta_{X^4} \\
 		0 & 0 & 0 & 0 & \frac{1}{2}+\beta_{X^4} & 0 & 0 \\
 		0 & 0 & 0 & 0 & 0 & \frac{1}{2}+\beta_{X^4} & 0 \\
 		\frac{3}{2} & 0 & 0 & \frac{1}{2}+\beta_{X^4} & 0 & 0 & 2+\beta_{X^4} \\
	\end{mpmatrix}.$$
\noindent
Let us define the moments
	of degree 5 by (\ref{B3-bc3}),
	$B_3$ as in Proposition \ref{B3,C3,structure} and $M_1, M_2, C_3$ as in Theorem \ref{flat-bc3-r6} (\ref{point4-bc3-r6}).
	It is easy to check that
	$$({\mc{M}_2}|_{M_1})^{-1}=
\begin{pmatrix} 
 \frac{4 \beta_{X^4}}{4\beta_{X^4}-1} & 0 & 0 & \frac{2}{1-4\beta_{X^4}}  \\
 0 & 2 & 0 & 0 &  \\
 0 & 0 & \frac{2}{3} & 0 &  \\
 \frac{2}{1-4\beta_{X^4}} & 0 & 0 & \frac{4}{4\beta_{X^4}-1} 
 \end{pmatrix}
 \bigoplus
 \begin{pmatrix} \frac{2}{1+2\beta_{X^4}} & 0\\ 0 &  \frac{2}{1+2\beta_{X^4}} \end{pmatrix}$$
	By a straightworward calculation of $\left({B_3}|_{M_1,M_2}\right)^t ({\mc{M}_2}|_{M_1})^{-1} \left({B_3}|_{M_1,M_2}\right)$
	we get that
	\begin{eqnarray*}
		C_{47}(p,q,\beta_{X^4}) &=&
			\frac{1}{2} + 2 \beta_{X^4}+ 2 \beta_{X^4}^2 + \frac{4 q^2}{1 + 2 \beta_{X^4}} + 
			\frac{4 p^2}{-1 + 4 \beta_{X^4}},\\
		C_{66}(p,q,\beta_{X^4}) &=&
			\frac{4q^2}{1+2\beta_{X^4}}+\frac{4p^2}{-1+4\beta_{X^4}},
		\end{eqnarray*}
and hence 
	$$(C_{47}-C_{66})(p,q,\beta_{X^4}) = \frac{1}{2}(1+2\beta_{X^4})^2\neq 0.$$
Therefore the  system from Theorem \ref{flat-bc3-r6} does not have a solution
and $\mc{M}_2(\beta_{X^4})$ does not admit a flat extension with a moment structure $\mc M_3(\beta_{X^4})$.
However, we will show that for every $\beta_{X^4}>\frac{1}{4}$, 
$\mc{M}_2(\beta_{X^4})$ admits a nc measure. 	We define the matrix function
		$$B(\alpha):=\mc{M}_2-\alpha \big(\mc{M}_2^{(0,1)}+\mc{M}_2^{(0,-1)}\big)=
		\begin{mpmatrix}
 		1-2\alpha & 0 & 0	 & \frac{1}{2} & 0 & 0 & \frac{3}{2}-2\alpha \\
 		0 & \frac{1}{2} & 0 & 0 & 0 & 0 & 0 \\
 		0 & 0 & \frac{3}{2}-2\alpha & 0 & 0 & 0 & 0 \\
		 \frac{1}{2} & 0 & 0 & \beta_{X^4} & 0 & 0 &
 		  \frac{1}{2}+\beta_{X^4} \\
 		0 & 0 & 0 & 0 & \frac{1}{2}+\beta_{X^4} & 0 & 0 \\
 		0 & 0 & 0 & 0 & 0 & \frac{1}{2}+\beta_{X^4} & 0 \\
 		\frac{3}{2}-2\alpha & 0 & 0 & \frac{1}{2}+\beta_{X^4} & 0 & 0 & 2+\beta_{X^4}-2\alpha \\
		\end{mpmatrix}
		.$$
	We have that
		\begin{equation}\label{r6-det-c31}
			\det(B(\alpha)|_{M_1})=\frac{1}{2}\det(B(\alpha)|_{\{\mds 1, \bX^2\}})\Big(\frac{3}{2}-2\alpha\Big)\Big(\frac{1}{2}+\beta_{X^4}\Big)^2,
		\end{equation}
	where
		\begin{equation}\label{r6-det-c32}
			\det(B(\alpha)|_{\{\mds 1, \bX^2\}})=\beta_{X^4}-\frac{1}{4}-2\alpha\beta_{X^4}.
		\end{equation}	
	Let $\alpha_0>0$ be the smallest positive number such that the rank of $B(\alpha_0)$ is smaller than 6.
	By (\ref{r6-det-c31}) and (\ref{r6-det-c32})
	we get that
		$$\alpha_0=\min\Big(\frac{3}{4}, \frac{-1+4\beta_{X^4}}{8\beta_{X^4}}\Big).$$
	For $\beta_{X^4}>\frac{1}{4}$ we have that
		$\alpha_0=\frac{-1+4\beta_{X^4}}{8\beta_{X^4}}.$
	It is easy to check that the kernel of $B(\alpha_0)$ satisfies the relations
		$$\bX^2=2\beta_{X^4} \mds 1, \quad \bY^2=(1+2\beta_{X^4})\mds 1.$$
	We also have
		$$\beta^{(B)}_{X}=\beta^{(B)}_Y=\beta^{(B)}_{X^3}=\beta^{(B)}_{X^2Y}=		
								\beta^{(B)}_{XY^2}=\beta^{(B)}_{Y^3}=0,$$
	where $\beta^{(B)}_{w(X,Y)}$ are the moments of $B(\alpha_0)$.
	This is a special case in the proof of Proposition \ref{structure-of-rank5-2}, i.e., Case 1. Following the proof
	we see that after using only transformations of type	
		$$(x,y)\mapsto (\alpha_1 x+ \beta_1 y, \alpha_2 x+ \beta_2 y)$$
	for some $\alpha_1,\alpha_2,\beta_1,\beta_2\in \RR$,
	we come into the basic case 1 of rank 5 with
		$\widetilde\beta_X=\widetilde\beta_Y=\widetilde\beta_{X^3}=0$. 
	But every such sequence admits a measure of type (2,1)
	by Theorem \ref{M(2)-XY+YX=0-bc1}.
	Hence $\beta$ admits a measure of type $(4,1)$. 
\end{example}

\subsection{Relation $\bY^2=\mds{1}$.}

The form of $\mc{M}_2$ is given by the following proposition.

\begin{proposition} \label{r6-rel-bc1}
	Let $\beta\equiv \beta^{(4)}$ be a nc sequence with a moment matrix $\mc{M}_2$ satisfying the relation
		\begin{equation} \label{r6-rel-bc1-eq}
			\bY^2=\mds{1}.
		\end{equation}
	Then $\mc{M}_2$ is of the form
\begin{equation}\label{bc1-r6}
	\mc{M}_2=
		\begin{mpmatrix}
 	\beta_{1} & \beta_{X} & \beta_{Y} & \beta_{X^2} & \beta_{XY} & \beta_{XY} & \beta_{1} \\
 	\beta_{X} & \beta_{X^2} & \beta_{XY} & \beta_{X^3} & \beta_{X^2Y} & \beta_{X^2Y} & \beta_{X} \\
	 \beta_{Y} & \beta_{XY} & \beta_{1} & \beta_{X^2Y} & \beta_{X} & \beta_{x} & \beta_{Y} \\
	 \beta_{X^2} & \beta_{X^3} & \beta_{X^2Y} & \beta_{X^4} & \beta_{X^3Y} & \beta_{X^3Y} & \beta_{X^2} \\
	 \beta_{XY} & \beta_{X^2Y} & \beta_{X} & \beta_{X^3Y} & \beta_{X^2} & \beta_{XYXY} & 	\beta_{XY} \\
	 \beta_{XY} & \beta_{X^2Y} & \beta_{X} & \beta_{X^3Y} & \beta_{XYXY} & \beta_{X^2} & 	\beta_{XY} \\
	 \beta_{1} & \beta_{X} & \beta_{Y} & \beta_{X^2} & \beta_{XY} & \beta_{xy} & \beta_{1} 
		\end{mpmatrix}.
\end{equation}
\end{proposition}

\begin{proof}
	The relation (\ref{r6-rel-bc1-eq}) gives us the following system in $\mc{M}_2$
		\begin{multicols}{2}
		\begin{equation}\label{eq-bc1-r6}
			\begin{aligned}
		\beta_{Y^2}	= \beta_{1},\\
		\beta_{XY^2}	= \beta_{X},\\
		\beta_{Y^3}	= \beta_{Y},
			\end{aligned}
		\end{equation}
	\vfill
	\columnbreak
		\begin{equation*}
			\begin{aligned}
		\beta_{X^2Y^2}	= \beta_{X^2},\\
		\beta_{XY^3} 	= \beta_{XY},\\
		\beta_{Y^4} 	= \beta_{Y^2}.
			\end{aligned}
		\end{equation*}
	\end{multicols}
	\noindent This gives the form (\ref{bc1-r6}) of $\mc{M}_2$.
\end{proof}
	
The candidate for $B_3$ in a moment matrix $\mc{M}_3$ generated by the measure
for $\mc{M}_2$ is given by the following.

\begin{proposition}\label{bc2-r6-B(3)}
	Let $\beta\equiv \beta^{(4)}$ be a nc sequence with a moment matrix $\mc{M}_2$ of rank 6 satisfying the relation 
	$\bY^2=\mds{1}$.
	Suppose $\beta$ admits a nc measure $\mu$. If
		$\mc{M}_3 =\begin{pmatrix} \mc{M}_2 & B_3 \\ B_3^t & C_3 \end{pmatrix}$
	is a moment matrix generated by the measure $\mu$, then 
	$B(3)$ is of the form
	\begin{equation} \label{bc1-r6-B(3)-form}
		B(3)=
		\begin{mpmatrix}
		\beta_{X^3} & \beta_{X^2Y} & \beta_{X^2Y} & \beta_{X} & \beta_{X^2Y} & \beta_{X} & 	
				\beta_{X} & \beta_{Y}\\
		\beta_{X^4} & \beta_{X^3Y} & \beta_{X^3Y} & \beta_{X^2} & \beta_{X^4Y} & 		
			\beta_{XYXY} & \beta_{X^2} & \beta_{XY}\\
		\beta_{X^3Y} & \beta_{X^2} & \beta_{XYXY} & \beta_{XY} & \beta_{X^2} & \beta_{XY}
		 			& \beta_{XY} & \beta_{1}\\
		p & q & q & \beta_{X^3} & q & r & \beta_{X^3} & \beta_{X^2Y}\\
		q & \beta_{X^3} & r & \beta_{X^2Y} & r & \beta_{X^2Y} & \beta_{X^2Y} & \beta_{X}\\
		q & r & r & \beta_{X^2Y} & \beta_{X^3} & \beta_{X^2Y} & \beta_{X^2Y} & \beta_{X}\\
		\beta_{X^3} & \beta_{X^2Y} & \beta_{X^2Y} & \beta_{X} & \beta_{X^2Y} & \beta_{X} & 
					\beta_{X} & \beta_{Y}
		\end{mpmatrix},
	\end{equation}
	where $p,q,r\in \RR$ are parameters.
\end{proposition}

\begin{proof}
	The RG relations which must hold in $\mc{M}_3$ are
	\begin{equation*}\label{rel-bc1-r6}
		\bY^3 = \bY, \quad \bX\bY^2  = \bX, \quad \bY^2\bX= \bX.
	\end{equation*}
	From these relations we get the following system:
	\begin{multicols}{3}
	\begin{equation}\label{y2-1-rg-system-b}
		\begin{aligned}
		\beta_{X^2Y^3} &=& \beta_{X^2Y},\\
		\beta_{XY^4} &=& \beta_{XY},
		\end{aligned}
	\end{equation}
\vfill
\columnbreak
	\begin{eqnarray*}
		\beta_{Y^5} &=& \beta_{1}, \nonumber\\
		\beta_{X^3Y^2} &=& \beta_{X^3},
	\end{eqnarray*}
\vfill
\columnbreak
	\begin{eqnarray*}
		\beta_{XY^2XY} &=& \beta_{X^2Y}, \nonumber\\
		\beta_{XY^4} &=& \beta_{X}.
	\end{eqnarray*}
\end{multicols}
	\noindent	Now the solution of the system (\ref{y2-1-rg-system-b}) is given by the statement of the proposition.
\end{proof}

\begin{theorem}\label{flat-bc4-r6}
	Suppose $\beta\equiv \beta^{(4)}$ is a nc sequence with a moment matrix $\mc{M}_2$ of rank 6 satisfying the relation 
	$\bY^2=\mds{1}$ and let $B_3$ be as in formula (\ref{bc1-r6-B(3)-form}). 
	Then the following are true:
	\begin{enumerate} 
		\item There exists a matrix $W\in \RR^{7\times 10}$ satisfying 
			$$B_3 = \mc{M}_2W.$$
		\item \label{point4-bc4-r6-33}
		We write $M_1 =\{\mds 1,\bX,\bY,\bX^2, \bX\bY,\bY^2\}$
		and $M_2 =\{\bX^3, \bX^2 \bY, \bX\bY\bX, \bX\bY^2, \bY\bX^2, \bY\bX\bY, \bY^{2}\}$.
		Let $W_1\in \RR^{6\times 10}$
		be the matrix 	
			$$W_1=({\mc{M}_2}|_{M_1})^{-1} {B_3}|_{M_1,M_2}.$$
		If $\mc{M} =\begin{pmatrix} \mc{M}_2 & B_3 \\ B_3^t & C_3 \end{pmatrix}$
		is a flat extension of $\mc M_2$, then
		$C_3=(C_{ij})_{ij}$ is equal to
		$W_1^t {\mc{M}_2}|_{M_1} W_1$ and
		$\mc{M}$ has a moment structure if and only if
	\begin{multicols}{2}
		\begin{equation}\label{C3-bc4-r6}
			\begin{aligned}
				C_{66}= \beta_{X^2}.\\
				C_{25}=C_{33}, \\
				C_{12}= C_{13},
			\end{aligned}
		\end{equation}
	\vfill
	\columnbreak
		\begin{equation*}
			\begin{aligned}
				C_{16}= C_{23},\\
				C_{22}= \beta_{X^4},	 \\
				C_{26}= \beta_{XYXY}=\beta_{X^3Y}.
			\end{aligned}
		\end{equation*}
\end{multicols}
	\end{enumerate}
\end{theorem}

\begin{proof}
			The proof is analogous to the proof of Theorem \ref{flat-bc2-r6}. Using similar arguments we prove the following claim:\\
	
	\noindent\textbf{Claim:}
			$C_{47}=\beta_{X^2}$,
			$ C_{38}=C_{46},$ 
			$C_{48}=C_{68},$
			$C_{14}= \beta_{X^4}$,
			$C_{28}= \beta_{44},	$ 
			$C_{27}= \beta_{X^3Y},$
			$C_{37}= \beta_{XYXY}.$\\
			
	By Claim the system (\ref{hankel-system}) holds if and only if the system
		(\ref{C3-bc4-r6}) holds.
\end{proof}

We present now a special case which highlights the difference between the classical commutative and the tracial bivariate quartic moment problems with a moment matrix $\mc M_2$ satisfying $\bY^2=\mds{1}$ . 
By Theorem \ref{com-case} (\ref{pt3}), in the cm case $\mc M_2$ always admits a flat extension to the moment matrix $\mc M_3$.
The following example shows that this is not true in the nc case. However, a nc measure in this example still exists.

\begin{example}
For $\beta_{X^4}>1$,
the following matrices are psd moment matrices of rank 6 satisfying $\bY^2=\mds 1$, 
	$$\mc{M}_2(\beta_{X^4})=
		\begin{mpmatrix}
 			1 & 0 & 0 & 1 & 0 & 0 & 1 \\
 			0 & 1 & 0 & 0 & 0 & 0 & 0 \\
			 0 & 0 & 1 & 0 & 0 & 0 & 0 \\
 			1 & 0 & 0 & \beta_{X^4} & 0 & 0 & 1 \\
 			0 & 0 & 0 & 0 & 1 & 0 & 0 \\
 			0 & 0 & 0 & 0 & 0 & 1 & 0 \\
			 1 & 0 & 0 & 1 & 0 & 0 & 1
		\end{mpmatrix}.$$
Let us define 
	$B_3$ as in Proposition \ref{B3,C3,structure} and $M_1, M_2, C_3$ as in Theorem \ref{flat-bc4-r6} (\ref{point4-bc4-r6-33}).
	It is easy to check that
	$$({\mc{M}_2}|_{M_1})^{-1}=
\begin{pmatrix} 
 \frac{\beta_{X^4}}{\beta_{X^4}-1} & 0 & 0 & \frac{1}{1-\beta_{X^4}}  \\
 0 & 1 & 0 & 0 &  \\
 0 & 0 & 1 & 0 &  \\
 \frac{1}{1-\beta_{X^4}} & 0 & 0 & \frac{1}{\beta_{X^4}-1} 
 \end{pmatrix}
 \bigoplus
 \begin{pmatrix} 1& 0\\ 0 &  1 \end{pmatrix}.$$
	By a straightworward calculation of $\left({B_3}|_{M_1,M_2}\right)^t ({\mc{M}_2}|_{M_1})^{-1} \left({B_3}|_{M_1,M_2}\right)$
	we get that
$$
\begin{mpmatrix}
 \frac{p^2}{\beta_{X^4}-1}+2 q^2+\beta_{X^4}^2 & r q+\frac{p q}{\beta_{X^4}-1} & 2 r q+\frac{p q}{\beta_{X^4}-1} & \beta_{X^4} & r q+\frac{p
   q}{\beta_{X^4}-1} & \frac{p r}{\beta_{X^4}-1} & \beta_{X^4} & 0 \\
 r q+\frac{p q}{\beta_{X^4}-1} & \frac{q^2}{\beta_{X^4}-1}+r^2+1 & \frac{q^2}{\beta_{X^4}-1}+r^2 & 0 & \frac{q^2}{\beta_{X^4}-1}+1 & \frac{q
   r}{\beta_{X^4}-1} & 0 & 1 \\
 2 r q+\frac{p q}{\beta_{X^4}-1} & \frac{q^2}{\beta_{X^4}-1}+r^2 & \frac{q^2}{\beta_{X^4}-1}+2 r^2 & 0 & \frac{q^2}{\beta_{X^4}-1}+r^2 & \frac{q
   r}{\beta_{X^4}-1} & 0 & 0 \\
 \beta_{X^4} & 0 & 0 & 1 & 0 & 0 & 1 & 0 \\
 r q+\frac{p q}{\beta_{X^4}-1} & \frac{q^2}{\beta_{X^4}-1}+1 & \frac{q^2}{\beta_{X^4}-1}+r^2 & 0 & \frac{q^2}{\beta_{X^4}-1}+r^2+1 & \frac{q
   r}{\beta_{X^4}-1} & 0 & 1 \\
 \frac{p r}{\beta_{X^4}-1} & \frac{q r}{\beta_{X^4}-1} & \frac{q r}{\beta_{X^4}-1} & 0 & \frac{q r}{\beta_{X^4}-1} & \frac{r^2}{\beta_{X^4}-1} & 0
   & 0 \\
 \beta_{X^4} & 0 & 0 & 1 & 0 & 0 & 1 & 0 \\
 0 & 1 & 0 & 0 & 1 & 0 & 0 & 1 \\
\end{mpmatrix}.$$
For the flat extension we must have, by Theorem \ref{flat-bc4-r6},
	\begin{multicols}{2}
	\begin{align*}
		\frac{r^2}{\beta_{X^4}-1} = 1,\\
		 \frac{q^2}{\beta_{X^4}-1}+1 =  \frac{q^2}{\beta_{X^4}-1}+2 r^2,\\
		 r q+\frac{p q}{\beta_{X^4}-1} = 2 r q+\frac{p q}{\beta_{X^4}-1},
	\end{align*}
	\vfill
	\columnbreak
		\begin{align*}
		\frac{p r}{\beta_{X^4}-1} =  \frac{q^2}{\beta_{X^4}-1}+r^2,\\
		 \frac{q^2}{\beta_{X^4}-1}+r^2+1 = \beta_{X^4},\\
		 \frac{qr}{\beta_{X^4}-1} = 0.
	\end{align*}
	\end{multicols}
\noindent It is easy to see these equations are satisfied if and only if 
	$$\beta_{X^4}=\frac{3}{2}\quad \text{in which case}\quad 
		p=\pm \frac{1}{2\sqrt{2}},\quad q=0,\quad r= \pm \frac{1}{\sqrt{2}}.$$
However, we will prove that for every $\beta_{X^4}>1$, $\mc{M}_2(\beta_{X^4})$ admits a nc measure. We define the matrix function
	$$B(\alpha):=\mc{M}_2(\beta_{X^4})-\alpha A,$$
where 
	$$A=\begin{mpmatrix}
		 1 & 0 & 0 & 1 & 0 & 0 & 1 \\
 		0 & 1 & 0 & 0 & 0 & 0 & 0 \\
		 0 & 0 & 1 & 0 & 0 & 0 & 0 \\
		 1 & 0 & 0 & 1 & 0 & 0 & 1 \\
 		0 & 0 & 0 & 0 & 1 & -1 & 0 \\
		 0 & 0 & 0 & 0 & -1 & 1 & 0 \\
		 1 & 0 & 0 & 1 & 0 & 0 & 1 \\
		\end{mpmatrix}.$$
$A$ is a psd moment matrix of rank 4 satisfying the relations 
	$$\bX^2=\mds 1,\quad \bX\bY+\bY\bX=\mbf 0,\quad \bY^2=\mds 1$$
and thus admits a nc measure by Theorem \ref{rank4-soln-aljaz} (\ref{point-3-rank4}).
But then
$$B\Big(\frac{1}{2}\Big)=
	\begin{mpmatrix}
 		\frac{1}{2} & 0 & 0 & \frac{1}{2} & 0 & 0 & \frac{1}{2} \\
 		0 & \frac{1}{2} & 0 & 0 & 0 & 0 & 0 \\
 		0 & 0 & \frac{1}{2} & 0 & 0 & 0 & 0 \\
 		\frac{1}{2} & 0 & 0 & \beta_{X^4}-\frac{1}{2} & 0 & 0 & \frac{1}{2} \\
 		0 & 0 & 0 & 0 & \frac{1}{2} & \frac{1}{2} & 0 \\
		 0 & 0 & 0 & 0 & \frac{1}{2} & \frac{1}{2} & 0 \\
		 \frac{1}{2} & 0 & 0 & \frac{1}{2} & 0 & 0 & \frac{1}{2} \\
	\end{mpmatrix}$$
is a psd cm moment matrix of rank 5 satisfying $\bY^2=\mds 1$ and $\bX\bY=\bY\bX$
and hence admits a 5-atomic measure with the atoms of the from $(x_j,y_j)\in \RR^2$, $j=1,\ldots,5$, by Theorem \ref{com-case}.
\end{example}

\end{document}